\documentclass[10pt]{amsart}
\usepackage{amsmath}
\usepackage{amssymb}
\usepackage{amsfonts}
\usepackage{pxfonts}
\usepackage[OT2,T1]{fontenc}
\usepackage[
textwidth=14.5cm, 
textheight=21cm,
hmarginratio=1:1,
vmarginratio=1:1]{geometry}
\usepackage{pifont}
\DeclareSymbolFont{cyrletters}{OT2}{wncyr}{m}{n}
\DeclareMathSymbol{\Lfun}{\beta}{cyrletters}{"62}
\DeclareMathSymbol{\Rfun}{\beta}{cyrletters}{"76}

\newtheorem{thm}{Theorem}[section]
\newtheorem{cor}[thm]{Corollary}
\newtheorem{lem}[thm]{Lemma}
\newtheorem{prop}[thm]{Proposition}
\theoremstyle{definition}

\theoremstyle{remark}
\newtheorem{rem}[thm]{Remark}

\numberwithin{equation}{subsection}
\numberwithin{figure}{subsection}

\newcommand{\diff}{\mathrm{d}}
\newcommand{\C}{{\mathbb C}}
\newcommand{\R}{{\mathbb R}}
\newcommand{\D}{{\mathbb D}}
\newcommand{\Te}{{\mathbb T}}

\newcommand{\imag}{\mathrm{i}}
\newcommand{\e}{\mathrm{e}}

\newcommand{\Pop}{{\mathbf P}}

\newcommand{\Gop}{{\mathbf G}}
\newcommand{\Mop}{{\mathbf M}}

\newcommand{\Iop}{\mathbf{I}}

\newcommand{\dbar}{\bar\partial}

\newcommand{\dd}{{\pmb{\partial}}}

\newcommand\bslashnabla{\boldsymbol{\nabla}}

\newcommand{\calR}{{\mathfrak R}}

\newcommand{\Sop}{{\mathbf S}}
\newcommand{\Hop}{{\mathbf H}}
\newcommand{\Nop}{{\mathbf N}}
\newcommand{\Aop}{{\mathbf A}}
\newcommand{\Bop}{{\mathbf B}}

\newcommand{\hDelta}{\varDelta}
\newcommand{\Ordo}{\mathrm{O}}

\newcommand{\vbeta}{\lambda}

\DeclareMathOperator{\re}{Re}

\begin{document}

%
\title
{Asymptotic expansion of polyanalytic Bergman kernels}

\author{Antti Haimi}

\address{Haimi: Department of Mathematics\\
The Royal Institute of Technology\\
S -- 100 44 Stockholm\\
SWEDEN}

\email{anttih@kth.se}

\author{Haakan Hedenmalm}

\address{Hedenmalm: Department of Mathematics\\
The Royal Institute of Technology\\
S -- 100 44 Stockholm\\
SWEDEN}

\email{haakanh@math.kth.se}

\date{14 January 2013}


\thanks{The second author was supported by the G\"oran Gustafsson
Foundation (KVA) and by Vetenskapsr\r{a}det (VR)}

\begin{abstract} 

We consider the $q$-analytic functions on a given planar domain $\Omega$,
square integrable with respect to a weight. This gives us a
$q$-analytic Bergman kernel, which we use to extend the Bergman metric to 
this context. We recall that $f$ is $q$-analytic if $\bar \partial^q f = 0$ for 
the given positive integer $q$.

We also obtain asymptotic formulae for the $q$-analytic 
Bergman kernel in the setting of degenerating power weights $\e^{-2mQ}$,
as the positive real parameter $m$ tends to infinity. This is only known 
for $q=1$ in view of the work of Tian, Yau, Zelditch, and Catlin. We remark 
here that since a $q$-analytic function may be identified with a vector-valued 
holomorphic function, the Bergman space of $q$-analytic functions may be 
understood as a vector-valued holomorphic Bergman space supplied with a 
certain singular local metric on the vectors.
Finally, we apply the obtained asymptotics for $q=2$ to the bianalytic Bergman
metrics, and after suitable blow-up, the result is independent of $Q$ for a wide 
class of potentials $Q$. We interpret this as an instance of geometric universality.  
\end{abstract}

\maketitle


\section{Overview}

In Section \ref{WPABSK}, we define, in the one-variable context, the weighted 
Bergman spaces and their polyanalytic extensions, and in Section 
\ref{musing-BM}, we consider the various possible ramifications for Bergman 
metrics.
It should be remarked that the polyanalytic Bergman spaces can be understood 
as vector-valued (analytic) Bergman spaces with singular local inner product matrix. 

Generally speaking, Bergman kernels are difficult to obtain in explicit form. 
However, it is sometimes possible to obtain an asymptotic expansion for them, for 
instance as the weight degenerates in a power fashion. In Sections 
\ref{sec-asexp}--\ref{sec-locglob}, we extend the asymptotic expansion to the 
polyanalytic context. Our analysis is based on the microlocal PDE approach of 
Berman, Berndtsson, Sj\"ostrand \cite{bbs}. We focus mainly on the biholomorphic 
(bianalytic) case, and obtain the explicit form of the first few terms of the 
expansion. 
In Section \ref{sec-interlude}, we estimate the norm of point evaluations on 
bianalytic Bergman spaces, which is later needed to estimate the bianalytic 
Bergman kernel along the diagonal. 

Finally,  in Section \ref{sec-Bergmanmetric}, we apply the obtained asymptotics 
for $q=2$ to the bianalytic Bergman metrics introduced in Section \ref{musing-BM}, 
and after suitable blow-up, the resulting metrics turn out to be independent of the 
given potential (which defines the power weight). We interpret this as an instance of 
geometric universality.

\section{Weighted polyanalytic Bergman spaces and kernels} 
\label{WPABSK}


\subsection{Basic notation} We let $\C$ denote the complex plane and
$\R$ the real line. 
For $z_0\in\C$ and positive real $r$, let $\D(z_0,r)$ be the open 
disk centered at $z_0$ with radius $r$; moreover, we let $\Te(z_0,r)$ be
the boundary of $\D(z_0,r)$ (which is a circle). When $z_0=0$ and $r=1$,
we simplify the notation to $\D:=\D(0,1)$ and $\Te:=\Te(0,1)$. 
We let 
\[
\hDelta:=\frac14\bigg(\frac{\partial^2}{\partial x^2}+
\frac{\partial^2}{\partial y^2}\bigg),
\qquad\diff A(z):=\diff x\diff y,
\]
denote the normalized Laplacian and the area element, respectively. 
Here, $z=x+\imag y$ is the standard decomposition into real and 
imaginary parts. 
The complex differentiation operators 
\[
\partial_z:=\frac{1}{2}\bigg(\frac{\partial}{\partial x}-\imag 
\frac{\partial}{\partial y}\bigg)
,\qquad \bar\partial_z:=\frac{1}{2}\bigg(\frac{\partial}{\partial x}+\imag 
\frac{\partial}{\partial y}\bigg),
\]
will be useful. It is well-known that $\hDelta=\partial_z\bar\partial_z$. 

\subsection{The weighted Bergman spaces and kernels}
\label{subsec-wB}

Let $\Omega$ be a domain in $\C$, and let $\omega:\Omega\to\R_+$ be a 
continuous function ($\omega$ is frequently called a \emph{weight}). 
Here, we write $\R_+:=]0,+\infty[$ for the positive half-axis. The
space $L^2(\Omega,\omega)$ is the weighted $L^2$-space on $\Omega$ with finite
norm
\begin{equation}
\|f\|_{L^2(\Omega,\omega)}^2:=\int_\Omega |f(z)|^2\omega(z)\diff A(z),
\label{eq-norm0}
\end{equation}
and associated sesquilinear inner product
\begin{equation}
\langle f,g\rangle_{L^2(\Omega,\omega)}:=\int_\Omega f(z)\overline{g(z)}
\omega(z)\diff A(z).
\label{eq-innprod1}
\end{equation}
The corresponding {\em weighted Bergman space} $\mathrm{A}^2(\Omega,\omega)$ 
is the linear subspace of $L^2(\Omega,\omega)$ consisting of functions 
holomorphic in $\Omega$, supplied with the inner product structure of 
$L^2(\Omega,\omega)$. 
Given the assumptions on the weight $\omega$, it is easy to check that point 
evaluations are locally uniformly bounded on $\mathrm{A}^2(\Omega,\omega)$, 
and, therefore, $\mathrm{A}^2(\Omega,\omega)$ is a norm-closed subspace of 
$L^2(\Omega,\omega)$. As $L^2(\Omega,\omega)$ is separable, so is 
$\mathrm{A}^2(\Omega,\omega)$, and we may find a countable orthonormal basis 
$\phi_1,\phi_2,\phi_3,\ldots$ in $\mathrm{A}^2(\Omega,\omega)$. 
We then form the 
function $K=K_{\Omega,\omega}$ -- called the \emph{weighted Bergman kernel} --
given by
\begin{equation}
K(z,w):=\sum_{j=1}^{+\infty}\phi_j(z)\overline{\phi_j(w)},\qquad
(z,w)\in\Omega\times\Omega,
\label{eq-bk1}
\end{equation}
and observe that for fixed $w\in\Omega$, the function 
$K(\cdot,w)\in \mathrm{A}^2(\Omega,\omega)$ has the reproducing property  
\begin{equation}
f(w)=\langle f,K(\cdot,w)\rangle_{L^2(\Omega,\omega)},\qquad w\in\Omega.
\label{eq-reprprop1}
\end{equation}
Here, it is assumed that $f\in \mathrm{A}^2(\Omega,\omega)$.
In fact, the weighted Bergman kernel $K$ is uniquely determined by these
two properties, which means that $K$ -- initially defined by \eqref{eq-bk1}
in terms of an orthonormal basis -- actually is independent of the choice 
of basis.
Note that above, we implicitly assumed that $\mathrm{A}^2(\Omega,\omega)$  is 
infinite-dimensional, which need not generally be the case. If it is 
finite-dimensional, the corresponding sums would range over a finite set of 
indices $j$ instead.  

\subsection{The weighted polyanalytic Bergman spaces and kernels}
Given an integer $q=1,2,3,\ldots$, a continuous function $f:\Omega\to\C$ 
is said to be \emph{$q$-analytic (or $q$-holomorphic) in $\Omega$} if it 
solves the partial differential equation
\begin{equation*}
\bar\partial_z^q f(z)=0,\qquad z\in\Omega,
\end{equation*}
in the sense of distribution theory. So the $1$-analytic functions are just 
the ordinary holomorphic functions. A function $f$ is said to be 
\emph{polyanalytic} if it is $q$-analytic for some $q$; then the number $q-1$
is said to be the \emph{polyanalytic degree} of $f$. 
By solving the $\dbar$-equation repeatedly, it is easy to see that $f$ is 
$q$-holomorphic if and only if it can be expressed in the form 
\begin{equation}
f(z)=f_1(z)+\bar zf_2(z)+\cdots+\bar z^{q-1}f_q(z) ,
\label{eq-Alm-1}
\end{equation}
where each $f_j$ is holomorphic in $\Omega$, for $j=1,\ldots,q$. 
So the dependence on $\bar z$ is polynomial of degree at most $q-1$. We observe
quickly that the vector-valued holomorphic function
\begin{equation*}
\mathbf{V}[f](z):=(f_1(z),f_2(z),\ldots,f_q(z)),
\end{equation*} 
is in a one-to-one relation with the $q$-analytic function $f$. 
We will think of $\mathbf{V}[f](z)$ as a \emph{column vector}.
In a way, this means that we may think of a polyanalytic function as a
vector-valued holomorphic function supplied with the additional structure of
scalar point evaluations $\C^q\to\C$ given by
\[
(f_1(z),f_2(z),\ldots,f_q(z)) \mapsto 
f_1(z)+\bar zf_2(z)+\cdots+\bar z^{q-1}f_q(z).
\]
We associate to $f$ not only the vector-valued holomorphic function 
$\mathbf{V}[f]$ but also the function of two complex variables
\begin{equation}
\mathbf{E}[f](z,z')=f_1(z)+\bar z' f_2(z)+\cdots+(\bar z')^{q-1}f_q(z),
\label{eq-Alm-2}
\end{equation}
which we call the \emph{extension} of $f$. The function $\mathbf{E}[f](z,z')$
is holomorphic in $(z,\bar z')\in\Omega\times\C$, with polynomial dependence on 
$\bar z'$. To recover the function $f$, we just restrict to the diagonal:
\begin{equation}
\mathbf{E}[f](z,z)=f(z),\qquad z\in\Omega.
\label{eq-Alm-3}
\end{equation}
 We observe here for the moment that 
\begin{equation}
\vert f(z)\vert^2=\mathbf{V}[f](z)^\ast\mathbf{A}(z) \mathbf{V}[f](z),
\label{expsq1}
\end{equation}
where the asterisk indicates the adjoint, and $\mathbf{A}(z)$ is the singular 
$q\times q$ matrix
\begin{equation}
\mathbf{A}(z)=
\begin{pmatrix}
1
&\cdots&
{\bar z}^{q-1}
\\
 \vdots & \ddots & \vdots  \\
z^{q-1}
&\cdots&
z^{q-1}\bar z^{q-1}
\end{pmatrix}
.
\label{eq-singmat1}
\end{equation}
For some general background material on polyanalytic functions, we refer to 
the book \cite{balk}. 

As before, we let the weight $\omega:\Omega\to\R_+$ be continuous, and
define $\mathrm{PA}^2_{q}(\Omega,\omega)$ to be the linear subspace of 
$L^2(\Omega,\omega)$ consisting of $q$-analytic functions in $\Omega$. Then
$\mathrm{PA}^2_{1}(\Omega,\omega)=\mathrm{A}^2(\Omega,\omega)$, 
the usual weighted Bergman we encountered in Subsection \ref{subsec-wB}.
For general $q=1,2,3,\ldots$, it is not difficult to show that point 
evaluations are locally uniformly bounded on 
$\mathrm{PA}^2_{q}(\Omega,\omega)$, and 
therefore, $\mathrm{PA}^2_{q}(\Omega,\omega)$ is a norm-closed subspace of 
$L^2(\Omega,\omega)$. We will refer to $\mathrm{PA}^2_{q}(\Omega,\omega)$ as
a \emph{weighted $q$-analytic Bergman space}, or as a 
\emph{weighted polyanalytic Bergman space of degree $q-1$}. In view of 
\eqref{expsq1} and \eqref{eq-singmat1}, we may view  
$\mathrm{PA}^2_{q}(\Omega,\omega)$ as a space of vector-valued holomorphic
functions on $\Omega$, supplied with the singular matrix-valued weight 
$\omega(z)\mathbf{A}(z)$. 

If we let 
$\phi_1,\phi_2,\phi_3,\ldots$ be an orthonormal basis for 
$\mathrm{PA}^2_{q}(\Omega,\omega)$, we can form the 
\emph{weighted polyanalytic Bergman kernel} $K_q=K_{q,\Omega,\omega}$ given by
\begin{equation}
K_q(z,w):=\sum_{j=1}^{+\infty}\phi_j(z)\overline{\phi_j(w)},\qquad
(z,w)\in\Omega\times\Omega.
\label{eq-bk2}
\end{equation}
As was the case with the weighted Bergman kernel, $K_q$ is independent of
the choice of basis $\phi_j$, $j=1,2,3,\ldots$, and has the reproducing
property
\begin{equation}
f(w)=\langle f,K_q(\cdot,w)\rangle_{L^2(\Omega,\omega)},\qquad w\in\Omega,
\label{eq-reprprop2}
\end{equation}
for all $f\in\mathrm{PA}_q^2(\Omega,\omega)$. Alongside with the kernel $K_q$,
we should also be interested in its \emph{lift}
\begin{equation}
\mathbf{E}_{\otimes2}[K_q](z,z';w,w'):=
\sum_{j=1}^{+\infty}\mathbf{E}[\phi_j](z,z')\overline{\mathbf{E}[\phi_j](w,w')},
\qquad (z,z',w,w')\in\Omega\times\C\times\Omega\times\C.
\label{eq-bk3}
\end{equation}
The lifted kernel $\mathbf{E}_{\otimes2}[K_q]$ is also independent of the choice 
of basis, just like the kernel $K_q$ itself. 
If  $\mathrm{PA}^2_{q}(\Omega,\omega)$ would happen to be finite-dimensional, 
the above sums defining kernels should be replaced by sums ranging over a 
finite set of indices $j$.

\subsection{The polyanalytic Bergman space in the model case of 
the unit disk with constant weight}
For $q=1,2,3,\ldots$, we consider the spaces 
$\mathrm{PA}^2_q(\D):=\mathrm{PA}^2_q(\D,\pi^{-1})$ where the domain is the unit 
disk $\D$ and the weight is $\omega(z)\equiv1/\pi$. The corresponding 
reproducing kernel $K_q$ was obtained Koshelev \cite{koselev}:
\begin{equation}
K_q(z,w)=q\sum_{j=0}^{q-1}(-1)^j\binom{q}{j+1}\binom{q+j}{q}
\frac{(1-w\bar z)^{q-j-1}|z-w|^{2j}}{(1-z\bar w)^{q+j+1}}.
\label{eq-mcdisk1}
\end{equation}
Its diagonal restriction is given by
\begin{equation}
K_q(z,z)=\frac{q^2}{(1-|z|^2)^{2}}.
\end{equation}
Based on \eqref{eq-mcdisk1}, the lift is $K_q$ is then easily calculated,
\begin{equation}
\mathbf{E}_{\otimes2}[K_q](z,z';w,w')=
q\sum_{j=0}^{q-1}(-1)^j\binom{q}{j+1}\binom{q+j}{q}
\frac{(1-w'\bar z')^{q-j-1}(z-w')^j(\bar z'-\bar w)^{j}}{(1-z\bar w)^{q+j+1}}.
\end{equation}
so that
\begin{multline}
\mathbf{E}_{\otimes2}[K_q](z,z';z,z')=
q\sum_{j=0}^{q-1}\binom{q}{j+1}\binom{q+j}{q}
\frac{(1-|z'|^2)^{q-j-1}|z-z'|^{2j}}{(1-|z|^2)^{q+j+1}}
\\
=q(1-|z|^2)^{-2}\sum_{j=0}^{q-1}\binom{q}{j+1}\binom{q+j}{q}
\bigg(\frac{1-|z'|^2}{1-|z|^2}\bigg)^{q-j-1}\bigg(\frac{|z-z'|}{1-|z|^2}
\bigg)^{2j}.
\end{multline}

\section{Musings on polyanalytic Bergman metrics} 
\label{musing-BM}

\subsection{Bergman's first metric}
\label{subsec-BM}
Stefan Bergman \cite{Berg} considered the Bergman kernel for the weight 
$\omega(z)\equiv1$ only. He also introduced the so-called 
\emph{Bergman metric} in two different ways. We will now discuss the 
ramifications of Bergman's ideas in the presence of a non-trivial weight
$\omega$. 
We interpret the introduction of the weight $\omega$ as equipping the 
domain $\Omega$ with the isothermal Riemannian metric and associated 
two-dimensional volume form
\begin{equation}
\diff s_\omega(z)^2:=\omega(z)|\diff z|^2,\qquad \diff A_\omega(z):=
\omega(z)\diff A(z).
\label{eq-metric1}
\end{equation}
\emph{Bergman's first metric} on $\Omega$ is then given by
\begin{equation}
\diff s_\omega^{\text{\ding{172}}}(z)^2:=K(z,z)\diff s_\omega(z)^2
=K(z,z)\omega(z)|\diff z|^2,
\label{eq-Bmetric1}
\end{equation}
where $K=K_{\Omega,\omega}$ is the weighted Bergman kernel given by \eqref{eq-bk1}.

\subsection{Bergman's second metric}
\label{subsec-BM2}
The second metric is related to the Gaussian curvature.
The curvature form for the original metric \eqref{eq-metric1} is (up to a
positive constant factor) given by
\begin{equation}
\varkappa:=-\hDelta\log\omega(z)\diff A(z),
\label{eq-curv1}
\end{equation}
The curvature form for Bergman's first metric \eqref{eq-Bmetric1} is similarly
\begin{equation}
\varkappa^{\text{\ding{172}}}:=-\hDelta\log[K(z,z)\omega(z)]\diff A(z)
=-[\hDelta\log K(z,z)+\hDelta\log\omega(z)]\diff A(z),
\label{eq-curv2}
\end{equation}
and we are led to propose the difference
\begin{equation}
\varkappa-\varkappa^{\text{\ding{172}}}=(\hDelta\log K(z,z))\diff A(z),
\label{eq-curv3}
\end{equation} 
as the two-dimensional volume form of a metric, \emph{Bergman's second metric}:
\begin{equation}
\diff s_\omega^{\text{\ding{173}}}(z)^2:=\hDelta\log K(z,z)|\diff z|^2.
\label{eq-Bmetric2}
\end{equation} 
We should remark that unless $K(z,z)\equiv0$, the function 
$z\mapsto\log K(z,z)$ is subharmonic in $\Omega$. This is easily seen from 
the identity
\begin{equation}
K(z,z)=\sum_{j=1}^{+\infty}|\phi_j(z)|^2,
\label{eq-diagrestr1}
\end{equation}
so $\hDelta\log K(z,z)\ge0$ holds throughout $\Omega$. This means that we 
can expect \eqref{eq-Bmetric2} to define a Riemannian metric in $\Omega$
except in very degenerate situations. Bergman's second metric appears to be 
the more popular metric the several complex variables setting (see, e.g.,
Chapter 3 of \cite{GKK}).  In the case of the unit disk $\Omega=\D$ 
and the constant weight $\omega(z)\equiv1/(2\pi)$, we find that 
\[
K(z,w)=\frac{2}{(1-z\bar w)^2},
\]
so that 
\[
\diff s_\omega^{\text{\ding{172}}}(z)^2=\frac{2|\diff z|^2}{(1-|z|^2)^2}
\quad\text{and}\quad
\diff s_\omega^{\text{\ding{173}}}(z)^2=\frac{2|\diff z|^2}{(1-|z|^2)^2},
\]
which apparently coincide. This means that the first and second Bergman 
metrics are the same in this model case (the reason is that 
the curvature of the first Bergman metric equals $-1$).

\subsection{The first polyanalytic Bergman metric}
We continue the setting of the preceding subsection. 
Following in the footsteps of Bergman (see Subsection \ref{subsec-BM}), 
we would like to introduce polyanalytic analogues of Bergman's first and
second metrics, respectively. Let us first discuss a property of the 
weighted Bergman kernel $K=K_1$ (with $q=1$). \emph{The function $K(z,w)$ on
$\Omega\times\Omega$ is uniquely determined by its diagonal restriction
$K(z,z)$}. This is so because because the diagonal $z=w$ is a set of uniqueness
for functions holomorphic in $(z,\bar w)$. So, if we fix the weight $\omega$,
the first and second Bergman metrics both retain all essential properties of
the kernel $K(z,w)$ itself. The same can not be said for the weighted 
$q$-analytic Bergman kernel $K_q$. To remedy this, we consider the 
double-diagonal restriction $z=w$ and $z'=w'$ in the lifted kernel 
$\mathbf{E}_{\otimes2}[K_q]$ instead:
\begin{equation}
\mathbf{E}_{\otimes2}[K_q](z,z';z,z')=
\sum_{j=1}^{+\infty}|\mathbf{E}[\phi_j](z,z')|^2,
\qquad (z,z')\in\Omega\times\C.
\label{eq-bk4}
\end{equation}
If we know just the restriction to the double diagonal $z=w$ and $z'=w'$ of
$\mathbf{E}_{\otimes2}[K_q]$ we are able to recover the full kernel 
lifted kernel $\mathbf{E}_{\otimes2}[K_q]$. 
If we put $z'=z+\epsilon$, where $\epsilon\in\C$, 
we may expand the extension of $\phi_j$ in a finite Taylor series:
\begin{equation}
\mathbf{E}[\phi_j](z,z+\epsilon)=\sum_{k=0}^{q-1}\frac{1}{k!}\dbar_z^k\phi_j(z)
\bar\epsilon^k.
\label{eq-bk5}
\end{equation}
As we insert \eqref{eq-bk5} into \eqref{eq-bk4}, the result is
\begin{equation}
\mathbf{E}_{\otimes2}[K_q](z,z+\epsilon;z,z+\epsilon)=
\sum_{k,k'=0}^{q-1}\frac{\bar\epsilon^{k}\epsilon^{k}}{k!(k')!}
\sum_{j=1}^{+\infty}\dbar_z^{k}\phi_j(z)\partial_z^{k'}\bar\phi_j(z),
\qquad (z,\epsilon)\in\Omega\times\C.
\label{eq-bk6}
\end{equation}
So, to generalize the first Bergman metric we consider the family of (possibly
degenerate) metrics 
\begin{equation}
\diff s_{q,\omega,\epsilon}^{\text{\ding{174}}}(z)^2:=
\mathbf{E}_{\otimes2}[K_q](z,z+\epsilon;z,z+\epsilon)\omega(z)|\diff z|^2
=\sum_{k,k'=0}^{q-1}\frac{\bar\epsilon^{k}\epsilon^{k'}}{k!(k')!}
\sum_{j=1}^{+\infty}\dbar_z^{k}\phi_j(z)\partial_z^{k'}\bar\phi_j(z)
\omega(z)|\diff z|^2,
\label{eq-bm101}
\end{equation}
indexed by $\epsilon\in\C$. We observe that in \eqref{eq-bm101} we may think of 
$(\epsilon^0,\ldots,\epsilon^{q-1})$ as a general vector in $\C^q$, by forgetting
about the interpretation of the superscript as a power, and \eqref{eq-bm101}
still defines a (possibly degenerate) metric indexed by the $\C^q$-vector
$(\epsilon^0,\ldots,\epsilon^{q-1})$; this should have an interpretation in 
terms of jet manifolds. In other words, the $q\times q$ matrix
\begin{equation}
\omega(z)
\begin{pmatrix}
\mathbf{E}_{\otimes2}[K_q](z,z;z,z)
&\cdots&
\frac{1}{(q-1)!}
\partial_{z'}^{q-1}\mathbf{E}_{\otimes2}[K_q](z,z';z,z')\big|_{z':=z}
\\
 \vdots & \ddots & \vdots  \\
\frac{1}{(q-1)!}
\dbar_{z'}^{q-1}\mathbf{E}_{\otimes2}[K_q](z,z';z,z')\big|_{z':=z}
&\cdots&
\frac{1}{(q-1)!(q-1)!}
\dbar_{z'}^{q-1}\partial_{z'}^{q-1}\mathbf{E}_{\otimes2}[K_q](z,z';z,z')\big|_{z':=z}
\end{pmatrix}
,
\label{eq-varmat1}
\end{equation}
which depends on $z\in\Omega$, is positive semi-definite. 

\subsection{The second polyanalytic Bergman metric}

We turn to Bergman's second metric.  The function 
\begin{equation}
L_q(z,z'):=\log\mathbf{E}_{\otimes2}[K_q](z,z';z,z')
\label{eq-Lfun1}
\end{equation}
is basic to the analysis, where the expression on the right-hand side is as in
\eqref{eq-bk4}. We think of $(z,\bar z')$ as holomorphic coordinates, and 
form the corresponding Bergman metric:
\begin{equation*}
\partial_z\bar\partial_{z}L_q(z,z')\vert\diff z\vert^2
+\partial_z\partial_{z'}L_q(z,z')\diff z\diff z'+
\bar\partial_{z}\bar\partial_{z'}L_q(z,z') 
\diff\bar z\diff\bar z'+
\partial_{z'}\bar\partial_{z'}L_q(z,z')\vert\diff z'\vert^2
\end{equation*}
Next, we write $z'=z+\epsilon$, so that $\diff z'=\diff z+\diff\epsilon$; 
to simplify as much as possible, we restrict to $\diff\epsilon=0$, 
so that $\diff z'=\diff z$. Then the above metric becomes
\[
(\hDelta_z+\hDelta_{z'})L_q(z,z')\vert\diff z\vert^2
+2\re[\partial_z\partial_{z'}L_q(z,z')(\diff z)^2],
\]
which after full implementation of the coordinate change becomes
\begin{multline}
\diff s_{q,\omega,\epsilon}^{\text{\ding{175}}}(z)^2
:=(\hDelta_z+\hDelta_{z'})L_q(z,z')\vert\diff z\vert^2
+2\re[\partial_z\partial_{z'}L_q(z,z')(\diff z)^2]\big|_{z'=:z+\epsilon}
\\
=(\hDelta_z+2\hDelta_{\epsilon}-\bar\partial_z\partial_\epsilon
-\partial_z\bar\partial_\epsilon)[L_q(z,z+\epsilon)]\vert\diff z\vert^2
+2\re\big\{(\partial_z\partial_{\epsilon}-
\partial_\epsilon^2)[L_q(z,z+\epsilon)](\diff z)^2\big\}.
\label{eq-bpam2}
\end{multline}
This gives us a metric parametrized by $\epsilon$, for $\epsilon$ close 
to $0$, which we may think of as Bergman's second polyanalytic metric.
This metric \eqref{eq-bpam2} is, generally speaking, not isothermal. 
We note here that to a given $C^\infty$-smooth non-isothermal metric we 
may associate an appropriate quasiconformal mapping which maps the 
non-isothermal metric to an isothermal one.

\subsection{The two polyanalytic Bergman metrics in the model case of 
the unit disk with constant weight}
To make this presentation as simple as possible, we now focus our attention 
on the biholomorphic (bianalytic) case $q=2$. Then
\begin{equation}
\mathbf{E}_{\otimes2}[K_2](z,z';z,z')=
4\frac{1-|z'|^2}{(1-|z|^2)^3}+6\frac{|z-z'|^2}{(1-|z|^2)^4},
\end{equation}
and if we substitute $z'=z+\epsilon$, the result is
\begin{multline}
\mathbf{E}_{\otimes2}[K_2](z,z+\epsilon;z,z+\epsilon)=
4\frac{1-|z+\epsilon|^2}{(1-|z|^2)^3}+6\frac{|\epsilon|^2}{(1-|z|^2)^4}
\\
=4\frac{1}{(1-|z|^2)^2}
-4\frac{\bar\epsilon z+\epsilon\bar z}{(1-|z|^2)^3}
+|\epsilon|^2\frac{2+4|z|^2}{(1-|z|^2)^4},
\end{multline}
so that Bergman's first metric \eqref{eq-bm101} becomes
\begin{equation}
\diff s_{2,\omega,\epsilon}^{\text{\ding{174}}}(z)^2
=\bigg\{4\frac{1}{(1-|z|^2)^2}
-8\frac{\re[\bar\epsilon z]}{(1-|z|^2)^3}
+|\epsilon|^2\frac{2+4|z|^2}{(1-|z|^2)^4}\bigg\}|\diff z|^2, 
\label{eq-temp1}
\end{equation}
indexed by $\epsilon$. The corresponding $2\times2$ matrix
\[
2
\begin{pmatrix}
2(1-|z|^2)^{-2} & -2z(1-|z|^2)^{-3}
\\
-2\bar z(1-|z|^2)^{-3} & (1+2|z|^2)(1-|z|^2)^{-4}
\end{pmatrix}
\]
is then positive definite (this fact generalizes to arbitrary $q$ as we 
mentioned previously). 

A more involved calculation shows that Bergman's 
second polyanalytic metric \eqref{eq-bpam2} obtains the form 
\begin{equation}
\diff s_{2,\omega,\epsilon}^{\text{\ding{175}}}(z)^2
=\frac{4}{(1-|z|^2)^2}\vert\diff z\vert^2\quad (\mathrm{mod}\,
\epsilon,\bar\epsilon),
\end{equation}
where the modulo is taken with respect to the ideal generated by $\epsilon$ 
and $\bar\epsilon)$. In this case, this is the same as putting $\epsilon=0$
at the end.
This is actually the same as \eqref{eq-temp1} to this lower degree of precision.

\section{Interlude: a priori control on point evaluations}
\label{sec-interlude}

\subsection{Introductory comments}
Let us consider the unit disk $\D$, and a given subharmonic function 
$\psi:\D\to\R$. If $u:\D\to\C$ is holomorphic and nontrivial, then
$\log|u|$ is subharmonic. Then $\log|u|+\psi$ is subharmonic, and by
convexity, $|u|^2\e^{2\psi}$ is subharmonic as well. By the sub-mean value
property of subharmonic functions, we have the estimate
\[
|u(0)|^2\e^{2\psi(0)}\le\frac{1}{\pi}\int_{\D} |u|^2 
\e^{2\psi} \diff A,
\]  
which allows us to control the norm of the point evaluation at the origin in
$\mathrm{A}^2(\D,\e^{2\psi})$. If we would try this with a bianalytic function 
$u$, we quickly run into trouble as $\log|u|$ need not be subharmonic then 
(just consider, e.g., $u(z)=1-|z|^2$). So we need a different approach.

\subsection{The basic estimate}
We begin with a lemma. We decompose the given bianalytic function as 
$u(z)= u_1(z)+c\bar z+|z|^2u_2(z)$, 
where $c$ is a constant and $u_j$ is holomorphic for $j=1,2$. Let $\diff s(z)
:=|\diff z|$ denote arc length measure.


\begin{lem} 
\label{lem1}
If $u(z)= u_1(z)+c\bar z+|z|^2u_2(z)$ is bianalytic and $\psi$ is subharmonic 
in $\D$, then
\begin{equation*}
\int_0^1  \bigg| u_1(0) + r^2 u_2(0) + \frac{cr}{\pi} 
\int_{\Te}\bar\zeta\psi(r \zeta)\diff s(\zeta)  \bigg|^2 r \diff r 
\leq \frac{\e^{-2\psi(0)}}{2\pi} \int_{\D} |u|^2 \e^{2\psi} \diff A. 
\end{equation*}
\end{lem}

\begin{proof}
We write $\psi_r(\zeta):= \psi(r\zeta)$ for the dilation of $\psi$, and 
let $g_r$ be the holomorphic and zero-free function given by 
\[ \log g_r(z) := \frac{1}{2\pi} \int_{\Te} 
\frac{1+z\bar\zeta}{1-z\bar\zeta}\psi_r(\zeta)\diff s(\zeta). 
\]
This of course defines $g_r$ uniquely, and also picks a suitable branch of
$\log g_r$. Taking real parts, we see that 
\[ 
\log|g_r| =\re\log g_r =\Pop[\psi_r](z):=\frac{1}{2\pi} \int_{\Te} 
\frac{1-|z|^2}{|1-z\bar\zeta|^2}\psi_r(\zeta)\diff s(\zeta), 
\qquad z\in\D, 
\]
where $\Pop[\psi_r]$ denotes the Poisson extension of $\psi_r$. 
Then, in the standard sense of boundary values,  
$|g_r|^2 = \e^{2\psi_r}$ on $\Te$. By the mean value property, we have that 
\begin{multline}
\frac{1}{2\pi}\int_{\Te}u(r\zeta)g_r(\zeta) \diff s(\zeta) 
= \frac{1}{2\pi}\int_{\Te}\{u_1(r \zeta)+cr\bar\zeta+r^2u_2(r\zeta)\}\,
g_r(\zeta) \diff s(\zeta) 
\\
=u_1(0) g_r(0) +cr g'_r(0)+  r^2 u_2(0) g_r(0)  
= g_r(0) \bigg\{u_1(0)+r^2u_2(0)+cr\frac{g_r'(0)}{g_r(0)}\bigg\}.
\label{eq-4.1.1}
\end{multline}
Here, 
\[ 
\frac{g_r'(z)}{g_r(z)}=\frac{\diff}{\diff z}\log g_r(z) = \frac{1}{\pi} 
\int_{\Te}\frac{\bar\zeta}{(1-z\bar\zeta)^2} 
\psi_r(\zeta)\,\diff s(\zeta),
\] 
so that in particular
\begin{equation} 
\label{Grformula}
\frac{g_r'(0)}{g_r(0)} =\frac{1}{\pi}\int_{\Te}\bar\zeta\psi_r(\zeta) 
\diff s(\zeta). 
\end{equation}
By \eqref{eq-4.1.1} combined with the Cauchy-Schwarz inequality,  
\begin{equation*}
|g_r(0)|^2\bigg|u_1(0)+r^2u_2(0)+cr\frac{g_r'(0)}{g_r(0)}\bigg|^2 
\leq\frac{1}{2\pi}\int_{\Te}|u(r\zeta)g_r(\zeta)|^2 \diff s(\zeta) 
=\frac{1}{2\pi}\int_{\Te} |u(r\zeta)|^2 \e^{2\psi_r(\zeta)}\diff s(\zeta).
\end{equation*}
Next, we multiply both sides by $2r$ and integrate with respect to $r$:
\begin{equation}
\int_0^1|g_r(0)|^2 \bigg|u_1(0)+r^2u_2(0)+cr\frac{g_r'(0)}{g_r(0)} 
\bigg|^22r\diff r\leq 
\frac{1}{\pi} \int_{\D}|u(z)|^2\e^{2\psi(z)}\diff A(z). 
\label{eq-4.1.4}
\end{equation}
The sub-mean value property applied to $\psi$ gives that 
$\e^{2\psi(0)}\leq|g_r(0)|^2$, so we obtain from \eqref{eq-4.1.4} that
\begin{equation} 
\label{lemma1mainformula}
\e^{2\psi(0)}\int_0^1\bigg|u_1(0)+r^2u_2(0)+cr 
\frac{g_r'(0)}{g_r(0)}\bigg|^2 r\diff r
\leq \frac{1}{2\pi}\int_{\D}|u(z)|^2\e^{2\psi(z)}\diff A(z),
\end{equation}
as claimed.
\end{proof}

\subsection{Applications of the basic estimate}

We can now estimate rather trivially the $\dbar$-derivative at the origin.

\begin{prop} 
\label{pointwiseneg0}
If $u(z)$ is bianalytic and $\psi$ is subharmonic 
in $\D$, then 
\[
|\dbar u(0)|^2\le\frac{3}{\pi}\e^{-2\psi(0)}\int_{\D}|u|^2\e^{2\psi}\diff A.
\]
\end{prop}

\begin{proof}
We apply Lemma \ref{lem1} to the function 
$v(z):=zu(z)=zu_1(z)+|z|^2(c+zu_2(z))$:
\begin{equation*}
\frac{|c|^2}{6}=\int_0^1|r^2c|^2r\diff r 
\leq \frac{\e^{-2\psi(0)}}{2\pi}\int_{\D}|v|^2\e^{2\psi}\diff A\le
\frac{\e^{-2\psi(0)}}{2\pi}\int_{\D}|u|^2\e^{2\psi}\diff A. 
\end{equation*}
It remains to observe that $c=\dbar u(0)$.
\end{proof}

We can also estimate the value at the origin, under an additional assumption.

\begin{prop} 
\label{pointwiseneg}
Suppose $u$ is bianalytic and that $\psi$ is subharmonic in $\D$.
If $\psi\le0$ in $\D$, then 
\[ 
|u(0)|^2\leq\frac{8}{\pi}[1 + 6|\psi(0)|^2 ]\e^{-2\psi(0)}\int_{\D} |u|^2 
\e^{2\psi}\diff A. 
\]
\end{prop}

\begin{proof}
We will use the decomposition $u(z)=u_1(z)+c\bar z +|z|^2u_2(z)$, where 
$u_1,u_2$ are both holomorphic. Proposition \ref{pointwiseneg0} allows us
estimate $c=\dbar u(0)$, so that
\begin{multline}
\int_0^1\bigg| \frac{cr}{\pi}\int_{\Te} \bar \zeta \psi(r \zeta)
\diff s(\zeta) \bigg|^2 r\diff r 
\le \int_0^1\bigg\{\frac{|c|r}{\pi}\int_{\Te} |\psi(r \zeta)|
\diff s(\zeta) \bigg\}^2 r\diff r 
\\
\leq \int_0^14|c|^2r^2|\psi(0)|^2r\diff r
=|c|^2|\psi(0)|^2\leq \frac{3}{\pi} |\psi(0)|^2 \e^{-2\psi(0)} 
\int_{\D} |u|^2 \e^{2\psi} \diff A.
\label{eq-4.1.6}
\end{multline}
Note that we used the subharmonicity of $\psi$, and that $\psi\le0$.
By the standard Hilbert space inequality $\|x+y\|^2\le 2(\|x\|^2+\|y\|^2)$,
it follows from Lemma \ref{lem1} and the above estimate \eqref{eq-4.1.6} that
\begin{multline} 
\label{almostfinalform}
\int_0^1|u_1(0)+ r^2 u_2(0)|^2 r \diff r\le 2
\int_0^1  \bigg| u_1(0) + r^2 u_2(0) + \frac{cr}{\pi} 
\int_{\Te}\bar\zeta\psi(r\zeta)\diff s(\zeta)\bigg|^2 r\diff r
\\
+2\int_0^1  \bigg|\frac{cr}{\pi} 
\int_{\Te}\bar\zeta\psi(r \zeta)\diff s(\zeta)\bigg|^2 r\diff r
\leq\frac{\e^{-2\psi(0)}}{\pi}
\int_{\D} |u|^2 \e^{2\psi} \diff A +  \frac{6}{\pi}|\psi(0)|^2 \e^{-2\psi(0)} 
\int_{\D} |u|^2 \e^{2\psi} \diff A
\\
= \frac{1}{\pi}\left\{1+6|\psi(0)|^2\right\}\e^{-2\psi(0)} 
\int_{\D} |u|^2 \e^{2\psi} 
\diff A. 
\end{multline}
Next, we expand the left hand side of \eqref{almostfinalform}: 
\begin{multline} 
\int_0^1|u_1(0)+ r^2 u_2(0)|^2 r \diff r=
\int_0^1 \left\{ |u_1(0)|^2 + r^4|u_2(0)|^2 + 2r^2 
\mathrm{Re}[ \overline{u_1(0)}u_2(0)]\right\} 
r \diff r 
\\
= \frac12 |u_1(0)|^2 + \frac16 |u_2(0)|^2 + 
\frac12 \mathrm{Re}[ \overline{u_1(0)} u_2(0)]
= \frac18  |u_1(0)|^2 +\frac16 
\left|\frac{3}{2} u_1(0)+ 
u_2(0)\right|^2 
\geq \frac18 |u_1(0)|^2. 
\label{eq-4.1.7}
\end{multline}
So, \eqref{almostfinalform} and \eqref{eq-4.1.7} together give that
\[ 
|u_0(0)|^2 \le\frac{8}{\pi}\e^{-2\psi(0)}\left[1+6|\psi(0)|^2\right]\,\e^{-2\psi(0)} 
\int_{\mathbb{D}} |u|^2 \e^{2\psi} \diff A,
\]
as needed.
\end{proof}

\subsection{The effective estimate of the point evaluation}
The problem with Proposition \ref{pointwiseneg} as it stands is the need to
assume that $\psi\le0$. We now show how to reduce the assumption to a minimum.

Given a positive Borel measure $\mu$ on $\D$ with finite \emph{Riesz mass}
\[
\int_\D (1-|z|^2)\diff\mu(z)<+\infty,
\]
we associate the \emph{Green potential} $\mathbf{G}(\mu)$, which is given by
 
\[ 
\Gop[\mu](z):= \frac{1}{\pi}\int_{\D}\log 
\bigg|\frac{z-w}{1-z\bar w}\bigg|^2\diff\mu(w),\qquad z\in\D. 
\]
The Green potential of a positive measure is subharmonic.
It is well-known that the Laplacian $\hDelta\psi$ of a subharmonic function
$\psi$ is a positive distribution, which can be identified with a positive 
Borel measure on $\D$. Then, if $\hDelta\psi$ has finite Riesz mass, we
may form the potential $\Gop[\hDelta\psi]$, which differs from $\psi$ by
a harmonic function. 

\begin{prop} 
\label{pointwise1}
Let $\psi$ be subharmonic and $u$ bianalytic on $\D$. If $\hDelta\psi$ has 
finite Riesz mass, then 
\[ 
|u(0)|^2\leq\frac{8}{\pi}\big[1+6|\mathbf{G}[\hDelta \psi](0)|^2 
\big]\, \e^{-2\psi(0)}\int_{\mathbb{D}} |u|^2 \e^{2\psi} \diff A. 
\]
\end{prop}

\begin{proof}
We decompose $\psi =\Gop[\hDelta\psi]+h$, where $h$ is a harmonic function. 
To the harmonic function $h$ we associate a zero-free analytic function 
$H$ such that $h =\log|H|$. As the Green potential has 
$\mathbf{G}[\hDelta\psi]\leq 0$ (this is trivial by computation; it also
expresses a version of the maximum principle), we may invoke Proposition 
\ref{pointwiseneg} with $\Gop[\hDelta\psi]$ in place of $\psi$, and $uH$ in
place of $u$ (The function $uH$ is bianalytic because $H$ is holomorphic 
and $u$ is bianalytic). The result is
\begin{multline}
|u(0)|^2\e^{2h(0)}=|u(0)H(0)|^2 
\leq \frac{8}{\pi}\big[1+6|\Gop\big[\hDelta\psi](0)|^2 \big] 
\,\e^{-2\Gop[\hDelta\psi](0)} 
\int_{\D} |u(w)H(w)|^2\e^{2\Gop[\hDelta\psi](w)} \diff A(w) 
\\
=\frac{8}{\pi}\big[1+6|\mathbf{G}\big[\hDelta \Psi](0)|^2\big]
\,\e^{-2\Gop[\hDelta\psi](0)} 
\int_{\D}|u(w)|^2\e^{2h(w)+2\Gop{\hDelta\psi](w)}}\diff A(w)
\\
=\frac{8}{\pi}\big[1+6|\mathbf{G}[\hDelta\psi](0)|^2\big]
\,\e^{-2\Gop[\hDelta\psi](0)}\int_{\D}|u(w)|^2\e^{2\psi(w)}\diff A(w),
\end{multline}
which is equivalent to the claimed inequality.
\end{proof}

\begin{rem}
As stated, the proposition is void unless
\[
\Gop[\hDelta\psi](0)=\frac{1}{\pi}\int_\D\log|w|^2\diff\mu(w)>-\infty.
\] 
\end{rem}

\section{Asymptotic expansion of polyanalytic Bergman kernels}
\label{sec-asexp}

\subsection{The weighted polyanalytic Bergman space and kernel for 
power weights}
\label{subsec-wpbspw}
As before, $\Omega$ is a domain in $\C$, and we consider weights of the form
$\omega=\e^{-2mQ}$, where $Q$ is a real-valued potential (function) on 
$\Omega$ and $m$ a scaling parameter, which is real and positive, and we 
are interested in the asymptotics as $m\to+\infty$. By abuse of notation, 
we may at times refer to $Q$ as a weight.
We may interpret this family as formed by the powers of an initial weight 
$\e^{-Q}$. We will think of the potential $Q$ as fixed, and to simplify the 
notation, we write $\mathrm{A}^2_m$ in place of $\mathrm{A}^2(\Omega,\e^{-2mQ})$.
Likewise, the reproducing kernel of $\mathrm{A}^2_m$ will be written $K_m$.
More generally, we are also interested in the weighted Bergman spaces of 
polyanalytic functions $\mathrm{PA}^2_{q}(\Omega,\e^{-2mQ})$,
which we simplify notationally to $\mathrm{PA}^2_{q,m}$. The associated 
polyanalytic Bergman kernel will be written   $K_{q,m}$. Hopefully this notation 
will not lead to any confusion. The related kernel
\begin{equation}
\tilde K_{q,m}(z,w)=K_{q,m}(z,w)\e^{-m[Q(z)+Q(w)]},
\label{eq-corrkern1}
\end{equation}
known as the \emph{correlation kernel}, then acts boundedly on $L^2(\Omega)$. 
We may think of both $m$ and $q$ as quantization parameters, because we would
generally expect the integral operator associated with \eqref{eq-corrkern1} to
tend to the identity as $m$ or $q$ tends to infinity.  When $q$ tends to 
infinity, this is related to the question of approximation by polynomials in 
$z$ and $\bar z$ in the weighted space $L^2(\Omega,\e^{-2mQ})$. When instead 
$m$ tends to infinity, ``geometric'' properties of $\Omega$ and $Q$ become 
important; e.g., we would need to ask that $\hDelta Q\ge0$. 

\subsection{Weighted Bergman kernel expansions for power weights}
We should observe first that the above setting of power weights 
$\omega=\e^{-2mQ}$
has a natural extension to the setting of line bundles over (usually compact) 
complex manifolds in one or more dimensions. 
As the Bergman kernel is of geometric importance, it has been studied 
extensively in the line bundle setting (see, e.g., \cite{mm}). To make 
matters more precise, suppose we are given a (compact) complex manifold 
$\mathcal{X}$ and a holomorphic line bundle $L$ over $\mathcal{X}$ supplied 
with an Hermitian metric $h$, where in our above simplified setting 
$h\sim \e^{-2Q}$ is a local representation of the metric $h$. Next,  if 
$m$ is a positive integer, the weight $\e^{-2mQ}$ appears when as the local 
metric associated with the tensor power $L^{\otimes m}$ of the line bundle $L$. 
This is often done because the initial line bundle $L$ might have admit 
only very few holomorphic sections. 
To be even more precise, the one-variable analogue of the weight which is 
studied in the context of general complex manifolds is $\e^{-2mQ}\hDelta Q$, 
and it is clearly necessary to ask that $\hDelta Q$  is  nonnegative. 
Here, we shall stick to the choice of weights  $\omega=\e^{-2mQ}$ which seems 
more basic.
 
It is -- as mentioned previously -- generally speaking a difficult task 
to obtain the weighted Bergman kernel explicitly. Instead, we typically 
need to resort to approximate formulae. Provided that the potential $Q$ is 
sufficiently smooth with $\hDelta Q>0$, the weighted Bergman kernel has 
an asymptotic expansion (up to an error term of order $\Ordo(m^{-k})$)
 \begin{equation}
K_m(z,w) \sim \Lfun^{\langle k\rangle}(z,w)\e^{2mQ(z,w)} 
=\{m \Lfun_{0}(z,w) + \Lfun_{1}(z,w) + \dots + m^{-k+1} 
\Lfun_{ k}(z,w) \}\, \e^{2mQ(z,w)} 
\label{eq-asexp1.01}
\end{equation}
near the diagonal $z=w$.
Here, the function $Q(z,w)$ is a local polarization of $Q(z)$, i.e., it is 
an almost holomorphic function of $(z,\bar w)$ near the diagonal $z=w$ 
with $Q(z,z) = Q(z)$. For details on local polarizations, see, e.g., 
\cite{bbs}, \cite{ahm3}. It is implicit in \eqref{eq-asexp1.01} that
\begin{equation}
\Lfun^{\langle k\rangle}(z,w)
=m \Lfun_{0}(z,w)+\Lfun_{1}(z,w)+\cdots +m^{-k+1}\Lfun_{k}(z,w),   
\label{eq-asexp1.02}
\end{equation}
where each term $\Lfun_{j}(z,w)$ is holomorphic in $(z,\bar w)$ near the 
diagonal $z=w$. The asymptotic expansion \eqref{eq-asexp1.01} was developed 
in the work of  Tian \cite{tian}, Yau \cite{yau}, Zelditch \cite{zel}, and 
Catlin \cite{cat}. 
We should mention that in a different but related setting, some of the 
earliest results on asymptotic expansion of Bergman kernels were obtained 
by H\"ormander \cite{horm} and  Fefferman \cite{feff}. As for the asymptotics 
\eqref{eq-asexp1.01}, various approaches have been developed, see e.g. 
\cite{tian}, \cite{mm}, \cite{liulu}, \cite{bbs}, \cite{dougklev}. In 
\cite{lu}, Lu worked out the first four terms explicitly using the peak 
section method of Tian.   
Later, Xu \cite{xu} obtained a closed graph-theoretic formula for the 
general term of the expansion.  

\subsection{The connection with random normal matrix theory}
Apart from to the connection with to geometry, Bergman kernels have also 
been studied for the connection to the theory of random normal matrices. See, 
e.g., the papers \cite{ahm1}, \cite{ahm2}, \cite{ahm3} ,\cite{ber}. 
The joint probability density of the eigenvalues may be expressed in terms of 
the reproducing kernels of the spaces of polynomials
\[ 
\mathrm{Pol}_{n,m} := \mathrm{span}_{0 \leq j \leq n-1} z^j  \subset
L^2 (\C,\e^{-mQ}), 
\] 
where $n$ is the size of the matrices (i.e., the matrices are $n\times n$). The 
reproducing kernels are actually more natural from the point of view of 
marginal probability densities.  
The eigenvalues of random normal matrices follow the law of a Coulomb gas of 
negatively charged particles in an external field, with a special value for 
the inverse temperature $\beta=2$ (see, e.g., \cite{hm}). The asymptotics of 
these polynomial reproducing kernels can be analyzed in the limit as 
$m,n\to+\infty$ while $n=m+\Ordo(1)$, and in the so-called bulk of the spectral 
droplet its analysis is quite similar to that of the weighted Bergman kernel 
$K_m$ (see Subsection \ref{subsec-wpbspw}).  


\subsection{The polyanalytic Bergman spaces and Landau levels}
\label{subsec-landau}
The polyanalytic spaces $\mathrm{PA}^2_{q,m}$ (see Subsection 
\ref{subsec-wpbspw}) 
have been considered in, e.g., \cite{vasil}, \cite{abreu}, \cite{mo}, and 
\cite{hh}, focusing on $\Omega=\C$ and the quadratic potential $Q(z)= |z|^2$, 
so that the corresponding weight $\omega(z)=\e^{-2m\vert z\vert^2}$ is Gaussian. 
One instance where these spaces appear is as eigenspaces of the operator
\[ 
\tilde{\hDelta}_z = - \hDelta_z + 2m\bar z \bar \partial_z, 
\] 
which is densely defined in $L^2(\C,\e^{-2m|z|^2})$. More precisely, the 
eigenspaces are of the form $\mathrm{PA}^2_{q,m} \ominus \mathrm{A}^2_{q-1,m}$, 
where $\ominus$ denotes the orthogonal difference within the Hilbert space 
$L^2(\C,\e^{-2m|z|^2})$. The related operator 
\[
\Hop := \Mop_{\e^{-m |z|^2}} \tilde{\hDelta}_z\Mop_{\e^{m|z|^2}}, 
\]
where $\Mop$ denotes the operator of multiplication by the function in 
the subscript, is the Hamiltonian associated with a single electron in $\C$ 
within a uniform magnetic field perpendicular to the plane. The operator 
is called the \emph{Landau Hamiltonian} and the eigenspaces 
are commonly referred to as \emph{Landau levels}. 

Returning to a general potential $Q$, we may think of the orthogonal difference 
spaces 
\[ 
\delta \mathrm{PA}^2_{q,m} := \mathrm{PA}^2_{q,m}\ominus 
\mathrm{PA}^2_{q-1,m}, 
\]
as way to generalize the notion of Landau levels. Here, we note that in 
\cite{rt}, Rozenblum and Tashchiyan identify approximate spectral subspaces 
related to a Hamiltonian describing a single electron in a magnetic field 
with strength $\hDelta Q$. It may be observed that those spaces are in a sense 
dual to the orthogonal difference spaces $\delta\mathrm{PA}^2_{q,m}$. 

\subsection{Asymptotic analysis of weighted polyanalytic 
Bergman kernels}
It is our aim is to supply an algorithm to compute a near-diagonal asymptotic 
expansion of the weighted polyanalytic Bergman kernel $K_{q,m}$ as $m$ tends
to infinity. Under suitable assumptions on the potential $Q$,
the kernel has -- up to an error term -- the (local) expansion
\[ 
K_{q,m}(z,w)\sim
\big\{ m^{q} \Lfun^q_{0}(z, w) + m^{q-1}\Lfun^q_{1}(z, w)+\dots
+ m^{q-k}\Lfun^q_{k}(z,w) \big\} 
\e^{2mQ(z,w)},
\]
where near the diagonal $z=w$, the coefficient functions $\Lfun^q_{j}(z,w)$ are 
$q$-analytic in each of the variables $z$ and $\bar w$. As before, $Q(z,w)$
is a polarization of $Q(z)$. The control of the error term is of order 
$\Ordo(m^{q-k-1})$ in the parameter $m$. We work this out in detail in the
bianalytic case $q=2$, and obtain explicit formulae for $\Lfun^2_j$ when
$j=0,1,2$ (see Remark \ref{rem-AA}). Note that we may use the formula 
$\delta K_{q,m}:=K_{q,m}-K_{q-1,m}$ to asymptotically express the reproducing 
kernel $\delta K_{q,m}$ associated with the orthogonal difference space 
$\delta\mathrm{PA}^2_{q,m}$. 

We base our approach is based on the recent work of Berman, Berndtsson, and 
Sj\"ostrand \cite{bbs}, where they give an elementary algorithm to compute 
the coefficients $\Lfun_{j}=\Lfun^1_{j}$ of the asymptotic expansion of the 
weighted Bergman kernel. 
We now extend their method to polyanalytic functions of one variable. 
We explain our results in detail in the case $q=2$ and compute the first 
two terms of the asymptotic expansion. We supply some hints on how to the 
approach applies for bigger values of $q$. The actual computations become 
unwieldy and for this reason they are not presented. 
Our asymptotic analysis leads to the following blow-up result. 

\begin{thm} 
\label{blowupthm}
$(q=2)$ Let $\Omega$ be a domain in $\C$, and suppose $Q:\Omega\to\R$ 
is $C^4$-smooth 
with $\hDelta Q>0$ on $\Omega$ and
\[
\sup_{\Omega}\frac{1}{\hDelta Q}\hDelta\log\frac{1}{\hDelta Q}<+\infty.
\]  
Suppose, in addition, that $Q$ is real-analytically smooth in a neighborhood
of a point $z_0\in\Omega$, and introduce the rescaled coordinates 
$\xi':=[2m\hDelta Q(z_0)]^{-1/2}\xi$ and $\eta':=[2m\hDelta Q(z_0)]^{-1/2}\eta$.
Then there exists a positive $m_0$ such that, for all $m \geq m_0$, we have 
\begin{equation}
\frac{1}{2m \hDelta Q(z_0)}\vert K_{2,m}(z_0+
\xi',z_0+\eta')\vert\e^{- mQ( z_0+\xi')-mQ (z_0+\eta')}
=\frac{|2-|\xi-\eta|^2|}{\pi}\,\e^{-\frac12|\xi-\eta|^2}+\Ordo(m^{-1/2}),
\label{eq-univ1}
\end{equation}
where the constant in the error term is uniform in compact subsets of 
$(\xi,\eta)\in\C^2$. 
\end{thm}


We remark that since the limiting kernel on the right-hand side of 
\eqref{eq-univ1} does not depend on the particular weight $Q$, we may 
view the above theorem as a universality result; for more on 
universality, see, e.g., \cite{deift}. The probabilistic interpretation is 
that in the large $m$ limit, the determinantal point 
process (for a definition, see \cite{hkpv}) defined by the kernel $K_{2,m}$ 
obeys local statistics (after the appropriate local blow-up) given by 
the kernel 
$\pi^{-1}(2-|\xi-\eta|^2)\e^{-|\xi-\eta|^2/2}$. This local blow-up kernel
appeared earlier in our previous paper \cite{hh} which was concerned with the 
quadratic potential $Q(z)=\vert z\vert^2$.
In that paper, we analyzed a system of noninteracting fermions described 
by the Landau Hamiltonian of Subsection \ref{subsec-landau}, so that each 
of the $q$ first Landau levels was filled with $n$ particles. As $n$ tends 
to infinity and the magnetic field is rescaled by $m$, with $m=n+\Ordo(1)$, 
the particles accumulate on the closed unit disk (which equals the 
spectral droplet in this instance) with uniform density, and the Laguerre 
polynomial kernel $\pi^{-1}L^{(1)}_{q-1}(|\xi-\eta|^2)\e^{-|\xi-\eta|^2/2}$ describes 
the local blow-up statistics in the limit in this polynomial case for 
general values of $q$. Here, $L^{(\alpha)}_r$ stands for the generalized 
Laguerre polynomial of degree $r$ with parameter $\alpha$; note that 
$L^{(1)}_1(x)=2-x$, which explains the expression in Theorem \ref{blowupthm}. 

To be more precise, this system of free fermions is a determinantal 
point process given by  the reproducing kernels of the spaces
\[ 
\mathrm{Pol}_{q,m, n} := \mathrm{span} \{ \bar z^r z^j \mid 0 \leq r 
\leq q-1, 0 \leq j \leq n-1 \} \subset L^2( \C,\e^{-2m|z|^2}). 
\]
In later work, we intend to replace the weight $\e^{-2m|z|^2}$ by a more 
general weight $\e^{-2mQ(z)}$ and apply the methods developed here to obtain
the asymptotic analysis of the corresponding stochastic processes. 
On a more general complex manifold, this should correspond to studying sections 
on line bundles $\bar L^{\otimes q} \otimes L^{\otimes n}$.

Finally, we suggest that it is probably possible to extend our results to the
several complex variables setting and in a second step, to more general 
complex manifolds.
Moreover, we would expect that asymptotic expansion results could be obtained 
for reproducing kernels associated with differential operators more general
than $\bar \partial^q$. 

\section{Tools from the approach of 
Berman-Berndtsson-Sj\"ostrand}

In this section, we review some aspects of the approach of \cite{bbs} in the 
one complex variable context (see, e.g., \cite{ahm1} for an extensive 
presentation). 

\subsection{Assumptions on the potential}
\label{subsec-assQ} 

We fix a $C^\infty$-smooth simply connected bounded domain $\Omega$.
The potential $Q:\bar\Omega\to\R$ is assumed real-analytically smooth. This 
assumption simplifies the choice of polarization $Q(z,w)$, as there is 
then a unique choice which is holomorphic in $(z,\bar w)$ locally near the 
diagonal $z=w$.
  
Fix an arbitrary point $z_0 \in \Omega$. We will carryout a local analysis of 
the kernels $K_m(z,w)$, where $z, w \in \mathbb{D}(z_0, r)\Subset\Omega$. 
We will assume that the radius $r$ and the potential $Q$ meet the following 
additional requirements (A:i)--(A:iv):

\medskip

\noindent (A:i) \hskip0.3cm
$Q$ is real-analytic in $\D(z_0,r)$ and $\hDelta Q(z) \ge \epsilon_0$ on 
$\D(z_0,r)$, for some positive constant $\epsilon_0$ (which we assume 
to be as big as possible). 

\noindent (A:ii) \hskip0.3cm
There exists a local polarization of $Q$ in $\D(z_0,r)$, i.e. a function 
$Q: \D(z_0,r) \times \D(z_0,r) \to \C$ which is holomorphic in the first 
variable and conjugate-holomorphic in the second variable, with 
$Q(z,z)=Q(z)$. 

\noindent (A:iii) \hskip0.3cm
For $z, w \in\D(z_0, r)$, we have 
$\partial_z \bar \partial_w Q(z,w)\neq0$ and $\bar\partial\theta(z,w)\neq0$ 
(this is possible because of condition (A:i)). Here, $\theta$ is the phase 
function, which is defined below.   

\noindent (A:iv) \hskip0.3cm
By Taylor's formula, we have that  
\[
2\re Q(z,w) - Q(w)- Q(z) = -\hDelta Q(z) |w-z|^2 + 
\mathrm{O}(|z-w|^3). 
\]
We then pick $r>0$ so small that  
\begin{equation}
 2\re Q(z,w) - Q(w)- Q(z) \leq -\frac12 \hDelta Q(z) |w-z|^2,
\qquad z,w \in \mathbb{D}(z_0,r).
\label{Qtaylor}
\end{equation}

\medskip


\subsection{The phase function}

We now introduce the \emph{phase function}: 
\[ 
\theta(z,w) = \frac{ Q(w) - Q(z,w)}{w-z},\qquad z\ne w.
\]
Clearly, the phase function $\theta(z,w)$ is holomorphic in $z$ and 
real-analytic in $w$. We extend it by continuity to the diagonal: 
$\theta(z, z):=\partial_z Q(z)$. In terms of the phase function, our 
assumption \eqref{Qtaylor} asks that
\begin{multline}
 2\re[(z-w)\theta(z,w)] =2\re Q(z,w)-2Q(w)
\\
\leq Q(z)-Q(w)
-\frac12 \hDelta Q(z) |w-z|^2,\qquad z,w \in \mathbb{D}(z_0,r).
\label{Qtaylor2}
\end{multline}

It will be convenient we record here some Taylor expansions which will be 
needed later on. 
In view of Taylor's formula, we have
\[
Q(w)=Q(w,w)=\sum_{j=0}^{+\infty}\frac{1}{j!}(w-z)^j\partial_z^j Q(z,w),
\]
and as consequence,
\begin{equation}
\theta(z,w) =\frac{Q(w)-Q(z,w)}{w-z}=\sum_{j=0}^{+\infty}\frac{1}{(j+1)!}
(w-z)^j\partial_z^{j+1}Q(z,w),
\label{eq-4.19}
\end{equation}
so that
\begin{equation}
\dbar_w\theta=
\sum_{j=0}^{+\infty}\frac{1}{(j+1)!}(w-z)^j\partial_z^{j+1}\dbar_w Q(z,w),
\label{eq-4.19.1}
\end{equation}
and
\begin{equation}
\partial_w\theta=
\sum_{j=0}^{+\infty}\frac{j+1}{(j+2)!}(w-z)^{j}\partial_z^{j+2} Q(z,w).
\label{eq-4.19.2}
\end{equation}
%

\subsection{Approximate local reproducing and Bergman kernels} 
\label{subsec-5.3}
In \cite{bbs}, the point of departure is an approximate reproducing identity 
for functions in $\mathrm{A}^2_m$. To state the result, we define the kernel 
\begin{equation}
M_m(z,w) := \frac{2m}{\pi} \e^{2mQ(z,w)}\bar\partial_w\theta(z,w).
\label{eq-Mm}
\end{equation}
This kernel is for obvious reasons only defined in some fixed neighborhood of
the diagonal. We shall be concerned with the (small) bidisk 
$\D(z_0,r)\times\D(z_0,r)$ where $M_m(z,w)$ is well-defined, by our 
assumptions (A:i)--(A:iv) above.To extend the kernel beyond the bidisk, we 
multiply by a smooth cut-off function $\chi_0(w)$, and think of the product 
as $0$ off the support of $\chi_0$ (also where $M_m(z,w)$ is undefined).
The function $\chi_0(w)$ is $C^\infty$-smooth with $0\le\chi_0\le1$ throughout
$\C$, and vanishes off $\D(z_0,\frac34r)$, while it has $\chi_0=1$ on 
$\mathbb{D}(z_0, \frac23 r)$, 
and is a function of the distance $|w-z_0|$. 
We ask in addition that the norm $\|\bar\partial\chi_0\|_{L^2(\Omega)}$ is 
bounded by an absolute constant, which is possible to achieve.
We will use the simplified notation
\begin{equation}
\|u\|_{m}:=\|u\|_{L^2(\Omega,\e^{-2mQ})}=\bigg\{\int_\Omega|u|^2\e^{-2mQ}
\diff A\bigg\}^{1/2}.
\label{eq-norm1}
\end{equation}

\begin{prop} 
\label{negliprop}
We have that for all $u\in\mathrm{A}^2_m$, 
\begin{equation*} 
u(z)=\int_{\Omega}u(w)\chi_0(w)M_m(z,w)\,\e^{-2mQ(w)}\diff A(w) 
+\Ordo(r^{-1}\|u\|_{m} \e^{mQ(z)-m\delta_0}),
\qquad z\in \mathbb{D}(z_0,\tfrac13 r),
\end{equation*}
where we write $\delta_0:=\frac{1}{18}r^2\epsilon_0>0$. 
The implied constant is absolute.
\end{prop}

To develop the necessary intuition, we supply the easy proof.

\begin{proof}[Proof of Proposition \ref{negliprop}]
Since the Cauchy kernel $1/(\pi z)$ is the fundamental solution to 
$\bar\partial$, we have by Cauchy-Green theorem (i.e., integration by parts) 
that
\begin{multline}
u(z) =\chi_0(z)u(z)=\int_{\Omega} \frac{1}{\pi(z-w)} \bar \partial_w \Big( u(w) 
\chi_0(w)\,\e^{2m(z-w)\theta(z,w)} \Big) \diff A(w) 
\\
= \int_{\Omega} u(w) \chi_0(w) M_m(z,w) \,\e^{-2mQ(w)} \diff A(w)  
\\
+ \frac{1}{\pi}\int_{\Omega} u(w) \frac{ \bar \partial_w \chi_0(w)}{z-w}
\e^{2m(z-w) \theta(z,w)}\diff A(w),\qquad z\in\D(z_0,\tfrac13 r).
\label{reproeq1}
\end{multline}
It will be enough to show that the last term on the right-hand side 
of \eqref{reproeq1} (the one which involves $\bar \partial_w \chi$) 
belongs to the error term. 
By \eqref{Qtaylor2} and the Cauchy-Schwarz inequality, we have that
\begin{multline}
\frac{1}{\pi}\int_{\Omega}\bigg|u(w) \frac{\bar\partial_w \chi_0(w)}{z-w}
\e^{2m(z-w)\theta}  \bigg|\diff A(w) 
\leq  
\int_{\Omega} \bigg| u(w) \frac{\bar\partial_w \chi_0(w)}{z-w}\bigg| 
\,\e^{mQ(z)- mQ(w)- \frac12 m \hDelta Q(z) |w-z|^2} \diff A(w) 
\\
\leq \|u\|_{m}\e^{mQ(z)} 
\bigg\{\int_{\Omega}\bigg|\frac{\bar\partial_w\chi_0(w)}{z-w} 
\bigg|^2 
\e^{-m\hDelta Q(z)|w-z|^2} \diff A(w) \bigg\}^{1/2} 
\leq \frac{3}{r}\|\bar\partial\chi_0\|_{L^2(\Omega)} \|u\|_{m}
\e^{mQ(z)}\e^{-\frac{1}{18}m r^2\hDelta Q(z)}, 
\end{multline}
again for $z\in\D(z_0,\frac13 r)$, and the desired conclusion is immediate.
\end{proof}

We will interpret Proposition \ref{negliprop} as saying that the function 
$M_m(z,w)$ is an local reproducing kernel $\text{mod}(\e^{-\delta m})$, with
$\delta=\delta_0>0$.
More generally, a function $L_m(z,w)$ defined on $\D(z_0,r)$ and holomorphic
in $z$ is a \emph{a local reproducing kernel} $\text{mod}(\e^{-\delta m})$ if
\begin{equation*} 
u(z)=\int_{\Omega}u(w)\chi_0(w)L_m(z,w)\,\e^{-2mQ(w)}\diff A(w) 
+\Ordo(\| u \|_{m} \e^{mQ(z)-m\delta}),
\qquad z\in \mathbb{D}(z_0,\tfrac13 r),
\end{equation*}
holds for large $m$ and all $u\in\mathrm{A}^2_m$. Here and in the sequel, 
it is implicit that $\delta$ should be positive. Analogously, 
a function $L_m(z,w)$ defined on $\D(z_0,r)$ and holomorphic
in $z$ is a \emph{a local reproducing kernel} $\text{mod}(m^{-k})$ if
\begin{equation*} 
u(z)=\int_{\Omega}u(w)\chi_0(w)L_m(z,w)\,\e^{-2mQ(w)}\diff A(w) 
+\Ordo(\| u \|_{m} m^{-k}\e^{mQ(z)}),
\qquad z\in \mathbb{D}(z_0,\tfrac13 r),
\end{equation*}
holds for large $m$ and all $u\in\mathrm{A}^2_m$.
We speak of \emph{ approximate local reproducing kernels} when it is implicit
 which of the above senses applies. 
The Bergman kernel $K_m(z,w)$ is of course an approximate local 
reproducing kernel in the above sense (no error term!), and it has the 
additional property of being holomorphic in $\bar w$. This suggests the
term \emph{approximate local Bergman kernel} for an approximate 
local reproducing kernel, which is holomorphic in $\bar w$. 
Starting with the local reproducing kernel $M_m$, Berman, Berndtsson, and 
Sj\"ostrand supply an algorithm to correct $M_m$ and obtain as a result an
approximate local Bergman kernel $\text{mod} (m^{-k})$ for any given positive 
integer $k$. 
It is a nontrivial step -- which actually requires additional assumptions on 
the potential $Q$ -- to show that the approximate local Bergman kernels 
obtained 
algorithmically in this fashion are indeed close to the Bergman kernel $K_m$ 
near the diagonal. This is achieved by an argument based on H\"ormander's 
$L^2$-estimates for the $\bar\partial$-operator. In this section, we only 
present the corrective algorithm. 

\subsection{A differential operator and negligible amplitudes}

The corrective algorithm involv\-es the differential operator
\begin{equation} 
\bslashnabla := 
\frac{1}{\bar\partial_w\theta} \bar \partial_w + 2m\Mop_{z-w}
\label{nabladef}
\end{equation}
and a (formal) diffusion operator $\Sop$ which will be described in detail
later on. In \eqref{nabladef} and more generally in the sequel, $\Mop$ with a 
subscript stands for the operator of multiplication by the function in 
the subscript.
The differential operator $\bslashnabla$ has the property that for 
smooth functions $A$ on $\D(z_0,r)\times\D(z_0,r)$,
\begin{equation}
\frac{1}{\bar\partial_w\theta}
\bar\partial_w \big\{ A(z,w) \,\e^{2m(z-w) \theta(z,w)} \big\}
=\e^{2m(z-w) \theta(z,w)}\bslashnabla A(z,w),
\label{eq-fundid0}
\end{equation}
which we may express in the more abstract (intertwining) form
\begin{equation}
\Mop_{\bar\partial_w\theta}
\bslashnabla=\Mop_{\e^{-2m(z-w)\theta}}\bar\partial_w \Mop_{\e^{2m(z-w)\theta}}.
\label{eq-fundid0.1}
\end{equation}
It follows from \eqref{eq-fundid0} that if $u$ is a holomorphic function 
on $\D(z_0,r)$, then
\begin{multline} 
\label{nablanegli}
\int_{\Omega} u(w) \chi_0(w) [\bar\partial_w \theta(z,w)] 
[\bslashnabla A(z,w)]\, \e^{2m(z-w) \theta(z,w)} \diff A(w) 
\\
= \int_{\Omega} u(w) \chi_0(w) \bar \partial_w 
\big\{ A(z,w) \,\e^{2m(z-w)\theta(z,w)} \big\} \diff A(w) 
=\int_{\Omega} u(w) A(z,w) \,\e^{2m(z-w)\theta(z,w)} \dbar\chi_0(w)\diff A(w) 
\end{multline}
which we may estimate using \eqref{Qtaylor2}:
\begin{multline} 
\label{nablanegli1.1}
\bigg|\int_{\Omega} u(w) \chi_0(w) [\bar\partial_w \theta(z,w)] 
[\bslashnabla A(z,w)]\, \e^{2m(z-w) \theta(z,w)} \diff A(w) \bigg|
\\
\le
\e^{mQ(z)-\delta_0m}\int_{\Omega} |u(w) A(z,w)| \,\e^{-mQ(w)}|\dbar\chi_0(w)|\diff A(w) 
\\
\le\e^{mQ(z)-\delta_0m}\|A\|_{L^\infty(\D(z_0,\frac34r)^2)}\|u\|_m
\|\bar\partial\chi_0\|_{L^2(\Omega)} =
\Ordo\big(\e^{mQ(z)-\delta_0m}\|A\|_{L^\infty(\D(z_0,\frac34r)^2)}\|u\|_m\big),
\qquad z\in\D(z_0,\tfrac13r),
\end{multline}
where the implied constant is absolute. Here, 
$\delta_0=\frac{1}{18}r^2\epsilon_0$ as before, and the norm of $A$ is the 
supremum norm on the bidisk $\D(z_0,\tfrac34r)^2=
\D(z_0,\tfrac34r)\times\D(z_0,\tfrac34r)$. 
Compared with the typical size  $\Ordo(\e^{mQ(z)})$, this means that we have 
exponential decay in $m$.
For this reason functions of the form $\Mop_{\bar\partial_w\theta}
\bslashnabla A$ are called \emph{negligible amplitudes}. 

\subsection{A formal diffusion operator and the characterization of 
negligible amplitudes}

Next, to define the diffusion operator $\Sop$, we need the two 
differential operators 
\begin{equation}
{\dd}_w=\partial_w-\frac{\partial_w\theta}{\dbar_{w}\theta}
\dbar_{w},\qquad
\dd_\theta=
\frac{1}{\dbar_{w}\theta}\dbar_{w}.
\label{eq-coordch1}
\end{equation} 
These operators come from the change of variables 
$(z,w, \bar w) \to (z, w, \theta)$, but this is of no real significance to
us here. What is, however, important is the property that they commute 
($\dd_w\dd_\theta=\dd_\theta\dd_w$), and the formula for the commutator of 
$\dd_w$ and $\Mop_{z-w}$ (see \eqref{eq-commutator1} below).
In terms of the above differential operators, $\bslashnabla$ simplifies:
\begin{equation}
\bslashnabla=\dd_\theta+2m\Mop_{z-w}.
\label{nabladef1.1}
\end{equation}
The diffusion operator $\Sop$ is defined (rather formally) by 
\begin{equation}
 \Sop := \e^{(2m)^{-1} \dd_w \dd_\theta} = \sum_{j=0}^{\infty} 
\frac{1}{j!(2m)^j}(\dd_w\dd_\theta)^j .  
\label{Sdef}
\end{equation}
The relation with diffusion comes from thinking about $\dd_w \dd_\theta$
as a generalized Laplacian, so that $\Sop$ is the $1/(2m)$ forward time step
for the corresponding generalized diffusion (or heat) equation. Our analysis
will not require the series expansion \eqref{Sdef} to converge in any 
meaningful way. It is supposed to act on asymptotic expansions of the type
\begin{equation}
a(z,w) \sim m a_0(z,w) + a_1(z,w) + m^{-1} a_{2}(z,w)+ m^{-2} a_{3}(z,w)+\ldots,
\label{eq-asseries1}
\end{equation}
in the obvious fashion. Here, the functions $a_j(z,w)$ are assumed not to 
depend on the parameter $m$. As for the meaning of an asymptotic expansion like 
\eqref{eq-asseries1}, we should require that
\[ 
a(z,w)=\sum_{j=0}^{k} m^{-j+1} a_j(z,w)+\Ordo(m^{-k}). 
\]
Actually, we may even think that the left hand side represents the right-hand
side more abstractly in the sense of the \emph{abschnitts} (partial sums)
\[
a^{\langle k\rangle}(z,w)=\sum_{j=0}^{k} m^{-j+1} a_j(z,w),
\]
so that $a(z,w)$ need not be a well-defined function, it would just be a 
stand-in for the sequence of abschnitts $a^{\langle k\rangle}(z,w)$, 
$k=1,2,3,\ldots$. In this sense, the meaning of $\Sop a$ clarifies completely.
We observe from the way $\Sop$ is defined that
\begin{equation}
[\Sop a^{\langle k\rangle}]^{\langle k\rangle}=
[\Sop a]^{\langle k\rangle}.
\label{eq-abschn1}
\end{equation}
and likewise for its inverse
\begin{equation}
[\Sop^{-1} a^{\langle k\rangle}]^{\langle k\rangle}=
[\Sop^{-1} a]^{\langle k\rangle}.
\label{eq-abschn2}
\end{equation}
Here, the inverse operator $\Sop^{-1}$ is given by the (rather formal) expansion
\begin{equation}
 \Sop^{-1}:=\e^{-(2m)^{-1} \dd_w \dd_\theta} =\sum_{j=0}^{\infty} 
\frac{(-1)^j}{j!(2m)^j}(\dd_w\dd_\theta)^j .  
\label{Sdefinv}
\end{equation}

The important property of $\Sop$, which is verified by a simple algebraic 
computation, is that 
\begin{equation}
 \Sop \bslashnabla = 2m\Mop_{z-w}\Sop. 
\label{Snabla}
\end{equation}
We include the calculation which gives \eqref{Snabla}, as it is elegant and
quite simple. First, we note that
\[
\dd_w^j\Mop_{z-w}\dd_{\theta}^j
=-j\,\dd_\theta(\dd_\theta\dd_w)^{j-1}+\Mop_{z-w}(\dd_\theta\dd_w)^j,
\qquad j=1,2,3,\ldots,
\]
which follows from the calculation
\begin{equation}
\dd_w^j\Mop_{z-w}=-j\,\dd_w^{j-1}+\Mop_{z-w}\dd_w^j,\qquad
j=1,2,3,\ldots.
\label{eq-commutator1}
\end{equation}
Finally, we expand $\Sop\bslashnabla$:
\begin{multline*}
\Sop\bslashnabla=\e^{(2m)^{-1}\dd_\theta\dd_w}\bslashnabla
=\e^{(2m)^{-1}\,\dd_\theta\dd_w}(\dd_{\theta}+2m\Mop_{z-w})\\
=\sum_{j=0}^{+\infty}\frac{1}{j!(2m)^j}\,\dd_w^j\dd_\theta^j(\dd_{\theta}
+2m\Mop_{z-w})
=\dd_{\theta}\,\e^{(2m)^{-1}\dd_\theta\dd_w}+
\sum_{j=0}^{+\infty}\frac{1}{j!(2m)^{j-1}}\,
\dd_w^j\,\Mop_{z-w}\,\dd_{\theta}^j
\\
=\dd_\theta\,\e^{(2m)^{-1}\dd_\theta\dd_w}
-\sum_{j=1}^{+\infty}\frac{1}{(j-1)!(2m)^{j-1}}\,
\dd_\theta(\dd_\theta\dd_w)^{j-1}
+\sum_{j=1}^{+\infty}\frac{1}{j!(2m)^{j-1}}\,\Mop_{z-w}(\dd_\theta\dd_w)^{j}
\\
=2m\Mop_{z-w}\,\e^{m^{-1}\dd_\theta\dd_w}=2m\Mop_{z-w}\Sop,
\end{multline*}
as needed.

The reason why we consider the diffusion operator $\Sop$ is outlined in
the following proposition (see \cite{bbs}). As for notation, let
$\calR_1$ denote the algebra of all functions $f(z,w)$ that are holomorphic
in $z$ and $C^\infty$-smooth in $w$ (near the diagonal $z=w$). 


\begin{prop}
Fix an integer $k=1,2,3\ldots$. 
Suppose $a(z,w)$ has the asymptotic expansion \eqref{eq-asseries1}, i.e.,
$a\sim m a_0 + a_1 +m^{-1}a_{2}+m^{-2}a_3+\ldots$, where the 
functions $a_j\in\calR_1$ are all independent of $m$. Then the following are 
equivalent:

\noindent{\rm(i)} We have that
\[ 
[\Sop a]^{\langle k\rangle} \in\Mop_{z-w}\calR_1.  
\]
\noindent{\rm(ii)} We have that
\[ 
a^{\langle k\rangle} = [\bslashnabla A]^{\langle k\rangle},  
\]
for some function $A(z,w) = \sum_{j=0}^{k} m^{-j} A_j(z,w)$, where each 
$a_j\in\calR_1$ is independent of $m$. 
\label{prop-neglampl1}
\end{prop}

\begin{proof}
Suppose first that (ii) holds. Then, by \eqref{eq-abschn1} and \eqref{Snabla},
we have that
\begin{equation}
[\Sop a]^{\langle k\rangle}=[\Sop a^{\langle k\rangle}]^{\langle k\rangle}
= [\Sop[\bslashnabla A]^{\langle k\rangle}]^{\langle k\rangle}
=[\Sop\bslashnabla A]^{\langle k\rangle}=\Mop_{z-w}[2m\Sop A]
^{\langle k\rangle}=2m\Mop_{z-w}[\Sop A]^{\langle k+1\rangle},  
\label{eq-calc1}
\end{equation}
and (i) follows. 

Next, we suppose instead that (i) holds, so that $[\Sop a]^{\langle k\rangle}=
\Mop_{z-w}\alpha$ for some function $\alpha(z,w)=\sum_{j=0}^{k}m^{1-j}
\alpha_j(z,w)$, where the $\alpha_j(z,w)$ are independent of $m$, and 
holomorphic in $z$ and $C^\infty$-smooth in $w$ near the diagonal $z=w$. 
We would like to find an $A$ of the form prescribed by (ii). 
In view of the calculation \eqref{eq-calc1}, it is clear that if $A$ can
be found, then we must have $\alpha=[2m\Sop A]^{\langle k\rangle}=
2m[\Sop A]^{\langle k+1\rangle}$. Finally, by \eqref{eq-abschn2}, we realize that 
$A=(2m)^{-1}[\Sop^{-1}\alpha]^{\langle k\rangle}$. Finally, we plug in 
$A:=(2m)^{-1}[\Sop^{-1}\alpha]^{\langle k\rangle}$ and check that it is of the right
form, with $[\bslashnabla A]^{\langle k\rangle}=a^{\langle k\rangle}$. 
Indeed, we see by applying $\Sop^{-1}$ from the left and the right on both
sides of \eqref{Snabla} using \eqref{eq-abschn1} and \eqref{eq-abschn2} that
\begin{equation}
\bslashnabla\Sop^{-1} = 2m\Sop^{-1}\Mop_{z-w}, 
\label{Snabla3}
\end{equation}
so that
\[
[\bslashnabla A]^{\langle k\rangle}=
[(2m)^{-1}\bslashnabla\Sop^{-1}\alpha]^{\langle k\rangle}
=[\Sop^{-1}\Mop_{z-w}\alpha]^{\langle k\rangle}=[\Sop^{-1}\Sop a]^{\langle k\rangle}
=a^{\langle k\rangle},
\]
as needed.
\end{proof}

\subsection{The local asymptotics for the weighted Bergman kernel:
the corrective algorithm}
\label{subsec-5.6}

We recall from Subsection \ref{subsec-5.3} that the kernel $M_m(z,w)$ given by
\eqref{eq-Mm} is a local reproducing kernel $\mathrm{mod}(\e^{-\delta_0 m})$ 
(see Proposition \ref{negliprop}).
The kernel $M_m(z,w)$ is automatically holomorphic in $z\in\D(z_0,r)$, and
we would like to correct it so that it becomes conjugate-holomorphic in $w$,
while maintaining the approximate reproducing property.
In view of the negligible amplitude calculation \eqref{nablanegli},
the approximate reproducing property (with a worse precision) will be kept 
if we replace $M_m$ by the kernel
\[
K^{\langle k\rangle}_{m}(z,w):=\Lfun^{\langle k\rangle}(z,w)\,\e^{2mQ(z,w)},
\]
where 
\begin{equation}
\Lfun^{\langle k\rangle}:=m\Lfun_0+\Lfun_1+\cdots+m^{-k+1}
\Lfun_k
\label{eq-Lfunexp1}
\end{equation}
is a finite asymptotic expansion (where each term $\Lfun_j$ is independent
of $m$), with
\begin{equation} 
\Lfun^{\langle k\rangle}=
\frac{2m}{\pi}\bar\partial_w \theta + \bar\partial_w\theta \, 
[\bslashnabla X]^{\langle k\rangle}.  
\label{negliampeq}
\end{equation}  
Here, $X\sim X_0+m^{-1}X_1+m^{-2}X_2+\ldots$ is an asymptotic expansion in $m$
(where every $X_j$ is independent of $m$), and each $X_j(z,w)$ should be 
holomorphic in $z$ and $C^\infty$-smooth in $w$. Then $K^{\langle k\rangle}_{m}(z,w)$
is a local reproducing kernel $\text{mod}(m^{-k})$ as a result of 
perturbing by the abschnitt of a negligible amplitude (multiplied by 
$\e^{2mQ(z,w)}$). We want to add the
condition that $\Lfun^{\langle k\rangle}(z,w)$ be conjugate-holomorphic 
in $w$, to obtain a local Bergman kernel $\text{mod}(m^{-k})$. This amounts
to asking that $\Lfun_j(z,w)$ is holomorphic in $\bar w$. The way this will 
enter into the algorithm is that such functions are uniquely determined by 
their diagonal restrictions. 

It is convenient to express \eqref{negliampeq} is the form 
\begin{equation} 
\frac{\Lfun^{\langle k\rangle}}{\bar\partial_w \theta}=
\frac{2m}{\pi}+[\bslashnabla X]^{\langle k\rangle}.  
\label{negliampeq1.1}
\end{equation} 
For convenience of notation, we think of $\Lfun^{\langle k\rangle}$ as the
abschnitt of an asymptotic expansion $\Lfun\sim m\Lfun_0+\Lfun_1+
m^{-1}\Lfun_2+\ldots$. Proposition \ref{prop-neglampl1} tells us how to solve
\eqref{negliampeq1.1}: we just apply the diffusion operator $\Sop$ 
to both sides and check that the abschnitts coincide along the diagonal $z=w$. 
More concretely, $\Lfun^{\langle k\rangle}$ solves \eqref{negliampeq1.1} for some 
$X$ if and only if
\begin{equation*} 
\bigg\{\Sop\bigg[\frac{\Lfun}{\bar\partial_w \theta}
\bigg]-\frac{2m}{\pi}\bigg\}^{\langle k\rangle}\in\Mop_{z-w}\calR_1,  
\end{equation*}
which is the same as
\begin{equation} 
\bigg\{\Sop\bigg[\frac{m^{-1}\Lfun}{\bar\partial_w \theta}
\bigg]\bigg\}^{\langle k+1\rangle}-\frac{2}{\pi}\in\Mop_{z-w}\calR_1.  
\label{Scond}
\end{equation}
When we expand the condition \eqref{Scond}, we find that it is equivalent 
to having
\begin{equation}
\frac{\Lfun_0}{\bar\partial_w \theta}-\frac{2}{\pi}\in\Mop_{z-w}\calR_1,
\label{Scond1}
\end{equation}
and 
\begin{equation}
\sum_{i=0}^{j}\frac{1}{i!2^i}(\dd_w\dd_\theta)^i
\Bigg\{\frac{\Lfun_{j-i}}{\bar\partial_w \theta}\Bigg\}\in\Mop_{z-w}\calR_1,
\qquad j=1,\ldots,k.
\label{Scond2}
\end{equation}
From the Taylor expansion \eqref{eq-4.19.1}, we see that
\[
\bar\partial_w\theta(z,w)-\partial_z\bar\partial_w Q(z,w)\in\Mop_{z-w}
\calR_1,
\]
so that \eqref{Scond1} is equivalent to 
\[
\Lfun_0(z,w)-\frac{2}{\pi}\partial_z\bar\partial_w Q(z,w)\in\Mop_{z-w}\calR_1.
\]
Taking the restriction to the diagonal, we obtain that 
$\Lfun_0(z,z)=\frac{2}{\pi}\hDelta Q(z)$, and as the diagonal $z=w$ is a 
uniqueness set for functions that are holomorphic in $(z,\bar w)$, the only 
possible choice is 
\[
\Lfun_0(z,w):=\frac{2}{\pi}\partial_z\bar\partial_w Q(z,w).
\] 
The rest of the functions $\Lfun_j$ are obtained in explicit form using 
\eqref{Scond2}. To illustrate how this works, we show how to obtain $\Lfun_1$.
By \eqref{Scond2} with $j=1$, we have
\begin{equation}
\frac{\Lfun_{1}}{\bar\partial_w\theta}+\frac12\dd_w\dd_\theta
\Bigg\{\frac{\Lfun_{0}}{\bar\partial_w \theta}\Bigg\}\in\Mop_{z-w}\calR_1.
\label{eq-5.29}
\end{equation}
We plug in the expression for $\Lfun_0$ which we obtained above, and 
restrict \eqref{eq-5.29} to the diagonal $z=w$:
\begin{equation*}
\frac{\Lfun_{1}(z,z)}{\hDelta Q(z,z)}=-\frac12\dd_w\dd_\theta
\Bigg\{\frac{\frac{2}{\pi}\partial_z\bar\partial_w Q(z,w)}
{\bar\partial_w \theta}
\Bigg\}\Bigg|_{w:=z}.
\end{equation*}
After some necessary simplifications, this leads to the formula
\[
\Lfun_{0}(z,z)=\frac{1}{2\pi}\hDelta\log\hDelta Q(z),
\]
and the only possible choice for $\Lfun_{0}(z,w)$ which is holomorphic in
$(z,\bar w)$ is the polarization of the above,
\[
\Lfun_{0}(z,w)=\frac1{2\pi}
\partial_z\bar\partial_w\log[\partial_z\bar\partial_w 
Q(z,w)].
\]
We note that the expressions for $\Lfun_{j}(z,w)$ are well-defined and
holomorphic in $(z,\bar w)$ on the bidisk $\D(z_0,r)\times\D(z_0,r)$,
at least for $j=0$ and $j=1$. This will be the case for all other indices $j$
as well, because the expressions for $\Lfun^{1}_{j}(z,w)$ will only involve
polynomial expressions in some derivatives of 
$\partial_z\bar\partial_wQ(z,w)$  possibly divided by a positive integer 
power of $\partial_z\bar\partial_wQ(z,w)$. 
This allows us to work with the local Bergman kernel $K^{\langle k\rangle}(z,w)$
in the context of the bidisk $\D(z_0,r)\times\D(z_0,r)$.

\section{The kernel expansion algorithm for bianalytic functions}
\label{sec-kernelexp}

\subsection{The local ring and module}
We recall that $\calR_1$ is the 
algebra of all functions $f(z,w)$ that are holomorphic in $z$ and 
$C^\infty$-smooth in $w$ (near the diagonal $z=w$). We will now need also
the $\calR_1$-module $\calR_q$ of all functions $f(z,w)$ that are 
$q$-analytic in $z$ and $C^\infty$-smooth in $w$ (near the diagonal $z=w$).
Here, we focus the attention to the case $q=2$. 

\subsection{The initial approximate local bi-reproducing kernel}
\label{subsec-6.2}
We now turn our attention to the space $\mathrm{PA}^2_{q,m}$ with $q=2$
(the bianalytic case). The first step is to find the analogue in this setting
of the kernel $M_m$ given by \eqref{eq-Mm}, and then to obtain an 
approximate reproducing identity similar to \eqref{reproeq1}. 

We keep the assumptions on the potential $Q$ and the disk $\D(z_0,r)$ 
from Subsection \ref{subsec-assQ}, and the smooth cut-off function $\chi_0$ 
will be as in Subsection \ref{subsec-5.3}. 
We will need to add the requirement that 
$r\|\bar\partial^2\chi_0\|_{L^2(\Omega)}$ is uniformly bounded by an 
\emph{absolute} constant, which is possible to achieve.


Suppose $u$ is \emph{biholomorphic} (or \emph{bianalytic}) in $\D(z_0,r)$, 
which means that $\bar\partial^2 u=0$ there. It is well-known that the 
fundamental solution to $\bar\partial_z^2$ is $\bar z/(\pi z)$.
By integration by parts applied twice, we have that
\begin{multline} 
\int_{\Omega} u(w) \chi_0(w)\bar\partial_w^2 \bigg\{\frac{\bar w-\bar z}{w-z}
\e^{2m(z-w)\theta} \bigg\}  \diff A(w) 
=\int_{\Omega}\frac{\bar w - \bar z}{w-z}
\e^{2m(z-w)\theta}\bar\partial_w^2\{u(w)\chi_0(w)\}\diff A(w) 
\\
=\int_{\Omega}\frac{\bar w-\bar z}{w-z}
\e^{2m(z-w)\theta}\big\{2\bar\partial_w u(w)\bar\partial_w\chi_0(w)+
 u(w)\bar\partial_w^2\chi_0(w)\big\}\diff A(w)
\\
=\int_{\Omega}
\e^{2m(z-w)\theta}u(w)\bigg\{2\frac{\bar\partial_w\chi_0(w)}{z-w}
-\frac{\bar w-\bar z}{w-z}\bar\partial_w^2\chi_0(w)
-4m(\bar z-\bar w)\bar\partial_w\chi_0(w)\bar\partial_w\theta\bigg\}\diff A(w). 
\label{eq-6.100.1}
\end{multline}
Here, the integral on the left-hand side needs to be understood in the sense
of distribution theory.
If we combine \eqref{eq-6.100.1} with \eqref{Qtaylor2}, it follows that 
we obtain the estimate
\begin{multline*} 
\bigg|\int_{\Omega} u(w) \chi_0(w)\bar\partial_w^2 
\bigg\{\frac{\bar w-\bar z}{w-z}\e^{2m(z-w)\theta} \bigg\}\diff A(w)\bigg| 
\\
\le\e^{mQ(z)-\frac{1}{18}mr^2\Delta Q(z)}\int_{\Omega}|u(w)|
\bigg\{2\bigg|\frac{\bar\partial_w\chi_0(w)}{z-w}\bigg|
+|\bar\partial_w^2\chi_0(w)|
+4m\big|(\bar z-\bar w)\bar\partial_w\chi_0(w)\bar\partial_w\theta\big|\bigg\}
\e^{-mQ(w)}\diff A(w),
\end{multline*}
for $z\in\D(z_0,\tfrac13r)$. 
Next, if we assume that $u\in\mathrm{PA}^2_{q,m}$, with $q=2$, so that -- in 
particular -- $u$ extends biholomorphically to $\Omega$, then 
\begin{multline} 
\bigg|\int_{\Omega} u(w) \chi_0(w)\bar\partial_w^2 
\bigg\{\frac{\bar w-\bar z}{w-z}\e^{2m(z-w)\theta} \bigg\}\diff A(w)\bigg| 
\\
\le
r^{-1}\e^{mQ(z)-\frac{1}{18}mr^2\epsilon_0}\|u\|_m\bigg\{
6\|\bar\partial\chi_0\|_ {L^2(\Omega)}
+r\|\bar\partial^2\chi_0\|_ {L^2(\Omega)}
+5mr^2\|\bar\partial\chi_0\|_ {L^2(\Omega)}
\|\bar\partial_w\theta\|_{L^\infty(\D(z_0,\frac34r)^2)}\bigg\},
\\
=\Ordo\big(r^{-1}\e^{mQ(z)-\delta_0m}\|u\|_m
\big\{1+mr^{2}\|\bar\partial_w\theta\|_{L^\infty(\D(z_0,\frac34r)^2)}\big\}\big), 
\label{eq-6.101}
\end{multline}
for $z\in\D(z_0,\tfrac13r)$, where we write as before 
$\delta_0=\frac{1}{18}r^2\epsilon_0$, and the implied constant is 
\emph{absolute}. 
We interpret \eqref{eq-6.101} as saying that the left-hand side decays 
exponentially small compared with the typical size $\e^{mQ(z)}$. 
Let $\boldsymbol{\delta}_0$ denote the distribution which corresponds to a 
unit point mass at $0$ in the complex plane $\C$. By a direct calculation, 
we find that
\begin{equation}
\frac{1}{\pi}\bar\partial_w^2 \bigg\{\frac{\bar w-\bar z}{w-z}
\e^{2m(z-w)\theta} \bigg\}=\boldsymbol{\delta}_0(z-w)-M_{2,m}(z,w)\e^{-2mQ(w)},
\label{eq-6.103}
\end{equation}
where $M_{2,m}$ is the kernel
\begin{equation}
M_{2,m}(z,w):=
\bigg\{
\frac{4m}{\pi}\bar\partial_w\theta - \frac{2m}{\pi}
(\bar z-\bar w)\bar\partial_w^2 \theta - \frac{4m^2}{\pi}
|z-w|^2 (\bar \partial_w \theta)^2\bigg\}\,\e^{2mQ(z,w)}.
\label{eq-6.104}
\end{equation}
The kernel $M_{2,m}$ is the bianalytic analogue of the kernel $M_m$ which was
our staring point in Subsection \ref{subsec-5.3}. 
It is also automatically bianalytic in $z$, but not necessarily in $\bar w$. 

\begin{prop} 
\label{neglipropbi}
We have that for all $u\in\mathrm{PA}^2_{2,m}$, 
\begin{equation*} 
u(z)=\int_{\Omega}u(w)\chi_0(w)M_{2,m}(z,w)\,\e^{-2mQ(w)}\diff A(w) 
+\Ordo\big(r^{-1}\e^{mQ(z)-\delta_0m}\|u\|_m
\big\{1+mr^{2}\|\bar\partial_w\theta\|_{L^\infty(\D(z_0,\frac34r)^2)}\big\}\big),
\end{equation*}
for $z\in \mathbb{D}(z_0,\tfrac13 r)$, where 
$\delta_0=\frac{1}{18}r^2\epsilon_0>0$. The implied constant is absolute.
\end{prop}

\begin{proof}
This is an immediate consequence of the relation \eqref{eq-6.103} and of
the estimate \eqref{eq-6.101}. 
\end{proof}

\subsection{Approximate local kernels}
We will interpret Proposition \ref{neglipropbi} as saying that the function 
$M_{2,m}(z,w)$ is a local bi-reproducing kernel $\text{mod}(m\e^{-\delta m})$, with
$\delta=\delta_0>0$.
Generally, we say that a function $L_m(z,w)$ defined on $\D(z_0,r)$ and 
biholomorphic
in $z$ is a \emph{a local bi-reproducing kernel} $\text{mod}(\e^{-\delta m})$ if
\begin{equation*} 
u(z)=\int_{\Omega}u(w)\chi_0(w)L_m(z,w)\,\e^{-2mQ(w)}\diff A(w) 
+\Ordo(\|u\|_{m} \e^{mQ(z)-m\delta}),
\qquad z\in \mathbb{D}(z_0,\tfrac13 r),
\end{equation*}
holds for large $m$ and all $u\in\mathrm{PA}^2_{2,m}$. Here and in the sequel, 
it is implicit that $\delta$ should be positive. Analogously, 
a function $L_m(z,w)$ defined on $\D(z_0,r)$ and biholomorphic
in $z$ is a \emph{a local bi-reproducing kernel} $\text{mod}(m^{-k})$ if
\begin{equation*} 
u(z)=\int_{\Omega}u(w)\chi_0(w)L_m(z,w)\,\e^{-2mQ(w)}\diff A(w) 
+\Ordo(\| u \|_{m} m^{-k}\e^{mQ(z)}),
\qquad z\in \mathbb{D}(z_0,\tfrac13 r),
\end{equation*}
holds for large $m$ and all $u\in\mathrm{PA}^2_m$.
We speak of \emph{ approximate local bi-reproducing kernels} when it is 
implicit which of the above senses applies. 
The bianalytic Bergman kernel $K_{2,m}(z,w)$ is of course an approximate local 
bi-reproducing kernel in the above sense (no error term!), and it has the 
additional property of being biholomorphic in $\bar w$. This suggests the
term \emph{approximate local bianalytic Bergman kernel} for an approximate 
local bi-reproducing kernel, which is biholomorphic in $\bar w$. 
It remains to describe the corrective algorithm which turns the approximate
local bi-reproducing kernel into an approximate local bianalytic Bergman kernel.

\subsection{Bi-negligible amplitudes}
We recall the definition of the differential operator $\bslashnabla$ in 
\eqref{nabladef}. By squaring the operator identity \eqref{eq-fundid0.1}, 
we arrive at 
\[
\big\{\Mop_{\bar \partial_w \theta} \bslashnabla 
\Mop_{\bar \partial_w \theta} \bslashnabla A(z,w) \big\}\, \e^{2m(z-w) \theta} 
=\bar \partial_w^2 \big\{ A(z,w) \e^{2m(z-w)\theta} \big\},
\]
provided that $A(z,w)$ depends $C^\infty$-smoothly on the pair $(z,w)$. 
As a consequence, we have by integration by parts (applied twice) that
\begin{multline} 
\label{bianalnegliampli}
\int_{\Omega} u(w) \chi_0(w) \big\{\Mop_{\bar \partial_w \theta} \bslashnabla 
\Mop_{\bar \partial_w \theta} \bslashnabla A(z,w) \big\}\, \e^{2m(z-w) \theta}\diff A(w)
\\
= \int_{\Omega} u(w) \chi_0(w) \bar \partial_w^2 \big\{A(z,w) \e^{2m(z-w) \theta} 
\big\} 
\diff A(w) 
=\int_{\Omega} A(z,w) \e^{2m(z-w) \theta} \bar\partial_w^2\{u(w)\chi_0(w)\} 
\diff A(w) 
\\
=\int_{\Omega} A(z,w) \e^{2m(z-w) \theta} \big\{2\bar\partial_w u(w)\,
\bar\partial_w\chi_0(w)+u(w)\bar\partial_w^2\chi_0(w)\big\} 
\diff A(w)
\\
=-\int_{\Omega} u(w) \e^{2m(z-w) \theta} \big\{2\bar\partial_w A(z,w)
\bar\partial_w\chi_0(w)
+A(z,w)[4m(z-w)\bar\partial_w\chi_0(w)\bar\partial_w\theta
+\bar\partial_w^2\chi_0(w)]\big\} 
\diff A(w),
\end{multline} 
provided $u$ is biholomorphic on $\D(z_0,r)$. If $u\in\mathrm{PA}^2_{2,m}$, 
we can now obtain the estimate (use \eqref{Qtaylor2})
\begin{multline}
\label{bianalnegliampli2}
\bigg|\int_{\Omega} u(w) \chi_0(w) 
\big\{\Mop_{\bar \partial_w \theta} \bslashnabla 
\Mop_{\bar \partial_w \theta} \bslashnabla A(z,w) \big\}\, 
\e^{2m(z-w) \theta}\diff A(w)\bigg|
\le\e^{mQ(z)-\delta_0m}
\\
\times
\int_{\Omega} |u(w)|\, \e^{-mQ(w)} \big\{2|\bar\partial_w A(z,w)
\bar\partial_w\chi_0(w)|+|A(z,w)|[5mr
|\bar\partial_w\chi_0(w)\bar\partial_w\theta|+|\bar\partial_w^2\chi_0(w)|]
\big\}\diff A(w)
\\
\le r^{-1}\e^{mQ(z)-\delta_0m}\|u\|_m\Big\{
2r\|\bar\partial_w A\|_{L^\infty(\D(z_0,\frac34r)^2)}\|\bar\partial\chi_0\|_{L^2(\Omega)}
+\|A\|_{L^\infty(\D(z_0,\frac34r)^2)}[r\|\bar\partial^2\chi_0\|_{L^2(\Omega)}
\\
+5mr^2\|\bar\partial_w\theta\|_{L^\infty(\D(z_0,\frac34r)^2)}
\|\bar\partial\chi_0\|_{L^2(\Omega)}]\Big\}
\\
=\Ordo\big(r^{-1}\e^{mQ(z)-\delta_0m}\|u\|_m
\big\{r\|\bar\partial_w A\|_{L^\infty(\D(z_0,\frac34r)^2)}
+\|A\|_{L^\infty(\D(z_0,\frac34r)^2)}
[1+mr^2\|\bar\partial_w\theta\|_{L^\infty(\D(z_0,\frac34r)^2)}]\big\}\big),
\end{multline}
for $z\in\D(0,\tfrac13r)$, where $\delta_0=\frac{1}{18}r^2\epsilon_0>0$, 
as before. 
We note that the implied constant in \eqref{bianalnegliampli2} is absolute.
If we allow the implied constant in ``O'' to depend on the triple $(A,z_0,r)$,
then the right-hand side of \eqref{bianalnegliampli2} may be condensed to 
$\Ordo(m\e^{mQ(z)-\delta_0m}\|u\|_m)$, for $m\ge1$.  

Compared with the typical size $\e^{mQ(z)}$, \eqref{bianalnegliampli2}
says that the estimated quantity decays exponentially in $m$. For this reason,
we call expressions of the form  $ \Mop_{\bar \partial_w \theta} \bslashnabla 
\Mop_{\bar \partial_w \theta} \bslashnabla A$ \emph{bi-negligible amplitudes}.

\subsection{The characterization of bi-negligible amplitudes}

We need to find the bianalytic analogue of Proposition \ref{prop-neglampl1},
so that we can tell when we have a bi-negligible amplitude.
To formulate the criterion, we first need to define the operator $\Nop$.
We note that a function $f\in\calR_2$ has a unique decomposition 
$f(z,w)=f_1(z,w)+\bar z f_2(z,w)$, where $f_1$ and $f_2$ are holomorphic in
in $z$ and $C^\infty$-smooth in $w$ (near the diagonal $z=w$). For such 
$f\in\calR_2$, we put
\[
 \Nop [f](z,w):=\frac{f_1(z,w)-f_1(w, w)}{z-w}+\bar z\,
\frac{f_2(z,w)-f_2(w, w)}{z-w}. 
\]
Then $\Nop \Mop_{z-w}$ is the identity operator.  
Second, we need the (formal) operator $\Sop'$ given by
\begin{equation}
\Sop' : = \Sop \Mop^{-1}_{\bar \partial_w \theta} \Sop^{-1} \Nop  \Sop.  
\label{eq-defS'}
\end{equation}

\begin{prop} 
\label{binegliprop}
Fix an integer $k=1,2,3,\ldots$.
Suppose $a$ has an asymptotic expansion $a\sim m^2 a_0+m a_1+a_2+m^{-1}a_3
+\ldots$, where the functions $a_j\in\calR_2$ are all independent of $m$.
Then the following are equivalent:
 
\noindent{\rm(i)}
There exist two finite asymptotic expansions 
$\alpha=m^2\alpha_0+m\alpha_1+\alpha_2+\cdots+
m^{-k+1}\alpha_{k+1}$ and $\alpha'=m^2\alpha'_0+m\alpha'_1+\alpha_2'+\cdots+
m^{-k+1}\alpha'_{k+1}$, where each $\alpha_j,\alpha_j'\in\calR_2$ is independent
of $m$, such that $[\Sop a]^{\langle k\rangle}=\Mop_{z-w}\alpha$
and $[\Sop' a]^{\langle k\rangle}=\Mop_{z-w}\alpha'(z,w)$.

\noindent{\rm(ii)}
There exists an asymptotic expansion $A\sim A_0+m^{-1}A_1+m^{-2}A_2+\ldots$, 
with $A_j\in\calR_2$ independent of $m$, such that 
\[ 
a^{\langle k\rangle}=[\bslashnabla\Mop_{\bar \partial_w \theta}\bslashnabla A]
^{\langle k\rangle}. 
\]
\end{prop}

\begin{proof}
We first show the implication $\mathrm{(i)}\Rightarrow\mathrm{(ii)}$. 
We recall the abschnitt properties \eqref{eq-abschn1} and \eqref{eq-abschn2},
which will be used repeatedly. We put $\beta:=[\Sop^{-1}\alpha]^{\langle k\rangle}$,
to obtain
\begin{equation}
[\Sop a]^{\langle k\rangle}=\Mop_{z-w}\alpha=[\Mop_{z-w}\Sop\beta]^{\langle k\rangle}, 
\label{bianaleq1}
\end{equation}
and, consequently, by \eqref{Snabla}, we have
\begin{equation}
a^{\langle k\rangle} = [\Sop^{-1}\Mop_{z-w}\Sop\beta]^{\langle k\rangle} 
= [(2m)^{-1}\bslashnabla\beta]^{\langle k\rangle}.
\label{bianaleq2}
\end{equation}
Moreover, we see from \eqref{bianaleq1} and the definition \eqref{eq-defS'}
of $\Sop'$ that
\begin{equation}
[\Sop' a]^{\langle k\rangle} = 
[\Sop\Mop^{-1}_{\bar \partial_w \theta} \Sop^{-1} \Nop  \Sop a]^{\langle k\rangle} 
=[\Sop\Mop^{-1}_{\bar\partial_w\theta}\Sop^{-1}\Nop\Mop_{z-w}\Sop\beta]^{\langle k\rangle}
=[\Sop\Mop^{-1}_{\bar\partial_w\theta}\beta]^{\langle k\rangle}, 
\label{eq-6.5.5}
\end{equation}
where we used that $\Nop\Mop_{z-w}$ is the identity. Next, we put
$\beta':=[\Sop^{-1}\alpha']^{\langle k\rangle}$, so that 
$\alpha'=[\Sop\beta']^{\langle k\rangle}$, so that in view of the second condition 
in (i) and \eqref{eq-6.5.5}, we get that
\begin{equation*}
\Mop_{z-w}[\Sop\beta']^{\langle k\rangle}=\Mop_{z-w}\alpha'=
[\Sop' a]^{\langle k\rangle}=
[\Sop\Mop^{-1}_{\bar\partial_w\theta}\beta]^{\langle k\rangle}.
\end{equation*}
We solve for $\beta$, using \eqref{Snabla}:
\begin{equation}
\beta=[\Mop_{\bar \partial_w \theta} \Sop^{-1} \Mop_{z-w} \Sop\beta']^{\langle k\rangle}
= [(2m)^{-1}\Mop_{\bar \partial_w \theta}\bslashnabla\beta']^{\langle k\rangle}.
\label{eq-6.5.6}
\end{equation}
By putting the relations \eqref{bianaleq2} and \eqref{eq-6.5.6} together, 
we now see that
\[
a^{\langle k\rangle}=[(2m)^{-2}\bslashnabla
\Mop_{\bar \partial_w \theta}\bslashnabla\beta']^{\langle k\rangle},
\]
which means that (ii) holds with $A^{\langle k+2\rangle}=(2m)^{-2}\beta'$. 

Finally, we turn to the reverse implication  
$\mathrm{(ii)}\Rightarrow\mathrm{(i)}$. This time we have $A$, and just need
to find $\alpha$ and $\alpha'$ with the given properties. Inspired by the
calculations we carried out for the forward implication, we put 
$\alpha':=[(2m)^2\Sop A]^{\langle k\rangle}$ and
$\alpha:=
[2m\Sop\Mop_{\bar \partial_w \theta}\bslashnabla A]^{\langle k\rangle}$.
It remains to verify that (i) holds with these choices $\alpha$ and 
$\alpha'$. We first check with the help of \eqref{Snabla} and (ii) that
\[
\Mop_{z-w}\alpha=
[2m\Mop_{z-w}\Sop\Mop_{\bar \partial_w \theta}\bslashnabla A]^{\langle k\rangle}
=[\Sop\bslashnabla\Mop_{\bar \partial_w \theta}\bslashnabla A]^{\langle k\rangle}
=[\Sop a]^{\langle k\rangle},
\]
so the first relation in (i) holds. Second, we check using (ii), 
\eqref{eq-defS'}, and \eqref{Snabla}, that
\begin{multline*}
[\Sop' a]^{\langle k\rangle}=
[\Sop'\bslashnabla\Mop_{\bar \partial_w \theta}\bslashnabla A]^{\langle k\rangle}
=[\Sop\Mop^{-1}_{\bar \partial_w \theta} \Sop^{-1} \Nop\Sop
\bslashnabla\Mop_{\bar \partial_w \theta}\bslashnabla A]^{\langle k\rangle}  
\\
=[2m\Sop\Mop^{-1}_{\bar \partial_w \theta} \Sop^{-1} \Nop
\Mop_{z-w}\Sop\Mop_{\bar \partial_w \theta}\bslashnabla A]^{\langle k\rangle}
=[2m\Sop\Mop^{-1}_{\bar \partial_w \theta}
\Mop_{\bar \partial_w \theta}\bslashnabla A]^{\langle k\rangle}
\\
=[2m\Sop\bslashnabla A]^{\langle k\rangle}
=[(2m)^2\Mop_{z-w}\Sop A]^{\langle k\rangle}=\Mop_{z-w}\alpha',
\end{multline*}
and the second relation in (i) holds as well. 
Note that in the above calculation we used that $\Nop\Mop_{z-w}$ is the identity.
The proof is complete.
\end{proof}

\begin{rem}
\label{rem-7.3}
(a) It is clear that a bi-negligible amplitude is automatically a negligible 
amplitude,
since $\bslashnabla\Mop^{-1}_{\bar \partial_w \theta}\bslashnabla A=
\bslashnabla B$, with $B:=\Mop^{-1}_{\bar \partial_w \theta}\bslashnabla A$. 
The first part of condition (i) of Proposition \ref{binegliprop}
captures this observation. But it is of course harder for amplitude
to be bi-negligible than to be negligible. The second part of condition (i),
which involves $\Sop'$, expresses that difference. 

\noindent{(b)} In the context of the proof of Proposition \ref{binegliprop},
the relation 
\[
A^{\langle k+2\rangle}=(2m)^{-2}\beta'=(2m)^{-2}[\Sop^{-1}\alpha']^{\langle k\rangle}
=(2m)^{-2}[\Sop^{-1}\Nop\Sop' a]^{\langle k\rangle}
\]
tells us how to obtain $A^{\langle k+2\rangle}$ starting from a given 
$a^{\langle k\rangle}$. 

\noindent{(c)} Proposition \ref{binegliprop} is stated for bi-negligible
amplitudes (i.e., where $q=2$). However, it is rather clear how to formulate 
the corresponding result for general $q$. In (i), we get $q$ operators 
$\Sop,\ldots,\Sop^{(q-1)}$ instead of $\Sop,\Sop'$, and in (ii), we get a
product of $q$ copies of the operator $\bslashnabla$ interlaced with $q-1$
copies of the multiplication operator $\Mop_{\bar\partial_w\theta}$.   
\end{rem}

\subsection{The local uniqueness criterion for bianalytic kernels}

For $F(z,w)$ holomorphic in $(z,\bar w)$ near the diagonal $z=w$ in $\C^2$,
it is well-known that the diagonal is a set of uniqueness. In particular,
the assumption $F\in\Mop_{z-w}\calR_1$ implies that $F(z,w)\equiv0$ holds
near the diagonal. We need the analogous criterion for functions biholomorphic
in each of the variables $(z,\bar w)$.  

\begin{lem}
Suppose $F\in\calR_2$, so that $F(z,w)$ is biholomorphic in $z$ and 
$C^\infty$-smooth in $w$ near the diagonal $z=w$. If 
$F(z,w)\in\Mop_{z-w}^2\calR_2$, and in addition, $F(z,w)$ is biholomorphic in
$\bar w$, then $F(z,w)\equiv0$ holds near the diagonal.     
\label{lem-uniq1}
\end{lem} 

\begin{proof}
We first decompose $F=F_0+\bar z F_1+wF_2+\bar z wF_3$, where each $F_j$ is 
holomorphic in $(z,\bar w)$ near the diagonal. Since $F$ is biholomorphic 
in $z$, the function $\bar\partial_z\partial_w F=F_3$ is holomorphic in 
$(z,\bar w)$, and the assumption $F\in\Mop_{z-w}^2\calR_2$ leads to 
$\bar\partial_z\partial_w F=F_3\in\Mop_{z-w}\calR_1$. But the diagonal is a 
set of uniqueness for $F_3$, and we find that $F_3(z,w)\equiv0$. So 
$F=F_0+\bar z F_1+wF_2$, and, consequently, $\bar\partial_z F=F_1\in
\Mop_{z-w}^2\calR_1$ is holomorphic in $(z,\bar w)$. Again the diagonal 
is a set of uniqueness for $F_1$, which implies that $F_1(z,w)\equiv0$.
What remains of the decomposition is now $F=F_0+wF_2$. From 
$F\in\Mop_{z-w}^2\calR_2$ we get that $\partial_w F=F_2\in\Mop_{z-w}\calR_2$,
and analogously we obtain that $F_2(z,w)\equiv0$. The final step, to show
that $F$ vanishes if $F=F_0\in\Mop_{z-w}^2\calR_2$ is holomorphic in 
$(z,\bar w)$, is trivial.    
\end{proof}

\begin{rem}
In the context of the lemma, it is not possible to weaken the assumption that 
$F\in\Mop_{z-w}^2\calR_2$ to just $F\in\Mop_{z-w}\calR_2$, as is clear from
the example $F(z,w):=z-w$.
\end{rem}

\subsection{The local asymptotics for the weighted bianalytic Bergman
kernel: the corrective algorithm I} 
\label{subsec-I}

The implementation of the corrective algorithm is rather analogous to the 
holomorphic situation of Subsection \ref{subsec-5.6}. However, since there
are certain differences, we explain everything in rather great detail. 
We recall from Subsection \ref{subsec-6.2} that the kernel $M_{2,m}$ given by
\eqref{eq-6.104} is a local bi-reproducing kernel $\mathrm{mod}(m\e^{-\delta_0m})$;
see Proposition \ref{neglipropbi}. If we correct $M_{2,m}$ by adding a 
bi-negligible amplitude, then it remains a local bi-reproducing kernel 
$\mathrm{mod}(m\e^{-\delta_0m})$, in view of \eqref{bianalnegliampli2}. 
That is, we replace $M_{2,m}$ by the kernel
\[
K^{\langle k\rangle}_{2,m}(z,w):=\Lfun^{\langle 2,k\rangle}(z,w)\e^{2mQ(z,w)},
\]
where 
\begin{equation}
\Lfun^{\langle 2,k\rangle}:=m^2\Lfun^{2}_0+m\Lfun^{2}_1+\cdots+m^{-k+1}
\Lfun^{2}_{k+1}
\label{eq-Lfunexp1.b}
\end{equation}
is a finite asymptotic expansion (where each term $\Lfun^{2}_j$ is independent
of $m$), with
\begin{equation} 
\Lfun^{\langle 2,k\rangle}=
- \frac{4m^2}{\pi}
|z-w|^2 (\bar \partial_w \theta)^2
+\frac{4m}{\pi}\bar\partial_w\theta - \frac{2m}{\pi}
(\bar z-\bar w)\bar\partial_w^2 \theta+\bar\partial_w\theta \, 
[\bslashnabla\Mop_{\bar\partial_w\theta}\bslashnabla X]^{\langle k\rangle}.  
\label{negliampeq.b}
\end{equation}  
Here, $X\sim X_0+m^{-1}X_1+m^{-2}X_2+\ldots$ is some asymptotic expansion in $m$
(where every $X_j$ is independent of $m$), where each $X_j(z,w)$ should be 
bi-holomorphic in $z$ and $C^\infty$-smooth in $w$. Then 
$K^{\langle k\rangle}_{2,m}(z,w)$ is a local bi-reproducing kernel 
$\text{mod}(m^{-k})$ as a result of 
perturbing by the abschnitt of a negligible amplitude (multiplied by 
$\e^{2mQ(z,w)}$). We want to add the
condition that $\Lfun^{\langle 2,k\rangle}(z,w)$ be bi-holomorphic 
in $\bar w$, to obtain a local weighted bianalytic Bergman kernel 
$\text{mod}(m^{-k})$. 
If we write 
\[
\Rfun^{(2)}:= -\frac{4m^2}{\pi}
|z-w|^2 (\bar \partial_w \theta)^2
+\frac{4m}{\pi}\bar \partial_w \theta-\frac{2m}{\pi}
(\bar z-\bar w)\bar\partial_w^2 \theta,
\]
so that the kernel $M_{2,m}$ in \eqref{eq-6.104} may be written as
\begin{equation}
\label{eq-RfunM}
M_{2,m}(z,w)=\Rfun^{(2)}(z,w)\,\e^{2mQ(z,w)},
\end{equation}
then the relation \eqref{negliampeq.b} may be expressed in condensed form:
\begin{equation} 
\bigg[\frac{\Lfun^{(2)}-\Rfun^{(2)}}{\bar\partial_w\theta}\bigg]
^{\langle k\rangle}=
[\bslashnabla\Mop_{\bar\partial_w\theta}\bslashnabla X]^{\langle k\rangle}.  
\label{negliampeq2.b}
\end{equation}
Here, $\Lfun^{(2)}\sim m^2\Lfun^{2}_0+m\Lfun^{2}_1+\Lfun^{2}_2+\ldots$ is the
asymptotic expansion with abschnitt $\Lfun^{\langle 2,k\rangle}$.
This now puts us in a position to apply Proposition \ref{binegliprop},
which says that \eqref{negliampeq2.b} holds for some $X$ if and only if
\begin{equation} 
\bigg\{
\Sop\bigg[\frac{\Lfun^{(2)}-\Rfun^{(2)}}{\bar\partial_w\theta}\bigg]
\bigg\}^{\langle k\rangle}\in\Mop_{z-w}\calR_2\quad\text{and}\quad
\Sop'\bigg[\frac{\Lfun^{(2)}-\Rfun^{(2)}}{\bar\partial_w\theta}\bigg]
\bigg\}^{\langle k\rangle}\in\Mop_{z-w}\calR_2.
\label{negliampeq3.b}
\end{equation}
Also, Remark \ref{rem-7.3} tells us how to recover the corresponding abschnitt
of $X$:
\begin{equation}
\label{eq-recoverX}
X^{\langle k+2\rangle}=(2m)^{-2}\bigg\{
\Sop^{-1}\Nop\Sop'\bigg[\frac{\Lfun^{(2)}-\Rfun^{(2)}}{\bar\partial_w\theta}\bigg]
\bigg\}^{\langle k\rangle}.
\end{equation}
This formula is valuable when we need to estimate the contribution from the
bi-negligible amplitude using \eqref{bianalnegliampli2}.
 
It remains to see how \eqref{negliampeq3.b} helps us determine the
asymptotic expansion $\Lfun^{(2)}$. A first remark is that it no longer
will be enough to determine $\Lfun^{(2)}(z,w)$ along the diagonal $z=w$, 
as functions which are biharmonic in each of the variables $(z,\bar w)$ need
not be uniquely determined by their diagonal restrictions (e.g., consider the
function $\bar z w-z\bar w$).

Before we carry on to analyze the relations \eqref{negliampeq3.b} further,
we recall the Taylor expansions \eqref{eq-4.19},  \eqref{eq-4.19.1}, and  
\eqref{eq-4.19.2}, and rewrite \eqref{eq-4.19.1} in the form
\begin{equation}
\dbar_w\theta=
\partial_z\dbar_w Q(z,w)
\bigg\{1+\sum_{j=1}^{+\infty}\frac{1}{(j+1)!}(w-z)^{j}
\frac{\partial_z^{j+1}\dbar_w Q(z,w)}{\partial_z\dbar_w Q(z,w)}\bigg\},
\label{eq-4.19.3}
\end{equation}
It follows from \eqref{eq-4.19.2} and \eqref{eq-4.19.3} that
\begin{equation}
\frac{\partial_w\theta}{\dbar_w\theta}-
\frac12\frac{\partial_z^2Q(z,w)}{\partial_z\dbar_w Q(z,w)}\in\Mop_{z-w}\calR_1.
\label{eq-4.19-ratio}
\end{equation}
A more refined version of \eqref{eq-4.19-ratio} is 
\begin{equation}
\frac{\partial_w\theta}{\dbar_w\theta}-
\frac12\frac{\partial_z^2Q(z,w)}{\partial_z\dbar_w Q(z,w)}
-(w-z)\bigg\{\frac13\frac{\partial_z^3Q(z,w)}{\partial_z\dbar_wQ(z,w)}
-\frac14\frac{[\partial_z^2Q(z,w)][\partial_z^2\dbar_wQ(z,w)]}
{[\partial_z\dbar_w Q(z,w)]^2}\bigg\} 
\in\Mop_{z-w}^2\calR_1.
\label{eq-4.19-ratio2}
\end{equation}
We now return to the corrective algorithm, which takes the criteria
\eqref{negliampeq3.b} 
as its starting point. We express $\Rfun^{(2)}$ in terms of the differential
operators $\dd_w$ and $\dd_\theta$, using that 
$\dd_\theta\bar w=1/\bar\partial_w\theta$: 
\begin{equation}
\frac{\Rfun^{(2)}}{\bar\partial_w\theta}=\Mop_{\dd_\theta \bar w}[\Rfun^{(2)}]
=-\frac{4m^2}{\pi}\frac{|z-w|^2}{\dd_\theta \bar w}
+\frac{4m}{\pi}-\frac{2m}{\pi}(\bar z-\bar w)\,
\dd_\theta\frac{1}{\dd_\theta\bar w}.
\label{eq-initial1.1}
\end{equation}
We first rewrite the criteria \eqref{negliampeq3.b} in the more appropriate 
form
\begin{equation} 
\big\{
\Sop\Mop_{\dd_\theta\bar w}[
\Lfun^{(2)}-\Rfun^{(2)}]
\big\}^{\langle k\rangle}\in\Mop_{z-w}\calR_2\quad\text{and}\quad
\big\{\Sop'\Mop_{\dd_\theta\bar w}[\Lfun^{(2)}-\Rfun^{(2)}]
\big\}^{\langle k\rangle}\in\Mop_{z-w}\calR_2.
\label{negliampeq4.b}
\end{equation}
In view of the asymptotic expansion for $\Lfun^{(2)}$ (see 
\eqref{eq-Lfunexp1.b}) and the definition \eqref{Sdef} of the operator 
$\Sop$, it follows that
\begin{equation}
\big\{\Sop\Mop_{\dd_\theta\bar w}[\Lfun^{(2)}]\big\}^{\langle k\rangle}
=\sum_{j=0}^{k+1}m^{2-j}\sum_{i=0}^{j}
\frac{2^{-i}}{i!}(\dd_w\dd_\theta)^i\Mop_{\dd_\theta\bar w}\Lfun^2_{j-i}.
\label{eq-firstcond1}
\end{equation}
We would also like to calculate the corresponding expression involving the 
operator $\Sop'$ in place of $\Sop$ (see \eqref{eq-defS'}). In a first step, 
we get that
\begin{equation*}
\big\{\Sop^{-1}\Nop\Sop\Mop_{\dd_\theta\bar w}[\Lfun^{(2)}]\big\}^{\langle k\rangle}
=\sum_{j=0}^{k+1}m^{2-j}\sum_{i_2=0}^{j}\sum_{i_1=0}^{i_2}
\frac{(-1)^{i_2-i_1}2^{-i_2}}{i_1!(i_2-i_1)!}(\dd_w\dd_\theta)^{i_2-i_1}
\Nop(\dd_w\dd_\theta)^{i_1}\Mop_{\dd_\theta\bar w}\Lfun^2_{j-i_2},
\end{equation*}
and, in a second step, we obtain (since 
$\Mop_{\dd_\theta\bar w}=\Mop_{\bar\partial_w\theta}^{-1}$)
\begin{multline}
\big\{\Sop'\Mop_{\dd_\theta\bar w}[\Lfun^{(2)}]
\big\}^{\langle k\rangle}
=\big\{\Sop\Mop_{\dd_\theta\bar w}\Sop^{-1}\Nop\Sop\Mop_{\dd_\theta\bar w}[\Lfun^{(2)}]
\big\}^{\langle k\rangle}
\\
=\sum_{j=0}^{k+1}m^{2-j}\sum_{i_3=0}^{j}\sum_{i_2=0}^{i_3}
\sum_{i_1=0}^{i_2}
\frac{(-1)^{i_2-i_1}2^{-i_3}}{i_1!(i_2-i_1)!(i_3-i_2)!}
(\dd_w\dd_\theta)^{i_3-i_2}\Mop_{\dd_\theta\bar w}
(\dd_w\dd_\theta)^{i_2-i_1}\Nop(\dd_w\dd_\theta)^{i_1}
\Mop_{\dd_\theta\bar w}\Lfun^2_{j-i_3}.
\label{eq-6.6.11}
\end{multline}
We also need to calculate the corresponding expressions for $\Rfun^{(2)}$. 
First, we find that (see \eqref{eq-initial1.1})
\begin{multline}
\big\{\Sop\Mop_{\dd_\theta\bar w}[\Rfun^{(2)}]\big\}^{\langle k\rangle}
=\bigg\{\Sop\bigg[-\frac{4m^2}{\pi}\frac{|z-w|^2}{\dd_\theta\bar w}
+\frac{4m}{\pi}-\frac{2m}{\pi}(\bar z-\bar w)\,
\dd_\theta\frac{1}{\dd_\theta\bar w}\bigg]\bigg\}^{\langle k\rangle}
\\
=\frac{4m}{\pi}
-\frac{4m^2}{\pi}
\frac{|z-w|^2}{\dd_\theta \bar w}-\sum_{j=1}^{k+1}\frac{(2m)^{2-j}}{j!\pi}
(\dd_w\dd_\theta)^{j-1}\bigg\{j(\bar z-\bar w)\dd_\theta\frac{1}
{\dd_\theta\bar w}
+\dd_w\dd_\theta\frac{|z-w|^2}{\dd_\theta \bar w}\bigg\},
\label{eq-6.6.12}
\end{multline}
and, consequently, 
\begin{multline}
\big\{\Nop\Sop\Mop_{\dd_\theta\bar w}[\Rfun^{(2)}]\big\}^{\langle k\rangle}
=-\frac{4m^2}{\pi}
\frac{\bar z-\bar w}{\dd_\theta \bar w}
\\
-\sum_{j=1}^{k+1}\frac{(2m)^{2-j}}{j!\pi}\Nop
(\dd_w\dd_\theta)^{j-1}\bigg\{j(\bar z-\bar w)\dd_\theta\frac{1}
{\dd_\theta\bar w}
+\dd_w\dd_\theta\frac{|z-w|^2}{\dd_\theta \bar w}\bigg\}
\\
=-\sum_{j=0}^{k+1}\frac{(2m)^{2-j}}{j!\pi}\Nop
\bigg\{j(\dd_w\dd_\theta)^{j-1}\Mop_{\bar z-\bar w}\,
\dd_\theta\frac{1}{\dd_\theta\bar w}
+(\dd_w\dd_\theta)^j\frac{|z-w|^2}{\dd_\theta \bar w}\bigg\},
\label{eq-6.6.12.1}
\end{multline}
where we interpret the term for $j=0$ is the natural fashion.   
By expansion of the operators involved, it is clear that, formally,
\begin{equation*}
\Sop\Mop_{\dd_\theta\bar w}\Sop^{-1}
=\sum_{j=0}^{+\infty}(2m)^{-j}\sum_{i=0}^{j}
\frac{(-1)^{i}}{i!(j-i)!}(\dd_w\dd_\theta)^{j-i}
\Mop_{\dd_\theta\bar w}(\dd_w\dd_\theta)^{i},
\end{equation*}
which we apply to \eqref{eq-6.6.12.1}:
\begin{multline}
\big\{\Sop'\Mop_{\dd_\theta\bar w}[\Rfun^{(2)}]\big\}^{\langle k\rangle}
=\big\{\Sop\Mop_{\dd_\theta\bar w}\Sop^{-1}
\Nop\Sop\Mop_{\dd_\theta\bar w}[\Rfun^{(2)}]\big\}^{\langle k\rangle}
=-\frac{1}{\pi}\sum_{j=0}^{k+1}(2m)^{2-j}\sum_{i_2=0}^{j}
\sum_{i_1=0}^{i_2}\frac{(-1)^{i_2-i_1}}{i_2!i_1!(i_2-i_1)!}
\\
\times
(\dd_w\dd_\theta)^{i_1}\Mop_{\dd_\theta\bar w}(\dd_w\dd_\theta)^{i_2-i_1}\Nop
\bigg\{(j-i_2)(\dd_w\dd_\theta)^{j-i_2-1}\Mop_{\bar z-\bar w}
\dd_\theta\frac{1}{\dd_\theta\bar w}
+(\dd_w\dd_\theta)^{j-i_2}\frac{|z-w|^2}{\dd_\theta \bar w}\bigg\}.
\label{eq-6.6.13.1}
\end{multline}
We may now apply the criterion \eqref{negliampeq4.b} for each power of $m$ 
separately. The first part of the criterion, which involves $\Sop$, 
says (for $j=0$)
\begin{equation}
\frac{4}{\pi}\frac{|z-w|^2}{\dd_\theta\bar w}
+\Mop_{\dd_\theta\bar w}\Lfun^2_{0}\in\Mop_{z-w}\calR_2, 
\label{eq-4.13'}
\end{equation}
and (for $j=1$)
\begin{equation}
-\frac{4}{\pi}+\frac{2}{\pi}(\bar z-\bar w)\dd_\theta\frac{1}{\dd_\theta\bar w}
+\frac{2}{\pi}\dd_w\dd_\theta\frac{|z-w|^2}{\dd_\theta\bar w}
+\Mop_{\dd_\theta\bar w}\Lfun^2_{1}
+\frac12\dd_w\dd_\theta\Mop_{\dd_\theta\bar w}\Lfun^2_{0}
\in \Mop_{z-w}\calR_2, 
\label{eq-4.14}
\end{equation}
while (for $j=2,3,4,\ldots$) 
\begin{multline}
\frac{2^{2-j}}{j!\pi}
(\dd_w\dd_\theta)^{j-1}\bigg\{j(\bar z-\bar w)\dd_\theta\frac{1}
{\dd_\theta\bar w}
+\dd_w\dd_\theta\frac{|z-w|^2}{\dd_\theta \bar w}\bigg\} 
+\sum_{i=0}^{j}
\frac{2^{-i}}{i!}(\dd_w\dd_\theta)^i\Mop_{\dd_\theta\bar w}\Lfun^2_{j-i}
\in\Mop_{z-w}\calR_2.
\label{eq-4.15}
\end{multline}
As for the second part of the criterion \eqref{negliampeq4.b}, which 
involves $\Sop'$, it says that (use \eqref{eq-6.6.11} and 
\eqref{eq-6.6.13.1}) for $j=0,1,2,\ldots$, 
\begin{multline}
\sum_{i_3=0}^{j}\sum_{i_2=0}^{i_3}
\sum_{i_1=0}^{i_2}
\frac{(-1)^{i_2-i_1}2^{-i_3}}{i_1!(i_2-i_1)!(i_3-i_2)!}
(\dd_w\dd_\theta)^{i_3-i_2}\Mop_{\dd_\theta\bar w}
(\dd_w\dd_\theta)^{i_2-i_1}\Nop(\dd_w\dd_\theta)^{i_1}
\Mop_{\dd_\theta\bar w}\Lfun^2_{j-i_3}
\\
+\frac{2^{2-j}}{\pi}\sum_{i_2=0}^{j}
\sum_{i_1=0}^{i_2}\frac{(-1)^{i_2-i_1}}{i_2!i_1!(i_2-i_1)!}
\\
\times
(\dd_w\dd_\theta)^{i_1}\Mop_{\dd_\theta\bar w}(\dd_w\dd_\theta)^{i_2-i_1}\Nop
\bigg\{(j-i_2)(\dd_w\dd_\theta)^{j-i_2-1}\Mop_{\bar z-\bar w}
\dd_\theta\frac{1}{\dd_\theta\bar w}
+(\dd_w\dd_\theta)^{j-i_2}\frac{|z-w|^2}{\dd_\theta \bar w}\bigg\}
\in\Mop_{z-w}\calR_2.
\label{eq-4.17.1}
\end{multline}
For $j=0$, the condition \eqref{eq-4.17.1} reads
\begin{equation}
\frac{4}{\pi}\Mop_{\dd_\theta\bar w}\Nop\frac{|z-w|^2}{\dd_\theta\bar w}
+\Mop_{\dd_\theta\bar w}\Nop\Mop_{\dd_\theta\bar w}\Lfun^2_{0}\in\Mop_{z-w}\calR_2, 
\label{eq-4.16.1}
\end{equation}
and for $j=1$, 
\begin{multline}
\Mop_{\dd_\theta\bar w}\Nop\Mop_{\dd_\theta\bar w}\Lfun^2_{1}
+\frac12\dd_w\dd_\theta\Mop_{\dd_\theta\bar w}\Nop
\Mop_{\dd_\theta\bar w}\Lfun^2_{0}
-\frac12\Mop_{\dd_\theta\bar w}^{-1}\dd_w\dd_\theta\Nop\Mop_{\dd_\theta\bar w}\Lfun^2_{0}
\\
+\frac12\Mop_{\dd_\theta\bar w}\Nop\dd_w\dd_\theta\Mop_{\dd_\theta\bar w}\Lfun^2_{0}
+\frac{2}{\pi}\dd_w\dd_\theta\Mop_{\dd_\theta\bar w}\Nop\bigg[
\frac{|z-w|^2}{\dd_\theta \bar w}\bigg]
-\frac{2}{\pi}\Mop_{\dd_\theta\bar w}\dd_w\dd_\theta\Nop\bigg[
\frac{|z-w|^2}{\dd_\theta \bar w}\bigg]
\\
+\frac{2}{\pi}\Mop_{\dd_\theta\bar w}\Nop\dd_w\dd_\theta\bigg[
\frac{|z-w|^2}{\dd_\theta \bar w}\bigg]+
\frac{2}{\pi}\Mop_{\dd_\theta\bar w}\Nop
\bigg[(\bar z-\bar w)\dd_\theta\frac{1}{\dd_\theta\bar w}\bigg]
\in\Mop_{z-w}\calR_2.
\label{eq-4.16.2}
\end{multline}
As the function $1/\dd_\theta\bar w$ is in $\calR_1$, we may rewrite 
\eqref{eq-4.16.1} in the form
\begin{equation}
\Nop\bigg[\frac{4}{\pi}\frac{|z-w|^2}{\dd_\theta \bar w}+\Mop_{\dd_\theta\bar w}
\Lfun^2_{0}\bigg]
\in\Mop_{z-w}\calR_2.
\label{eq-4.16.3}
\end{equation}
and we may reorganize \eqref{eq-4.16.2}:
\begin{multline}
\Nop\Mop_{\dd_\theta\bar w}\Lfun^2_{1}
+\Mop_{\dd_\theta\bar w}^{-1}\dd_w\dd_\theta\Mop_{\dd_\theta\bar w}\Nop
\bigg[\frac12\Mop_{\dd_\theta\bar w}\Lfun^2_{0}+\frac{2}{\pi}
\frac{|z-w|^2}{\dd_\theta \bar w}\bigg]
-\dd_w\dd_\theta\Nop
\bigg[\frac12\Mop_{\dd_\theta\bar w}\Lfun^2_{0}+
\frac{2}{\pi}\frac{|z-w|^2}{\dd_\theta \bar w}\bigg]
\\
+\Nop\dd_w\dd_\theta
\bigg[\frac12\Mop_{\dd_\theta\bar w}\Lfun^2_{0}+\frac{2}{\pi}
\frac{|z-w|^2}{\dd_\theta \bar w}\bigg]
+\frac{2}{\pi}\Nop\bigg[(\bar z-\bar w)\dd_\theta\frac{1}{\dd_\theta\bar w}\bigg]
\in\Mop_{z-w}\calR_2.
\label{eq-4.16.4}
\end{multline}
In terms of the usual commutator 
$\boldsymbol{\lfloor}\Aop\boldsymbol{,}\Bop\boldsymbol{\rfloor}
:=\Aop\Bop-\Bop\Aop$,
we may express \eqref{eq-4.16.4} as
\begin{multline}
\Nop\bigg[
\Mop_{\dd_\theta\bar w}[\Lfun^2_{1}]
+\dd_w\dd_\theta\bigg\{\frac12\Mop_{\dd_\theta\bar w}\Lfun^2_{0}+
\frac{2}{\pi}\frac{|z-w|^2}{\dd_\theta \bar w}\bigg\}
-\frac{4}{\pi}+\frac{2}{\pi}(\bar z-\bar w)\dd_\theta\frac{1}{\dd_\theta\bar w}
\bigg]
\\
+\Mop_{\dd_\theta\bar w}^{-1}{\boldsymbol{\Big\lfloor}}
\dd_w\dd_\theta\boldsymbol{,}\Mop_{\dd_\theta\bar w}{\boldsymbol{\Big\rfloor}}
\Nop\bigg[\frac12\Mop_{\dd_\theta\bar w}\Lfun^2_{0}+\frac{2}{\pi}
\frac{|z-w|^2}{\dd_\theta\bar w}\bigg]
\in\Mop_{z-w}\calR_2.
\label{eq-4.16.5}
\end{multline}

\subsection{The local asymptotics for the weighted bianalytic Bergman
kernel: the corrective algorithm II} 
\label{subsec-II}

Here, we show how to combine the given criteria involving $\Sop$ and $\Sop'$ 
in a single criterion for $\Lfun_j^2$, which we explain for $j=0,1$. 

We first observe first that \eqref{eq-4.13'} and \eqref{eq-4.16.3} may be 
united in a single condition:
\begin{equation}
\Mop_{\dd_\theta\bar w}\Lfun^2_{0}+\frac{4}{\pi}
\frac{|z-w|^2}{\dd_\theta \bar w}\in\Mop_{z-w}^2\calR_2.
\label{eq-2.05}
\end{equation}
Second, we observe that \eqref{eq-4.14} and \eqref{eq-4.16.5} may be united 
in a single condition too:
\begin{multline}
\Mop_{\dd_\theta\bar w}\Lfun^2_{1}
+\dd_w\dd_\theta\bigg\{\frac12\Mop_{\dd_\theta\bar w}\Lfun^2_{0}+
\frac{2}{\pi}\frac{|z-w|^2}{\dd_\theta \bar w}\bigg\}
-\frac{4}{\pi}+\frac{2}{\pi}(\bar z-\bar w)\dd_\theta\frac{1}{\dd_\theta\bar w}
\\
+\Mop_{z-w}\Mop_{\dd_\theta\bar w}^{-1}{\boldsymbol{\Big\lfloor}}
\dd_w\dd_\theta\boldsymbol{,}\Mop_{\dd_\theta\bar w}{\boldsymbol{\Big\rfloor}}
\Nop\bigg[\frac12\Mop_{\dd_\theta\bar w}\Lfun^2_{0}+\frac{2}{\pi}
\frac{|z-w|^2}{\dd_\theta \bar w}\bigg]\in\Mop_{z-w}^2\calR_2.
\label{eq-2.06}
\end{multline}
Similar but more complicated statements apply for $j=2,3,4,\ldots$ as well. 
Finally, we observe that by Lemma \ref{lem-uniq1}, the condition 
\eqref{eq-2.05} determines $\Lfun^2_{0}$ uniquely, and as a consequence, 
the condition \eqref{eq-2.06} determines $\Lfun^2_{1}$ uniquely. 
Although in principle we are now set to obtain the concrete expressions for
$\Lfun_j^2$, the computation is quite messy. For this reason, we show  below
how to proceed.

\subsection{The local asymptotics for the weighted bianalytic Bergman
kernel: the corrective algorithm III} 
\label{subsec-III}

We continue the analysis of the corrective algorithm, with the aim to obtain
concrete expressions for $\Lfun_j^2$, for $j=0,1,2$. 
To simplify the notation, we shall at times write
\begin{equation}
\vbeta(z,w)=\vbeta_Q(z,w):=\partial_z\bar\partial_w Q(z,w),
\label{eq-beta1}
\end{equation}
and sometimes we abbreviate $\vbeta=\vbeta(z,w)$. The function $\lambda(z,w)$ 
has the interpretation of the polarization of $\hDelta Q$, which defines in a 
natural fashion a Riemannian metric wherever $\hDelta Q>0$ (see, e.g., 
\cite{ahm1}).   

We recall that $1/\dd_\theta\bar w=\bar\partial_w\theta$,
and observe that 
\begin{equation}
\dd_\theta\frac{1}{\dd_\theta\bar w}=
[\bar\partial_w\theta]^{-1}\bar\partial_w^2\theta
=\frac{\bar\partial_w^2\theta}{\bar\partial_w\theta}=
\bar\partial_w\log(\bar\partial_w\theta),
\label{eq-4.20.1}
\end{equation}
and that by \eqref{eq-4.19.1},
\begin{multline}
\frac{|z-w|^2}{\dd_\theta\bar w}=|z-w|^2\bar\partial_w\theta
=\sum_{i=0}^{+\infty}\frac{1}{(i+1)!}(\bar w-\bar z)(w-z)^{i+1}
\partial_z^{i+1}\dbar_w Q(z,w)
\\
=\sum_{i=0}^{+\infty}\frac{1}{(i+1)!}(\bar w-\bar z)(w-z)^{i+1}
\partial_z^{i}\vbeta(z,w).
\label{eq-4.30}
\end{multline}
By \eqref{eq-4.19.3} combined with Taylor expansion, we see that
\begin{equation}
\log(\dbar_w\theta)=
\log(\partial_z\dbar_w Q(z,w))+
\sum_{n=1}^{+\infty}\frac{(-1)^{n-1}}{n}
\bigg\{\sum_{i=1}^{+\infty}\frac{1}{(i+1)!}(w-z)^{i}
\frac{\partial_z^{i+1}\dbar_w Q(z,w)}{\partial_z\dbar_w Q(z,w)}
\bigg\}^n,
\label{eq-4.19.4}
\end{equation}
so that in view of \eqref{eq-4.20.1}, using standard coset notation,
\begin{equation}
\dd_\theta\frac{1}{\dd_\theta\bar w}
\in \bar\partial_w\log\vbeta(z,w)
+\frac{1}{2}(w-z)\partial_z\bar\partial_w\log\vbeta(z,w)
+\Mop_{z-w}^2\calR_1.
\label{eq-4.22}
\end{equation}
From the expansion \eqref{eq-4.30}, we get that
\begin{equation}
\frac{|z-w|^2}{\dd_\theta\bar w}
\in |z-w|^2\vbeta(z,w)+\frac12|z-w|^2(w-z)\partial_z\vbeta(z,w)+
\Mop_{z-w}^3\calR_2,
\label{eq-4.21}
\end{equation}
Next, we put
\begin{equation}
\Lfun^2_{0}(z,w):=-\frac{4}{\pi}|z-w|^2[\vbeta(z,w)]^2
=-\frac{4}{\pi}|z-w|^2[\partial_z\bar\partial_w Q(z,w)]^2,
\label{eq-4.24}
\end{equation}
which is biholomorphic in each of the variables $(z,\bar w)$. We quickly check
that the condition \eqref{eq-2.05} is met. By Lemma \ref{lem-uniq1}, 
this is the only possible choice of $\Lfun^2_{0}$. A more precise calculation
shows that
\begin{equation}
\Mop_{\dd_\theta\bar w}\Lfun^2_{0}+
\frac{4}{\pi}\frac{|z-w|^2}{\dd_\theta\bar w}
+\frac{4}{\pi}(z-w)|z-w|^2\partial_z\vbeta(z,w)\in\Mop_{z-w}^3\calR_2, 
\label{eq-4.25}
\end{equation}
from which we may conclude that
\begin{equation}
\Nop\bigg[\frac12\Mop_{\dd_\theta\bar w}\Lfun^2_{0}+
\frac{2}{\pi}\frac{|z-w|^2}{\dd_\theta\bar w}\bigg]
+\frac{2}{\pi}|z-w|^2\partial_z\vbeta(z,w)
\in\Mop_{z-w}^2\calR_2, 
\label{eq-4.25.0}
\end{equation}
and that
\begin{equation}
\dd_w\dd_\theta\bigg\{\Mop_{\dd_\theta\bar w}\Lfun^2_{0}+
\frac{4}{\pi}\frac{|z-w|^2}{\dd_\theta\bar w}
+\frac{4}{\pi}(z-w)|z-w|^2\partial_z\vbeta(z,w)\bigg\}
\in\Mop_{z-w}^2\calR_2. 
\label{eq-4.25:1}
\end{equation}
If we recall \eqref{eq-coordch1} and \eqref{eq-4.19-ratio}, we see that
\begin{multline}
\dd_\theta\dd_w\Big\{(z-w)|z-w|^2\partial_z\vbeta(z,w)\Big\}
=\frac{1}{\dbar_w\theta}\dbar_w\bigg(\partial_w-
\frac{\partial_w\theta}{\dbar_w\theta}\dbar_w\bigg)
\Big\{(z-w)|z-w|^2\partial_z\vbeta(z,w)\Big\}
\\
\in\frac{1}{\dbar_w\theta}\dbar_w\big\{-2|z-w|^2\partial_z\vbeta(z,w)
\big\}+
\Mop_{z-w}^2\calR_2
=2(z-w)\frac{\partial_z\vbeta(z,w)}{\vbeta(z,w)}
-2|z-w|^2\frac{\partial_z\dbar_w\vbeta(z,w)}{\vbeta(z,w)}+
\Mop_{z-w}^2\calR_2.
\label{eq-4.25:2}
\end{multline}
By combining this calculation with \eqref{eq-4.25:1}, we see that
\begin{equation}
\dd_\theta\dd_w\bigg\{\frac12\Mop_{\dd_\theta\bar w}\Lfun^2_{0}+
\frac{2}{\pi}\frac{|z-w|^2}{\dd_\theta\bar w}\bigg\}
+\frac{4}{\pi}(z-w)\frac{\partial_z\vbeta(z,w)}{\vbeta(z,w)}
-\frac{4}{\pi}|z-w|^2\frac{\partial_z\dbar_w\vbeta(z,w)}{\vbeta(z,w)} 
\in\Mop_{z-w}^2\calR_2.
\label{eq-4.25:3}
\end{equation}
In view of \eqref{eq-4.22}, \eqref{eq-4.25:3}, we see that \eqref{eq-2.06}
taken modulo $\Mop_{z-w}\calR_2$ (see \eqref{eq-4.14}) gives
\[
\frac{\Lfun^2_{1}}{\dbar_w\theta}-\frac{4}{\pi}+\frac{2}{\pi}(\bar z-\bar w)
\frac{\dbar_w\vbeta(z,w)}{\vbeta(z,w)}
\in\Mop_{z-w}\calR_2,
\]
which suggests that we should look for $\Lfun^2_{1}$ of the form
\begin{equation}
\Lfun^2_{1}(z,w):=\frac{4}{\pi}\vbeta(z,w)
-\frac{2}{\pi}(\bar z-\bar w)\dbar_w\vbeta(z,w)
+\frac{2}{\pi}(z-w)\partial_z\vbeta(z,w)+|z-w|^2\Xi_1(z,w),
\label{eq-4.40}
\end{equation}
where $\Xi_1(z,w)$ is holomorphic in $(z,\bar w)$. We find that
\begin{multline}
\Mop_{\dd_\theta\bar w}\Lfun^2_{1}=\frac{\Lfun^2_{1}}{\dbar_w\theta}
\in \frac{4}{\pi}-\frac{2}{\pi}(\bar z-\bar w)
\frac{\dbar_w\vbeta(z,w)}{\vbeta(z,w)}
-\frac{4}{\pi}(w-z)\frac{\partial_z\vbeta(z,w)}{\vbeta(z,w)}
\\
-\frac{1}{\pi}
|z-w|^2\frac{[\partial_z\vbeta(z,w)][\dbar_w\vbeta(z,w)]}
{[\vbeta(z,w)]^2}+|z-w|^2\frac{\Xi_1(z,w)}{\vbeta(z,w)}+\Mop_{z-w}^2\calR_2.
\label{eq-4.38:2}
\end{multline}
As we add up the relations \eqref{eq-4.22}, \eqref{eq-4.25:3}, and 
\eqref{eq-4.38:2}, the result is
\begin{multline}
\Mop_{\dd_\theta\bar w}\Lfun^2_{1}+\dd_\theta\dd_w
\bigg\{\frac12\Mop_{\dd_\theta\bar w}\Lfun^2_{0}+\frac{2}{\pi}
\frac{|z-w|^2}{\dd_\theta \bar w}\bigg\}
-\frac{4}{\pi}+\frac{2}{\pi}(\bar z-\bar w)\dd_\theta\frac{1}{\dd_\theta\bar w}
\\
\in \frac{4}{\pi}-\frac{2}{\pi}(\bar z-\bar w)
\frac{\dbar_w\vbeta(z,w)}{\vbeta(z,w)}
-\frac{4}{\pi}(w-z)\frac{\partial_z\vbeta(z,w)}{\vbeta(z,w)}
-\frac{1}{\pi}
|z-w|^2\frac{[\partial_z\vbeta(z,w)][\dbar_w\vbeta(z,w)]}
{[\vbeta(z,w)]^2}
\\
+|z-w|^2\frac{\Xi_1(z,w)}{\vbeta(z,w)}
-\frac{4}{\pi}(z-w)\frac{\partial_z\vbeta(z,w)}{\vbeta(z,w)}
+\frac{4}{\pi}
|z-w|^2\frac{\partial_z\dbar_w\vbeta(z,w)}{\vbeta(z,w)}
-\frac{4}{\pi}
\\
+\frac{2}{\pi}(\bar z-\bar w)\big\{\dbar_w\log\vbeta(z,w)
+\frac12(w-z)\partial_z\dbar_w\log\vbeta(z,w)\big\}
+\Mop_{z-w}^2\calR_2, 
\label{eq-4.41}
\end{multline}
which we simplify to 
\begin{multline}
\Mop_{\dd_\theta\bar w}\Lfun^2_{1}+\dd_w\dd_\theta
\bigg\{\frac12\Mop_{\dd_\theta\bar w}\Lfun^2_{0}+
\frac{2}{\pi}\frac{|z-w|^2}{\dd_\theta \bar w}\bigg\}
-\frac{4}{\pi}+\frac{2}{\pi}(\bar z-\bar w)\dd_\theta\frac{1}{\dd_\theta\bar w}
\\
\in |z-w|^2\bigg\{\frac{3}{\pi}\frac{\partial_z\dbar_w\vbeta(z,w)}{\vbeta(z,w)}
+\frac{\Xi_1(z,w)}{\vbeta(z,w)}\bigg\}+\Mop_{z-w}^2\calR_2.
\label{eq-4.42}
\end{multline}
We want to implement this information into \eqref{eq-2.06}. Then we need
the commutator calculation
\begin{equation}
\boldsymbol{\big\lfloor}\dd_w\dd_\theta\boldsymbol{,}\Mop_{\dd_\theta\bar w}
\boldsymbol{\big\rfloor}=
\Mop_{\dd_w\dd_\theta^2\bar w}+\Mop_{\dd_\theta^2\bar w}\dd_w+
\Mop_{\dd_w\dd_\theta\bar w}\dd_\theta,
\end{equation}
which together with \eqref{eq-4.25.0} leads to
\begin{multline}
\Mop_{\dd_\theta\bar w}^{-1}
\boldsymbol{\big\lfloor}\dd_\theta\dd_w\boldsymbol{,}
\Mop_{\dd_\theta\bar w}\boldsymbol{\big\rfloor}\,
\Nop\bigg[\frac12\Mop_{\dd_\theta\bar w}\Lfun^2_{0}+\frac{2}{\pi}
\frac{|z-w|^2}{\dd_\theta\bar w}\bigg]
\\
\in-\frac{2}{\pi}(\bar z-\bar w)
\frac{[\partial_z\vbeta(z,w)][\dbar_w\vbeta(z,w)]}
{[\vbeta(z,w)]^2}
+\Mop_{z-w}\calR_2.
\label{eq-4.42.1}
\end{multline}
Putting \eqref{eq-4.42} and \eqref{eq-4.42.1} together, we find that
\begin{multline}
\Mop_{\dd_\theta\bar w}\Lfun^2_{1}+\dd_w\dd_\theta
\bigg\{\frac12\Mop_{\dd_\theta\bar w}\Lfun^2_{0}
+\frac{2}{\pi}\frac{|z-w|^2}{\dd_\theta \bar w}\bigg\}
-\frac{4}{\pi}+\frac{2}{\pi}
(\bar z-\bar w)\dd_\theta\frac{1}{\dd_\theta\bar w}
\\
+
\Mop_{z-w}\Mop_{\dd_\theta\bar w}^{-1}\boldsymbol{\big\lfloor}
\dd_w\dd_\theta,\Mop_{\dd_\theta\bar w}\boldsymbol{\big\rfloor}\,
\Nop\bigg[\frac12\Mop_{\dd_\theta\bar w}\Lfun^2_{0}+\frac{2}{\pi}
\frac{|z-w|^2}{\dd_\theta \bar w}\bigg]
\\
\in |z-w|^2
\bigg\{\frac{3}{\pi}\frac{\partial_z\dbar_w\vbeta(z,w)}{\vbeta(z,w)}
+\frac{\Xi_1(z,w)}{\vbeta(z,w)}
-\frac{2}{\pi}\frac{[\partial_z\vbeta(z,w)][\dbar_w\vbeta(z,w)]}
{[\vbeta(z,w)]^2} 
\bigg\} +\Mop_{z-w}^2\calR_2
\label{eq-4.44}
\end{multline}
By Lemma \ref{lem-uniq1}, we may conclude from \eqref{eq-4.44} that
\begin{equation}
\Xi_1(z,w)=-\frac{3}{\pi}\partial_z\dbar_w\vbeta(z,w)+\frac{2}{\pi}
\frac{[\partial_z\vbeta(z,w)][\dbar_w\vbeta(z,w)]}
{\vbeta(z,w)}.
\end{equation}
In conclusion, we obtain that $\Lfun_1^2$ has the form
\begin{multline}
\label{eq-7.9.21}
\Lfun^2_{1}(z,w)=\frac{4}{\pi}\vbeta(z,w)
-\frac{2}{\pi}(\bar z-\bar w)\dbar_w\vbeta(z,w)
+\frac{2}{\pi}(z-w)\partial_z\vbeta(z,w)
\\
-\frac{3}{\pi}|z-w|^2\partial_z\dbar_w\vbeta(z,w)+\frac{2}{\pi}|z-w|^2
\frac{[\partial_z\vbeta(z,w)][\dbar_w\vbeta(z,w)]}
{\vbeta(z,w)}.
\end{multline}
We have also obtained an explicit expression for the third term $\Lfun^2_{2}$, 
but we omit the rather long computation. 
The result is
\begin{equation}
\label{eq-7.9.22}
\Lfun^2_{2}(z,w)=\frac{2}{\pi}\partial_z\bar\partial_w \log\vbeta+ 
\frac{1}{\pi}(\bar w-\bar z)\partial_z\bar\partial_w^2\log\vbeta
+\frac{1}{\pi}(z-w)\partial_z^2\bar\partial_w\log\vbeta
+|z-w|^2 \Xi_2(z,w), 
\end{equation}
where 
\begin{multline}
\Xi_2(z,w)= 
\frac3{2\pi}\frac{[\partial_z\bar\partial_w^2\vbeta][\partial_z\vbeta]}
{[\vbeta]^2} 
-\frac{13}{2\pi} 
\frac{[\partial_z b][\partial_z\bar\partial_w\vbeta][\bar\partial_w\vbeta]}
{[\vbeta]^3}+\frac{3}{2\pi}\frac{[\partial_z\bar\partial_w\vbeta]^2}{[\vbeta]^2}
\\
-\frac{1}{\pi}\frac{[\partial_z\vbeta]^2[\bar\partial_w^2\vbeta]}{[\vbeta]^3}
+ \frac{17}{4\pi} \frac{[\partial_z\vbeta]^2[\bar\partial_w\vbeta]^2}{[\vbeta]^4}
-\frac{2}{3\pi} \frac{\partial_z^2 \bar \partial_w^2\vbeta}{\vbeta}
+\frac{3}{2\pi}\frac{[\partial_z^2\bar\partial_w\vbeta][\bar\partial_w\vbeta]}
{[\vbeta]^2} 
\\
-\frac{1}{\pi}\frac{[\partial_z^2\vbeta][\bar\partial_w \vbeta]^2}{[\vbeta]^3}
+\frac{1}{3\pi}\frac{[\partial_z^2\vbeta][\bar \partial_w^2\vbeta]}{[\vbeta]^2}. 
\end{multline}
\medskip

\begin{rem}
\label{rem-AA}
The formula for $\Lfun_0^2$ is given by \eqref{eq-4.24}, the formula for  
$\Lfun_1^2$ is given by \eqref{eq-7.9.21}, while $\Lfun_2^2$ is expressed by
\eqref{eq-7.9.22}.
\end{rem}


\section{A priori diagonal estimates of weighted Bergman kernels}

\subsection{Estimation of point evaluations for local weighted Bergman
spaces}
The spaces $\mathrm{PA}^2_{q,m}$ considered in this paper involve the weight 
$\e^{-2mQ}$, where (locally at least) $Q$ is subharmonic and smooth. The
estimates we derived in Section \ref{sec-interlude} for $q=2$ do not apply 
immediately, as the weights considered there had rather the converse property 
of being exponentials of subharmonic functions. Nevertheless, this difficulty 
is easy to overcome. Compare with, e.g., \cite{ahm1} for the analytic case.

\begin{prop} 
\label{pointwise2}
Let $u$ be bianalytic in the disk $\D(z_0,\delta m^{-1/2})$, where $\delta$ 
is a positive real parameter and $m\ge1$. 
Suppose $Q$ is $C^{1,1}$-smooth in $\bar\D(z_0,\delta)$ and subharmonic in 
$\D(z_0,\delta)$, with 
\[
A:=
\mathrm{essup}_{z\in\D(z_0,\delta)}\hDelta Q(z)<+\infty.
\] 
We then have the estimate
\begin{equation*}
|u(z_0)|^2\le\frac{8m}{\pi\delta^2}(1+6A^2)\,\e^{2A\delta^2}\e^{2mQ(z_0)}
\int_{\D(z_0,\delta m^{-1/2})}|u|^2\e^{-2mQ}\diff A.
\end{equation*}
\end{prop}
\begin{proof}
Without loss of generality, we may assume that $z_0=0$. 
For $\xi\in\D$, we put 
\[
\psi_m(\xi):=A\delta^2|\xi|^2-m Q(\delta m^{-1/2}\xi)\quad\text{and}
\quad u_m(\xi):=u(\delta m^{-1/2}\xi).
\]
It is immediate that  
\begin{equation}
-m Q(\delta m^{-1/2}\xi)\le\psi_m(\xi)\le A\delta^2-m Q(\delta m^{-1/2}\xi),
\label{eq-8.1.1}
\end{equation}
and we may also check that
\begin{equation}
0\le\hDelta\psi_m(\xi)=A\delta^2-\delta^2(\hDelta Q)(\delta m^{-1/2}\xi)
\le A\delta^2,\qquad \xi\in\D.
\label{eq-8.1.2}
\end{equation}
so that $\psi_m$ is subharmonic, and we obtain that
\begin{equation}
0\ge\Gop[\hDelta \psi_m](0)=\frac{1}{\pi}\int_\D\log|\xi|^2\hDelta\psi_m
(\xi)\diff A(\xi)\ge A\delta^2\int_\D\log|\xi|^2\diff A(\xi)\ge-A\delta^2.
\label{eq-8.1.3}
\end{equation} 
Next, since $\psi$ is subharmonic with finite Riesz mass, we may apply 
Proposition \ref{pointwise1}:  
\begin{multline} 
|u(0)|^2\e^{-2mQ(0)} = |u_m(0)|^2\e^{2\psi_m(0)} \leq 
\frac{8}{\pi}\big[1+6|\mathbf{G}[\hDelta\psi_m](0)|^2\big] 
\int_{\D} |u_m|^2 \e^{2\psi_m} \diff A 
\\ 
\leq \frac{8}{\pi}(1+6A^2\delta^4) 
\int_{\D} |u_m|^2 \e^{2\psi_m} \diff A
\le \frac{8}{\pi}(1+6A^2\delta^4)\,\e^{2A\delta^2} \int_{\D} 
|u(\delta m^{-1/2}\xi)|^2 \e^{-2mQ(\delta m^{-1/2}\xi)} \diff A(\xi)
\\
=\frac{8m}{\pi\delta^2}(1+6A^2\delta^4)\,\e^{2A\delta^2} \int_{\D(0,\delta m^{-1/2})} 
|u(z)|^2 \e^{-2mQ(z)} \diff A(z),
\end{multline}
and the claim follows. Note that we used the estimates \eqref{eq-8.1.1} and
\eqref{eq-8.1.2}. The proof is complete.
\end{proof}

We can also estimate the $\dbar$-derivative:

\begin{prop} 
\label{pointwisederiv}
Suppose we are in the setting of Proposition \ref{pointwise2}.
We then have the estimate
\begin{equation*}
|\dbar u(z_0)|^2\le\frac{3m}{\pi\delta^2}\,\e^{2A\delta^2}\e^{2mQ(z_0)}
\int_{\D(z_0,\delta m^{-1/2})}|u|^2\e^{-2mQ}\diff A.
\end{equation*}
\end{prop}

\begin{proof}
The argument is analogous to that of the proof of Proposition \ref{pointwise2},
except that it is based on Proposition \ref{pointwiseneg0}.
\end{proof}

\subsection{Estimation of the weighted Bergman kernel on the 
diagonal}
We recall that $\mathrm{PA}^2_{2,m}$ is the weighted bianalytic Bergman space 
on the domain $\Omega$ in $\C$, supplied with the weight $\e^{-2mQ}$. The 
associated reproducing kernel is written $K_{2,m}$.

\begin{cor} 
\label{cor-8.2}
Suppose that $z_0\in\Omega$ is such that $\D(z_0,\delta)\Subset\Omega$ 
for some positive $\delta$. Suppose moreover that $Q$ is $C^{1,1}$-smooth and 
subharmonic in $\Omega$, with 
\[
A:=
\mathrm{essup}_{z\in\D(z_0,\delta)}\hDelta Q(z)<+\infty.
\] 
For $m\ge1$, we then have the estimate
\begin{equation*}
K_{2,m}(z_0,z_0)\le\frac{8 m}{\pi\delta^2}(1+6A^2)\,\e^{2A\delta^2}\e^{2mQ(z_0)}.
\end{equation*}
\end{cor}

\begin{proof}
This is a rather immediate consequence of Proposition \ref{pointwise2},
as it is well-known that $K_{2,m}(z_0,z_0)$ equals the square of the norm of 
the point evaluation functional at $z_0$.
\end{proof}

\section{Bergman kernels: from local to global}
\label{sec-locglob}

\subsection{Purpose of the section}
In this section we show how the approximate local Bergman kernels actually 
provide asymptotics for the weighted bianalytic Bergman kernel $K_{2,m}$ 
pointwise near the diagonal.

\subsection{The basic pointwise estimate}

We let the potential $Q$, the point $z_0$, the radius $r$, and the cut-off 
function $\chi_0$ be as in Subsections \ref{subsec-assQ} and \ref{subsec-5.3}, 
with the additional requirement mentioned in Subsection \ref{subsec-6.2}.
In particular, we have that $\D(z_0,r)\Subset\Omega$.
We begin with a local weighted bianalytic Bergman kernel 
$\mathrm{mod}(m^{-k-1})$, 
\[ 
K^{\langle k\rangle}_{2,m}(z,w)=\Lfun^{\langle2,k\rangle}(z,w)\,\e^{2mQ(z,w)},
\]
where 
\begin{equation}
\Lfun^{\langle 2,k\rangle}:=m^2\Lfun^{2}_0+m\Lfun^{2}_1+\cdots+m^{-k+1}
\Lfun^{2}_{k+1}
\label{eq-9.2.1}
\end{equation}
is a finite asymptotic expansion (here, it is assumed that 
$z,w\in\mathbb{D}(z_0, r)$ for the expression to make sense). 
The ``coefficients'' $\Lfun^2_j$ are biholomorphic separately in $(z,\bar w)$
in  $\D(z_0,r)\times\D(z_0,r)$,
and in principle they can be obtained from the criteria \eqref{negliampeq3.b}, 
but that is easy to say and hard to do. 
The first two ``coefficients'' were derived in Subsection \ref{subsec-III}, 
and the formula for the third ``coefficient'' was mentioned as well.  
Let us write $\Iop_{2,m}^{\langle k\rangle}$ for the integral operator 
\[ 
\Iop^{\langle k\rangle}_{2,m} [f](z):=\int_{\Omega}f(w) K^{\langle k\rangle}_{2,m}
(z,w)\chi_0(w)\,\e^{-2mQ(w)}\diff A(w). 
\]
We quickly check that 
\begin{equation} 
\Pop_{2,m}\big[K^{\langle k\rangle}_{2,m}(\cdot,w)\chi_0\big](z) = 
\overline{\Iop^{\langle k\rangle}_{2,m}[K_{2,m}(\cdot,z)](w)}, 
\label{eq-9.2.2}
\end{equation}
where we recall that $\Pop_{2,m}$ stands for the orthogonal projection 
$L^2(\Omega,\e^{-2mQ})\to\mathrm{PA}^2_{2,m}$, and $K_{2,m}$ is as before the 
weighted bianalytic Bergman kernel on $\Omega$ with weight $\e^{-2mQ}$. 
We would like to show that $K_{2,m}^{\langle k\rangle}$ and $K_{2,m}$ are close 
pointwise. By \eqref{eq-9.2.2} and the triangle inequality, we have that
\begin{multline} 
\big| K_{2,m}(z,w)-K^{\langle k\rangle}_{2,m}(z,w)\chi_0(z)\big| 
\le\big|K_{2,m}(w,z)-\Iop^{\langle k\rangle}_{2,m}[K_{2,m}(\cdot,z)](w)\big| 
\\
+\big|K^{\langle k\rangle}_{2,m}(z,w) 
\chi_0(z)-\Pop_{2,m}[K^{\langle k\rangle}_{2,m}(\cdot,w)\chi_0](z)\big|.
\label{glob1}
\end{multline}

\subsection{Analysis of the first term in the pointwise estimate}

In the sequel, $k$ is a \emph{fixed nonegative integer}.
 To estimate  the first term on the right-hand side of \eqref{glob1}, 
we use that $K^{\langle k\rangle}_{2,m}$ is a local reproducing kernel 
$\mathrm{mod}(m^{-k})$, which is expressed by \eqref{negliampeq2.b}. 
Before we go into that, we observe that if $u$ is biharmonic in $\D(z_0,r)$, 
then
\begin{multline} 
\Iop^{\langle k\rangle}_{2,m}[u](z)=\int_\Omega u(w)\chi_0(w)
\Lfun^{\langle 2,k\rangle}(z,w)\,\e^{2mQ(z,w)-2mQ(w)}\diff A(w)
\\
=\int_\Omega u(w)\chi_0(w)
\Rfun^{(2)}(z,w)\,\e^{2mQ(z,w)-2mQ(w)}\diff A(w)
\\
+\int_\Omega u(w)\chi_0(w)
[\Mop_{\bar\partial_w\theta}\bslashnabla\Mop_{\bar\partial_w\theta}\bslashnabla
X^{\langle k+2\rangle}]^{\langle k\rangle}(z,w)\,\e^{2mQ(z,w)-2mQ(w)}\diff A(w),
\label{eq-9.2.4}
\end{multline}
for $z\in\D(z_0,\tfrac13r)$, where $X^{\langle k+2\rangle}$ is given by 
\eqref{eq-recoverX}, and 
$\Lfun^{\langle 2,k\rangle}(z,w)$ is determined uniquely by the criteria 
\eqref{negliampeq3.b}
together with the requirement that $\Lfun^{\langle 2,k\rangle}(z,w)$ should be 
bianalytic in each of the variables $(z,\bar w)$, as worked out in Subsections 
\ref{subsec-I}, \ref{subsec-II}, and \ref{subsec-III}.  
We rewrite \eqref{eq-9.2.4} in the following form:
\begin{multline} 
\Iop^{\langle k\rangle}_{2,m}[u](z)
=\int_\Omega u(w)\chi_0(w)
\Rfun^{(2)}(z,w)\,\e^{2mQ(z,w)-2mQ(w)}\diff A(w)
\\
+\int_\Omega u(w)\chi_0(w)
[\Mop_{\bar\partial_w\theta}\bslashnabla\Mop_{\bar\partial_w\theta}\bslashnabla
X^{\langle k+2\rangle}]\,\e^{2mQ(z,w)-2mQ(w)}\diff A(w)
\\
-m^{-k}\int_\Omega u(w)\chi_0(w)(Y_k+m^{-1}Z_k)\,\e^{2mQ(z,w)-2mQ(w)}\diff A(w),
\label{eq-9.2.5}
\end{multline}
where $Y_k,Z_k$ are the expressions
\[
Y_k:=
\Mop_{\bar\partial_w\theta}\dd_\theta\Mop_{\bar\partial_w\theta}\dd_\theta X_k
+2\Mop_{\bar\partial_w\theta}\big(
\dd_\theta\Mop_{\bar\partial_w\theta}\Mop_{z-w}
+\Mop_{\bar\partial_w\theta}^2\Mop_{z-w}\dd_\theta\big)X_{k+1}.
\]
and
\[
Z_k:=\Mop_{\bar\partial_w\theta}\dd_\theta\Mop_{\bar\partial_w\theta}\dd_\theta
X_{k+1}.
\]
Here, we recall that $X^{\langle k+2\rangle}=\sum_{j=0}^{k+1}m^{-j}X_j$ is a finite
asymptotic expansion (or abschnitt).
If we use \eqref{Qtaylor2} to estimate the last term on the right-hand side
of \eqref{eq-9.2.5}, and combine with Proposition \ref{neglipropbi} and 
\eqref{bianalnegliampli2}, \eqref{eq-RfunM} (just write $ X^{\langle k+2\rangle}$
in place of $A$), 
\begin{multline} 
\Iop^{\langle k\rangle}_{2,m}[u](z)=u(z)
+\Ordo\big(r^{-1}\e^{mQ(z)-\delta_0m}\|u\|_m
\big\{r\|\bar\partial_w X^{\langle k+2\rangle}\|_{L^\infty(\D(z_0,\frac34r)^2)}
\\
+[1+\|X^{\langle k+2\rangle}\|_{L^\infty(\D(z_0,\frac34r)^2)}]
[1+mr^2\|\bar\partial_w\theta\|_{L^\infty(\D(z_0,\frac34r)^2)}]\big\}+m^{-k-\frac12}
\e^{mQ(z)}\|u\|_m\|Y_k+m^{-1}Z_k\|_{L^\infty(\D(z_0,\frac34r)^2)}\big),
\label{eq-9.2.6}
\end{multline}
for $z\in\D(z_0,\frac13r)$, where the implied constant is absolute. 
Given our standing assumptions, all the involved norms are finite.  
If we accept an implied constant which may depend on the triple $(Q,z_0,r)$, 
then we can compress \eqref{eq-9.2.6} to  
\begin{equation}
\Iop^{\langle k\rangle}_{2,m}[u](z)=u(z)+\Ordo(m^{-k-\frac12}
\e^{mQ(z)}\|u\|_m),\qquad z\in\D(z_0,\tfrac13r),
\label{eq-9.2.7}
\end{equation}
for $m\ge1$. 
By a duality argument, \eqref{eq-9.2.7} implies that
\begin{equation}
\big\Vert \Pop_{2,m}[K^{\langle k\rangle}_{2,m}(\cdot,z)\chi_0]-K_{2,m}(\cdot,z)
\big\Vert_m
=\Ordo(m^{-k-\frac12}\e^{mQ(z)}),\qquad z\in\D(z_0,\tfrac13r).
\label{eq-9.2.7.1}
\end{equation}

If we apply this estimate \eqref{eq-9.2.7} to the function 
$u(z)=K_{2,m}(z,w)$ (not the dummy variable $w$ used in the above integrals!),
we find that (for $m\ge1$)
\begin{multline}
\Iop^{\langle k\rangle}_{2,m}[K_{2,m}(\cdot,w)](z)=K_{2,m}(z,w)
+\Ordo(m^{-k-\frac12}\e^{mQ(z)}K_{2,m}(w,w)^{1/2})
\\
=K_{2,m}(z,w)+\Ordo(m^{-k}\e^{mQ(z)+mQ(w)}),\qquad 
z\in\D(z_0,\tfrac13r),
\label{eq-9.2.7.1}
\end{multline}
where again the implied constant depends on the triple $(Q,z_0,r)$. Note
that in the first step, we used that $\|K_{2,m}(\cdot,w)\|_m^2=
K_{2,m}(w,w)$, and in the last step the estimate of Corollary \ref{cor-8.2}
(with $z_0$ replaced by $z$ and with e.g. $\delta=\frac13r$). 
If we switch the roles of $z$ and $w$ in \eqref{eq-9.2.7.1}, so that 
e.g. $w\in\D(z_0,\frac13r)$ instead, we obtain an 
effective estimate of the first term on the right-hand side of \eqref{glob1}:
\begin{equation}
\Iop^{\langle k\rangle}_{2,m}[K_{2,m}(\cdot,z)](w)
=K_{2,m}(w,z)+\Ordo(m^{-k}\e^{mQ(z)+mQ(w)}),\qquad 
w\in\D(z_0,\tfrac13r).
\label{eq-9.2.8}
\end{equation}
\subsection{Analysis of the second term in the pointwise estimate I}

We proceed to estimate the second term on the right-hand side of 
\eqref{glob1}. 
We write
\begin{equation}
\label{eq-defv0}
v_0(z):=K^{\langle k\rangle}_{2,m}(w,z) 
\chi_0(z)-\Pop_{2,m}[K^{\langle k\rangle}_{2,m}(\cdot,w)\chi_0](z),
\end{equation}
and realize that $v=v_0$ is the norm minimal solution in $L^2(\Omega,\e^{-2mQ})$
of the partial differential equation 
\[ 
\bar\partial^2_z v(z) = \bar\partial^2_z [\chi_0(z)K^{\langle k\rangle}_{2,m}(z,w)].
\]
For the $\bar\partial$-equation, there are the classical H\"ormander 
$L^2$-estimates \cite{horm}, \cite{horm2}, which are based on integration 
by parts and a clever duality argument. Here, luckily we can just iterate 
the H\"ormander $L^2$-estimates to control the solution to the 
$\bar\partial^2$-equation. We recall that the H\"ormander $L^2$-estimate 
asserts that there exists a (norm-minimal) solution $u_0$ to 
$\bar\partial u_0=f$ with
\[
\int_{\Omega} |u_0|^2 \e^{-2\phi}\diff A \le\frac12\int_{\Omega} |f|^2 
\frac{\e^{-2\phi}}{\hDelta\phi}\diff A, 
\]
provided $f\in L^2_{\mathrm{loc}}(\Omega)$ and the right-hand side integral 
converges.

\begin{prop}
Let $\Omega$ be a domain in $\C$ and suppose $\phi:\Omega\to\R$ is a 
$C^4$-smooth function with both $\hDelta\phi>0$ and 
$\hDelta\phi+\frac12\hDelta\log\hDelta\phi>0$ on $\Omega$.  
Assume that $f\in L^2_{\mathrm{loc}}(\Omega)$. Then there exists a (norm-minimal)
solution $v_0$ to the problem $\bar\partial^2v_0=f$ with  
\[
\int_{\Omega} |v_0|^2 \e^{-2\phi} \le\frac{1}{4}\int_{\Omega} |f|^2 
\frac{\e^{-2\phi}}{[\hDelta\phi][\hDelta\phi+\frac12\hDelta\log\hDelta\phi]}
\diff A, 
\]
provided that the right-hand side is finite. 
\label{prop-hormiter}
\end{prop}

\begin{proof}
We apply H\"ormander's $L^2$-estimates for the $\bar\partial$-operator with 
respect to the two weights $\phi$ and $\phi+\frac12\log\hDelta\phi$. The 
details are quite straightforward and left to the reader. 
\end{proof}

\subsection{Digression on H\"ormander-type estimates and
an obstacle problem}

We should like to apply Proposition \ref{prop-hormiter} with $\phi:=mQ$. 
Then we need the assumptions of the proposition to be valid:
\begin{equation}
\label{eq-9.4.1}
\hDelta Q>0\quad\text{and}\quad\frac{1}{2\hDelta Q}\hDelta\log\hDelta Q>-m
\quad\text{on}\,\,\,\Omega.
\end{equation}
The latter criterion apparently gets easier to fulfill as $m$ grows.
In the analytic case $q=1$ only the criterion $\Delta Q>0$ is needed,
and in fact, it can be relaxed rather substantially (see \cite{ahm1} for
the polynomial setting). Basically, what happens is that we may sometimes
replace the potential $Q$ by its subharmonic minorant $\check Q$ given by
\[
\check Q(z):=\sup\big\{u(z):\,u\in\mathrm{Subh}(\Omega),\,\,\text{and}\,\,
u\le Q \,\,\text{on}\,\,\Omega\big\},
\]
where $\mathrm{Subh}(\Omega)$ denotes the cone of subharmonic functions;
see Theorem 4.1 \cite{ahm1} for details (in the polynomial setting). In 
principle, the Bergman kernel should be well approximated by the local 
Bergman kernel if $\check Q=Q$ in a neighborhood of the given point $z_0$. 
From the conceptual point of view, what is important is that for an analytic 
function $f$, the expression $\frac{1}{m}\log|f|$ is subharmonic (and in fact 
in the limit as $m$ grows to infinity we can approximate all subharmonic 
functions with expressions of this type). In our present bianalytic setting, 
$\frac{1}{m}\log|f|$ need not be subharmonic, but we still suspect 
that it approximately subharmonic for large $m$. So we expect that 
\eqref{eq-9.4.1} can be relaxed to $\check Q=Q$ near the point $z_0$ plus 
(if needed) some suitable replacement of the second (curvature type) 
criterion of \eqref{eq-9.4.1}. 

\subsection{Analysis of the second term in the pointwise estimate II} 

As we apply Proposition \ref{prop-hormiter} to $\phi=mQ$, the result is the 
following.

\begin{prop}
Let $\Omega$ be a domain in $\C$ and suppose $Q:\Omega\to\R$ is  
$C^4$-smooth with both $\hDelta Q>0$ on $\Omega$ and
\[
\kappa:=\sup_{\Omega}\frac{1}{2\hDelta Q}\hDelta
\log\frac{1}{\hDelta Q}<+\infty. 
\]
Assume that $f\in L^2_{\mathrm{loc}}(\Omega)$. Then there exists a (norm-minimal)
solution $v_0$ to the problem $\bar\partial^2v_0=f$ with  
\[
\int_{\Omega} |v_0|^2 \e^{-2mQ}\diff A \le\frac{1}{4m(m-\kappa)}\int_{\Omega} |f|^2 
\frac{\e^{-2mQ}}{[\hDelta Q]^2}\diff A, 
\]
provided that $m>\kappa$ and that the right-hand side is finite. 
\label{prop-hormiter2}
\end{prop} 

We now apply the above proposition with
\begin{multline*}
f(z):=\bar\partial_z^2[\chi_0(z)K^{\langle k\rangle}_{2,m}(z,w)]=
K^{\langle k\rangle}_{2,m}(z,w)\,\bar\partial^2\chi_0(z)+
2[\bar\partial\chi_0(z)][\bar\partial_zK^{\langle k\rangle}_{2,m}(z,w)]
\\
=\big\{\Lfun^{\langle2,k\rangle}(z,w)\,\bar\partial^2\chi_0(z)+
2[\bar\partial\chi_0(z)][\bar\partial_z\Lfun^{\langle2,k\rangle}(z,w)]\big\}
\e^{2mQ(z,w)},
\end{multline*}
and calculate that
\begin{multline*}
|f(z)|\e^{-mQ(z)}\le\big\{|\bar\partial^2\chi_0(z)\Lfun^{\langle2,k\rangle}(z,w)|+
2|\bar\partial\chi_0(z)\bar\partial_w\Lfun^{\langle2,k\rangle}(z,w)|\big\}
\e^{m\re[2Q(z,w)-Q(z)]}
\\
\le\big\{|\bar\partial^2\chi_0(z)\Lfun^{\langle2,k\rangle}(z,w)|+
2|\bar\partial\chi_0(z)\bar\partial_z\Lfun^{\langle2,k\rangle}(z,w)|\big\}
\e^{mQ(w)-\delta_0 m},
\end{multline*}
where as before $\delta_0=\frac{1}{18}r^2\epsilon_0$ and it is assumed that 
$w\in\D(z_0,\frac13r_0)$. 
Under the assumptions on $Q$ stated in Proposition \ref{prop-hormiter2},
we obtain from that proposition (recall the definition of \eqref{eq-defv0}) and
the above calculation that 
\begin{multline}
\big\|K^{\langle k\rangle}_{2,m}(\cdot,w) 
\chi_0-\Pop_{2,m}[K^{\langle k\rangle}_{2,m}(\cdot,w)\chi_0]\big\|_m
\\
\le \frac{[\epsilon_0]^{-1}}{2m^{1/2}(m-\kappa)^{1/2}}\e^{mQ(w)-\delta_0 m}
\big\{\|\bar\partial^2\chi_0\|_{L^2(\Omega)}
\,\|\Lfun^{\langle2,k\rangle}\|_{L^\infty(\D(z_0,\frac34r)^2)}+
2\|\bar\partial\chi_0\|_{L^2(\Omega)}
\|\bar\partial_z\Lfun^{\langle2,k\rangle}\|_{L^\infty(\D(z_0,\frac34r)^2)}\big\}
\\
=\Ordo\bigg((m\epsilon_0)^{-1}\e^{mQ(w)-\delta_0 m}\big\{
r^{-1}\|\Lfun^{\langle2,k\rangle}\|_{L^\infty(\D(z_0,\frac34r)^2)}+
\|\bar\partial_z\Lfun^{\langle2,k\rangle}\|_{L^\infty(\D(z_0,\frac34r)^2)}\big\}\bigg),
\label{eq-9.6.1}
\end{multline}
where the implied constant is absolute and $w\in\D(z_0,\frac13r_0)$, and we
assume that $m\ge2\kappa$.
If we accept a constant which depends on the triple $(Q,z_0,r)$, then 
we can simplify \eqref{eq-9.6.1} significantly: 
\begin{equation}
\big\|K^{\langle k\rangle}_{2,m}(\cdot,w) 
\chi_0-\Pop_{2,m}[K^{\langle k\rangle}_{2,m}(\cdot,w)\chi_0]\big\|_m
=\Ordo(m\e^{mQ(w)-\delta_0 m}),
\label{eq-9.6.2}
\end{equation}
because the two norms of $\Lfun^{\langle2,k\rangle}$ and 
$\bar\partial_z\Lfun^{\langle2,k\rangle}$ grow like $\Ordo(m^2)$. 
Finally, we combine this norm estimate \eqref{eq-9.6.2} with control on 
the point evaluation at $z$ (see Proposition \ref{pointwise2}):
\begin{equation}
\big|K^{\langle k\rangle}_{2,m}(z,w) 
\chi_0(z)-\Pop_{2,m}[K^{\langle k\rangle}_{2,m}(\cdot,w)\chi_0](z)\big|
=\Ordo(m^{3/2}\e^{mQ(z)+mQ(w)-\delta_0 m}),\qquad z,w\in\D(z_0,\tfrac13r).
\label{eq-9.6.3}
\end{equation}

\subsection{The pointwise distance to the weighted Bergman kernel}
Here, we apply the estimates \eqref{eq-9.2.8} and \eqref{eq-9.6.3} with
the (basic) estimate \eqref{glob1}. As before, the integer $k$ is fixed. 

\begin{thm}
\label{thm-9.3}
Let $\Omega$ be a domain in $\C$ and suppose $Q:\Omega\to\R$ is  
$C^4$-smooth with both $\hDelta Q>0$ on $\Omega$ and
\[
\kappa:=\sup_{\Omega}\frac{1}{2\hDelta Q}\hDelta
\log\frac{1}{\hDelta Q}<+\infty. 
\]
Assume moreover that $Q$ has the properties (A:i)--(A:iv) (see Subsection
\ref{subsec-assQ}) on the disk $\D(z_0,r)\subset\Omega$.
We fix an integer $k\ge0$ and assume that the approximate local Bergman
kernel $K_{2,m}^{\langle k\rangle}(z,w)=\Lfun^{\langle2,k\rangle}\e^{2mQ(z,w)}$ 
extends to a separately bianalytic function of the variables $(z,\bar w)$ in the small bidisk 
$\D(z_0,r)\times\D(z_0,r)$. 
Then, for $m\ge\max\{2\kappa,1\}$, we have that
\[
K_{2,m}(z,w)=K^{\langle k\rangle}_{2,m}(z,w)+\Ordo(m^{-k}\e^{mQ(z)+mQ(w)}),\qquad
z,w\in\D(z_0,\tfrac13r), 
\] 
where the implied constant depends only on the triple $(Q,z_0,r)$. 
\end{thm}

This means that the approximate kernel $K_{2,m}^{\langle k\rangle}$ is close to the 
true kernel $K_{2,m}$ locally near $z_0$.

\begin{proof}[Proof of Theorem \ref{thm-9.3}]
Since $\chi_0(z)=1$ for $z\in\D(z_0,\frac23r)$, the assertion follows from 
\eqref{eq-9.2.8} and \eqref{eq-9.6.3} once we observe that exponential decay
is faster than power decay. 
\end{proof}

\begin{proof}[Proof of Theorem \ref{blowupthm}]
We express $\Lfun^{\langle 2,0\rangle}=m^2\Lfun^2_0+m\Lfun_1^2$ using the 
formulae obtained for $\Lfun^2_j$ for $j=0,1$ (see Remark \ref{rem-AA}). 
We also use that $2\re Q(z,w)-Q(w)-Q(z)=-|w-z|^2\hDelta Q(z)+\Ordo(|z-w|^3)$.
The assertion follows after a number of tedious computations based on
Taylor's formula.
\end{proof}

\begin{rem}
We have already obtained $K^{\langle k\rangle}_{2,m}$ explicitly for $k\le1$
and the assumptions made on $K^{\langle k\rangle}_{2,m}$ in the theorem are clearly 
fulfilled then. We believe that they are fulfilled for all other values of $k$ as well.
 
\end{rem}

\section{Bianalytic extensions of Bergman's metrics and 
asymptotic analysis}
\label{sec-Bergmanmetric}

\subsection{Purpose of the section}

In this section, we apply the asymptotic formulae obtained in Sections
\ref{sec-kernelexp} (see especially Remark \ref{rem-AA}) and 
\ref{sec-locglob} to form asymptotic expressions for Bergman's first and
second metric (see Section \ref{musing-BM}) in the context of 
$\omega=\e^{-2mQ}$. We will generally assume that the assumptions of Theorem
\ref{thm-9.3} are fulfilled, so that the local analysis of Section 
\ref{sec-kernelexp} gives the global weighted bianalytic Bergman kernel with 
high precision.

\subsection{Bergman's first metric}

Our starting point is the abschnitt $\Lfun^{\langle 2,k\rangle}$ with $k=0$:
\begin{multline}
\label{eq-10.2.1}
\Lfun^{\langle 2,0\rangle}=m^2\Lfun_0^2+m\Lfun^2_1=-\frac{4m^2}{\pi}
|z-w|^2[\vbeta(z,w)]^2+m\bigg\{\frac{4}{\pi}\vbeta(z,w)
-\frac{2}{\pi}(\bar z-\bar w)\dbar_w\vbeta(z,w)
\\
+\frac{2}{\pi}(z-w)\partial_z\vbeta(z,w)
-\frac{3}{\pi}|z-w|^2\partial_z\dbar_w\vbeta(z,w)+\frac{2}{\pi}|z-w|^2
\frac{[\partial_z\vbeta(z,w)][\dbar_w\vbeta(z,w)]}
{\vbeta(z,w)}\bigg\},
\end{multline}
where we used the formulae \eqref{eq-4.24} and \eqref{eq-7.9.21}, and recall 
the notational convention $\vbeta(z,w)=\partial_z\bar\partial_w Q(z,w)$. 
The lift of the abschnitt is easily obtained by inspection of 
\eqref{eq-10.2.1}:
\begin{multline}
\label{eq-10.2.2}
\mathbf{E}_{\otimes2}[\Lfun^{\langle 2,0\rangle}](z,z';w,w')
=-\frac{4m^2}{\pi}
(z-w')(\bar z'-\bar w)[\vbeta(z,w)]^2
\\
+m\bigg\{\frac{4}{\pi}\vbeta(z,w)
-\frac{2}{\pi}(\bar z'-\bar w)\dbar_w\vbeta(z,w)
+\frac{2}{\pi}(z-w')\partial_z\vbeta(z,w)
\\
-\frac{3}{\pi}(z-w')(\bar z'-\bar w)\partial_z\dbar_w\vbeta(z,w)
+\frac{2}{\pi}(z-w')(\bar z'-\bar w)
\frac{[\partial_z\vbeta(z,w)][\dbar_w\vbeta(z,w)]}
{\vbeta(z,w)}\bigg\},
\end{multline} 
The double diagonal restriction to $z=w$ and $z'=w'$ is then (since 
$\lambda(z,z)=\hDelta Q(z)$)
\begin{multline}
\label{eq-10.2.3}
\mathbf{E}_{\otimes2}[\Lfun^{\langle 2,0\rangle}](z,z';z,z')
=\frac{4m^2}{\pi}
|z-z'|^2[\hDelta Q(z)]^2
+m\bigg\{\frac{4}{\pi}\hDelta Q(z)
+\frac{4}{\pi}\re[(z-z')\partial\hDelta Q(z)]
\\
+\frac{3}{\pi}|z-z'|^2\hDelta^2Q(z)
-\frac{2}{\pi}|z-z'|^2
\frac{|\dbar\hDelta Q(z)|^2}
{\hDelta Q(z)}\bigg\}.
\end{multline}  
As $K_{2,m}^{\langle k\rangle}(z,w)=\Lfun^{\langle 2,k\rangle}(z,w)\e^{2mQ(z,w)}$, 
it follows that
\begin{equation}
\label{eq-10.2.4}
\mathbf{E}_{\otimes2}[K_{2,m}^{\langle 0\rangle}](z,z';z,z')
=\e^{2mQ(z)}\mathbf{E}_{\otimes2}[\Lfun^{\langle 2,0\rangle}](z,z';z,z').
\end{equation}
We conclude that 
\begin{multline}
\label{eq-10.2.5}
\e^{-2mQ(z)}\mathbf{E}_{\otimes2}[K_{2,m}^{\langle 0\rangle}](z,z+\epsilon;z,z+\epsilon)
=\mathbf{E}_{\otimes2}[\Lfun^{\langle 2,0\rangle}](z,z+\epsilon;z,z+\epsilon)
\\
=\frac{4m^2}{\pi}
|\epsilon|^2[\hDelta Q(z)]^2
+m\bigg\{\frac{4}{\pi}\hDelta Q(z)
-\frac{4}{\pi}\re[\epsilon\partial\hDelta Q(z)]
+\frac{3}{\pi}|\epsilon|^2\hDelta^2Q(z)
-\frac{2}{\pi}|\epsilon|^2
\frac{|\dbar\hDelta Q(z)|^2}
{\hDelta Q(z)}\bigg\}.
\end{multline}
Next, we rescale the parameter $\epsilon$ to fit the typical local size 
$m^{-1/2}$:
\begin{equation}
\label{eq-rescale}
\epsilon=\frac{\epsilon'}{[2m\Delta Q(z)]^{1/2}}.
\end{equation}
Then, in view of \eqref{eq-10.2.5},
\begin{equation}
\label{eq-10.2.6}
\e^{-2mQ(z)}\mathbf{E}_{\otimes2}[K_{2,m}^{\langle 0\rangle}](z,z+\epsilon;z,z+\epsilon)
=\frac{2m}{\pi}
(2+|\epsilon'|^2)\hDelta Q(z)
+\Ordo(m^{1/2}).
\end{equation}
It is possible to extend the assertion of Theorem \ref{thm-9.3} to the lift
$\mathbf{E}_{\otimes2}[K_{2,m}]$ in place of the kernel itself; this requires
a nontrivial argument. Basically, what is needed is to know that 
the expressions $\bar\partial_z K_{2,m}(z,w)$, $\partial_w K_{2,m}(z,w)$, and
$\bar\partial_z\partial_w K_{2,m}(z,w)$ are all suitably  approximated 
by the corresponding expression where the kernel $K_{2,m}(z,w)$ is replaced 
by the approximate kernel $K_{2,m}^{\langle 0\rangle}(z,w)$. We should say some words
on how this may be achieved. First, we observe that \eqref{eq-9.2.7.1} and \eqref{eq-9.6.2} 
say that the three functions 
\[
K^{\langle k\rangle}_{2,m}(\cdot,w)\chi_0,\quad 
\Pop_{2,m}[K^{\langle k\rangle}_{2,m}(\cdot,w)\chi_0],\quad K_{2,m}(\cdot,w),
\]
are all close to one another in $L^2(\Omega,\e^{-2mQ})$ for $w\in\D(z_0,\frac13r)$. 
Furthermore, the associated error terms may be written out explicitly, and it is possible to check
that the $\partial_w$-derivatives of the above three kernels remain close to one another.  
In particular, 
\[
\partial_w K^{\langle k\rangle}_{2,m}(\cdot,w)\chi_0\quad\text{and}\quad 
\partial_wK_{2,m}(\cdot,w)
\]
are close to one another for $w\in\D(z_0,\frac13r)$:
\begin{equation*}
\big\Vert\partial_w[K^{\langle k\rangle}_{2,m}(\cdot,w)\chi_0-K_{2,m}(\cdot,w)]
\big\Vert_{m}=\Ordo(m^{-k-\frac12}\e^{mQ(w)}), \qquad 
w\in\D(z_0,\tfrac13r).
\end{equation*}
Next, we observe that 
\[
\overline{\Iop^{\langle k\rangle}_{2,m}[K_{2,m}(\cdot,z)](w)}
=\Pop_{2,m}\big[K^{\langle k\rangle}_{2,m}(\cdot,w)\chi_0\big](z)
=\big\langle K^{\langle k\rangle}_{2,m}(\cdot,w)\chi_0 ,K_{2,m}(\cdot,z)\big\rangle_m;
\]
the analogue of \eqref{glob1} then reads
\begin{multline} 
\big| \dbar_z\partial_w K_{2,m}(z,w)- \dbar_z\partial_w K^{\langle k\rangle}_{2,m}(z,w)\big| 
\le
\big|\big\langle
\partial_z K_{2,m}(\cdot,z),\partial_w[K^{\langle k\rangle}_{2,m}(\cdot,w)\chi_0
-K_{2,m}(\cdot,w)]\big\rangle_m\big| 
\\
+\big|\dbar_z 
\Pop_{2,m}^\perp[\partial_w K^{\langle k\rangle}_{2,m}(\cdot,w)\chi_0](z)\big|,
\label{glob2}
\end{multline}
for $z,w\in\D(z_0,\frac13r)$. Here, $\Pop_{2,m}^\perp:=\mathbf{I}-\Pop_{2,m}$ is the 
projection onto the orthogonal complement. It is clear how to handle the first term on the
right-hand side of \eqref{glob2} by using the above-mentioned estimate combined with 
Proposition \ref{pointwisederiv} (which lets us to control 
$\Vert\partial_z K_{2,m}(\cdot,z)\Vert_m$). 
The last term is easily controlled by applying \eqref{eq-9.6.2} together with Proposition 
\ref{pointwisederiv}, as the function
\[
\Pop_{2,m}^\perp[\partial_w K^{\langle k\rangle}_{2,m}(\cdot,w)\chi_0]
\]
is bianalytic in $\D(z_0,\frac23r)$ (for fixed $w$). In conclusion, \eqref{glob2} leads to
 \[
\dbar_z\partial_w K_{2,m}(z,w)=
\dbar_z\partial_w K^{\langle k\rangle}_{2,m}(z,w)+\Ordo(m^{-k}\e^{mQ(z)+mQ(w)}),
\qquad z,w\in\D(z_0,\tfrac13r),
\]
as desired. The remaining approximate identities -- for $\bar\partial_z K_{2,m}(z,w)$ and
$\partial_w K_{2,m}(z,w)$ -- are more straightforward, and therefore left to the reader.  


It then follows 
that in the context of Theorem \ref{thm-9.3}, we have from \eqref{eq-10.2.5}
that
\begin{equation}
\label{eq-10.2.7}
\e^{-2mQ(z)}\mathbf{E}_{\otimes2}[K_{2,m}](z,z+\epsilon;z,z+\epsilon)
=\frac{2m}{\pi}(2+|\epsilon'|^2)\hDelta Q(z)+\Ordo(m^{1/2}),
\end{equation}
for $z\in\D(z_0,\frac13r)$, and Bergman's first metric, suitably rescaled, 
becomes asymptotically (with $\omega=\e^{-2mQ}$)
\begin{equation}
\frac{\diff s_{2,\omega,\epsilon}^{\text{\ding{174}}}(z)^2}{2m\hDelta Q(z)}
=
\e^{-2mQ(z)}\mathbf{E}_{\otimes2}[K_{2,m}](z,z+\epsilon;z,z+\epsilon)
\frac{|\diff z|^2}{2m\hDelta Q(z)}
=\big\{\pi^{-1}(2+|\epsilon'|^2)+\Ordo(m^{-1/2})\big\}|\diff z|^2;
\end{equation}
cf. \eqref{eq-bm101}. In matrix form (see \eqref{eq-varmat1}), this 
corresponds to a diagonal matrix with entries $2/\pi$ and $1/\pi$. Since
$Q$ is a rather general potential, this is a sort of universality.

\subsection{Bergman's second metric}

Given that Theorem \ref{thm-9.3} extends to apply to the lifted kernel, 
we see from \eqref{eq-10.2.3} and \eqref{eq-10.2.4} that
\begin{multline}
\label{eq-10.3.1}
L(z,z')=\log\mathbf{E}_{\otimes2}[K_{2,m}](z,z';z,z')
=2mQ(z)+\log\frac{4m\hDelta Q(z)}{\pi}
\\
+\log\bigg\{1+m|z-z'|^2\hDelta Q(z)
+\re[(z-z')\partial\log\hDelta Q(z)]
+|z-z'|^2\bigg[\frac{3\hDelta^2Q(z)}{4\hDelta Q(z)}-
\frac{|\dbar\hDelta Q(z)|^2}{2[\hDelta Q(z)]^2}\bigg]\bigg\}
\\
+\Ordo(m^{-1/2}).
\end{multline}
Applying the Laplacian $\hDelta_z$ to \eqref{eq-10.3.1}, we should expect 
to obtain
\begin{equation}
\label{eq-10.3.2}
\Delta_z L(z,z')=2m\hDelta Q(z)+
\frac{m\hDelta Q(z)}{[1+m|z-z'|^2\hDelta Q(z)]^2}+\Ordo(m^{1/2}),
\end{equation}
provided that $|z-z'|=\Ordo(m^{-1/2})$ is assumed. Applying instead 
$\hDelta_{z'}$ results in 
\begin{equation}
\label{eq-10.3.3}
\Delta_{z'} L(z,z')=\frac{m\hDelta Q(z)}{[1+m|z-z'|^2\hDelta Q(z)]^2}
+\Ordo(m^{1/2}),
\end{equation}
provided that $|z-z'|=\Ordo(m^{-1/2})$. If instead of the Laplacian
we apply $\partial_z\partial_{z'}$, we should get that
\begin{equation}
\label{eq-10.3.4}
\partial_z\partial_{z'}L(z,z')=\frac{m^2(\bar z-\bar z')^2[\Delta Q(z)]^2}
{[1+m|z-z'|^2\hDelta Q(z)]^2}+\Ordo(m^{1/2}),
\end{equation}
again provided $|z-z'|=\Ordo(m^{-1/2})$. We conclude that
Bergman's second metric \eqref{eq-bpam2}, suitably rescaled, 
should be asymptotically given by (with $\omega=\e^{-2mQ}$)
\begin{equation}
\frac{\diff s_{2,\omega,\epsilon}^{\text{\ding{175}}}(z)^2}{2m\hDelta Q(z)}
=\bigg[1+\frac{4}{(2+|\epsilon'|^2)^2}+\Ordo(m^{-1/2})\bigg]|\diff z|^2+
2\re\bigg[\bigg\{\frac{[\epsilon']^2}{(2+|\epsilon'|^2)^2}
+\Ordo(m^{-1/2})\bigg\}(\diff z)^2\bigg],
\end{equation}
where $z'=z+\epsilon$ and $\epsilon'$ is the rescaled parameter in accordance 
with \eqref{eq-rescale}.
provided $|z-z'|=\Ordo(m^{-1/2})$. As $Q$ is a rather general potential, 
this appearance of the metric 
\[
\bigg[1+\frac{4}{(2+|\epsilon'|^2)^2}\bigg]|\diff z|^2+
2\re\bigg[\frac{[\epsilon']^2}{(2+|\epsilon'|^2)^2}
(\diff z)^2\bigg]
\]
in the limit may be interpreted as a kind of universality.

\end{document}